\documentclass{article}
\usepackage[utf8]{inputenc}

\usepackage[T1]{fontenc}

\usepackage{mathrsfs}
\usepackage{amsthm,amsmath,amsfonts,amssymb}
\usepackage{graphicx}
\usepackage{enumerate,enumitem}
\usepackage{authblk}
\usepackage{cite}

\let\OLDthebibliography\thebibliography
\renewcommand\thebibliography[1]{
  \OLDthebibliography{#1}
  \setlength{\parskip}{0pt}
  \setlength{\itemsep}{0pt plus 0.3ex}
}

\allowdisplaybreaks
 \usepackage[left=2.5cm,right=2.5cm,top=2.5cm,bottom=2.5cm]{geometry}

\makeatletter
\numberwithin{equation}{section}
\numberwithin{figure}{section}
\theoremstyle{plain}
\newtheorem{thm}{\protect\theoremname}
  \theoremstyle{plain}
  \newtheorem{lemma}[thm]{\protect\lemmaname}

\numberwithin{thm}{section}

\newtheorem{proposition}[thm]{Proposition}

\newtheorem{conjecture}[thm]{Conjecture}

\newtheorem{theorem}[thm]{Theorem}

\theoremstyle{remark}
\newtheorem*{rem}{Remark}

\makeatother

\usepackage{babel}
\providecommand{\propname}{Proposition}

\providecommand{\lemmaname}{Lemma}
\providecommand{\theoremname}{Theorem}

\newcommand{\imag}{\operatorname{Im} \,}
\newcommand{\real}{\operatorname{Re} \,}
\renewcommand{\Im}{\imag}
\renewcommand{\Re}{\real}

\newcommand{\Tr}{\text{{\rm tr\,}}}
\newcommand{\re}{\text{{\rm Re}\,}}
\newcommand{\im}{\text{{\rm Im}\,}}

\newcommand{\R}{\mathbb{R}}
\newcommand{\C}{\mathbb{C}}
\newcommand{\T}{\mathbb{T}}
\newcommand{\D}{\mathbb{D}}
\newcommand{\Z}{\mathbb{Z}}
\newcommand{\I}{{\rm i}}

\let \le \leqslant

\let \ge \geqslant

\let \epsilon \varepsilon

\usepackage{hyperref}
\hypersetup{
    colorlinks=true, 
    linkcolor=black, 
    urlcolor=black, 
    linktoc=all 
}

\usepackage[dvipsnames]{xcolor}

\title{Coulomb gas and the Grunsky operator on a\\ Jordan domain with corners}
\author{Kurt Johansson\thanks{KTH Royal Institute of Technology, email:  kurtj@kth.se}\,  and Fredrik Viklund\thanks{KTH Royal Institute of Technology, email:  frejo@kth.se}}
\date{}
\begin{document}

\maketitle

\begin{abstract}
Let $D$ be a Jordan domain of unit capacity. We study the partition function of a planar Coulomb gas in $D$ with a hard wall along $\eta = \partial D$,
\[Z_{n}(D) =\frac 1{n!}\int_{D^n}\prod_{1\le k < \ell \le  n}|z_k-z_\ell|^{2} \prod_{k=1}^n d^2z_k.\]
We are interested in how the geometry of $\eta$ is reflected in the large $n$ behavior of $Z_n(D)$. We prove that $\eta$ is a Weil-Petersson quasicircle if and only if
 \[
 \lim_{n \to \infty} \log \frac{Z_n(D)}{Z_n(\D)} = -\frac{1}{12}I^L(\eta),
 \]
 where $I^L$ is the Loewner energy, $\D$ is the unit disc, and $\log Z_n(\D)  = \log  \pi^n/n!$. 
 
 We next consider piecewise analytic $\eta$ with $m$ corners of interior opening angles $\pi \alpha_p, p=1,\ldots, m$. Our main result is the asymptotic formula
  \[
  \lim_{n\to\infty}\frac 1{\log n} \log  \frac{Z_n(D)}{Z_n(\mathbb{D})} =-\frac 16\sum_{p=1}^m \left(\alpha_p+\frac 1{\alpha_p}-2 \right)
  \]
which is consistent with physics predictions.
   The starting point of our analysis is an exact expression for $\log Z_{n}(D)$ in terms of a Fredholm determinant involving the truncated Grunsky operator for $D$. The proof of the main result is based on careful asymptotic analysis of the Grunsky coefficients.  

   As further applications of our method we also study the Loewner energy and the related Fekete-Pommerenke energy, a quantity appearing in the analysis of Fekete points, for equipotentials approximating the boundary of a domain with corners. We formulate several conjectures and open problems.
\end{abstract}

\section{Introduction and main results}\label{sec:intro}
 Given a probabilistic model defined on a surface, it is quite natural to ask about the interplay between the geometry of that surface and properties of the model.  

Let $D$ be a Jordan domain in the complex plane $\mathbb{C}$. We are interested in the planar Coulomb gas on $D$. At inverse temperature $\beta \ge 0$, its partition function is defined by
\begin{align}\label{CGpartfcn}
Z_{n, \beta}(D) =\frac 1{n!}\int_{D^n}\prod_{1\le k < \ell \le  n}|z_k-z_\ell|^{\beta} \prod_{k=1}^n d^2z_k=\frac 1{n!}\int_{D^n}e^{- \beta \sum_{1 \le k < \ell \le n}\log|z_k-z_\ell|^{-1}}\prod_{k=1}^nd^2z_k,
\end{align}
where $d^2z$ is two-dimensional Lebesgue measure. From the point of view of statistical mechanics this is a model of an ensemble of charged particles in $\C$ interacting pairwise via the electrostatic potential energy. There is a hard wall along $\partial D$, that is, an external potential equal to $0$ in $D$ and $ + \infty$ in the complement $D^*= \hat{\C} \smallsetminus \overline{D}$ forces the particles to occupy the prescribed \emph{droplet} $D$. 

Coulomb gases can be defined in some generality and enjoy a certain degree of universality. They appear in a wide range of contexts beyond that described above and the induced Gibbs probability measures as well as free energy expansions have received substantial attention in recent years. In the planar case, important applications include random matrix theory, conformal field theory, and the analysis of quantum Hall transitions. See, e.g., the recent surveys \cite{Serfaty18, lewin} and the references therein. A common setting, somewhat different from ours, is that a confining potential belonging to some class is fixed and one attempts to describe the set (droplet) on which particles concentrate as their number grows to infinity, along with statistical properties of their distribution. Here we instead prescribe the droplet and we are interested in how its geometry is reflected in the Coulomb gas. See, e.g., \cite{Ber} and the references therein for related points of view.

When $\beta=2$, which is the case we consider in the majority of the paper, the model is determinantal: Andr\'eief's identity, see, e.g., \cite{Andreief2011Reprint}, implies
\[ Z_{n,2}(D)= \det \left( \int_D z^{k-1} \overline{z}^{\ell-1} d^2 z \right)_{1 \le k, \ell \le n},\]
and we write $Z_n(D) = Z_{n,2}(D)$ from now on. Our main focus will be the asymptotic expansion of the \emph{free energy}, $\log Z_n(D)$, as $n\to \infty$.  It turns out to be closely related to the \emph{Grunsky operator} for $D$. Without loss in generality, we assume $ 0 \in D$ and write
 $\D^* =\{z:|z|>1\}$ for the exterior unit disk. The exterior conformal map $g:\D^*\to D^*$ is
uniquely determined by requiring that it has an expansion
\begin{equation}\label{psiexp}
g(z)=r_\infty z+\sum_{k=0}^\infty g_k z^{-k},
\end{equation}
around infinity, where $r_\infty = r_\infty(D)= g'(\infty)>0$ denotes the capacity (transfinite diameter) of $D$. The \emph{Grunsky coefficients} $a_{k\ell}$ are defined by the expansion
\begin{equation}\label{defGrunsky}
\log\frac{g(\zeta)-g(z)}{\zeta-z}=\log r_\infty-\sum_{k,\ell=1}^\infty a_{k\ell}\zeta^{-k}z^{-\ell},
\end{equation}
for $z,\zeta\in\D^*$. Define 
\begin{equation}\label{bkl}
b_{k\ell}=\sqrt{k\ell}a_{k\ell}.
\end{equation}
The operator $B=(b_{k\ell})_{k,\ell\ge 1}$ acting on $\ell^2(\Z_+)$ is the Grunsky operator, a classical tool in geometric function theory, see, e.g., \cite{Po}. A perhaps surprising fact is that the geometry of $D$ is reflected directly in spectral properties of $B$; for instance, $D$ is bounded by a quasicircle (the image of a circle under a quasiconformal map) if and only if $B$ is a strict contraction on $\ell^2(\Z_+)$. Let $P_n$ denote projection onto the first $n$ coordinates in $\ell^2(\Z_+)$. The following proposition which relates the free energy to the Grunsky operator is essentially contained in
\cite[Ch. II, Sec. 3.3]{TT}.
\begin{proposition}\label{Prop:PartDet}
Let $D$ be a bounded Jordan domain containing $0$. We have the identity 
\begin{equation}\label{PartDetId}
\log Z_n(D)=\log \frac{\pi^n}{n!} + n(n+1) \log  r_\infty(D)  + \log \det(I-P_nBB^*P_n)_{\ell^2(\Z_+)}
\end{equation}
for $n\ge 1$.
\end{proposition}
We will give the proof in Section \ref{Subsec:PartDet}. There is an analogous identity for  a Coulomb gas on the unbounded domain $D^*$ with an appropriate potential, see Proposition~\ref{prop:partdet2}. 

Note that when $D=\D$ then $B=0$, $r_\infty(\mathbb{D})=1$, and $Z_n(\D)=\pi^n/n!$, so \eqref{PartDetId} can also be written
\begin{equation}\label{PartDetId2}
\log \frac{Z_n(D)/r_\infty(D)^{n(n+1)}}{Z_n(\D)/r_\infty(\D)^{n(n+1)}}=\log \det(I-P_nBB^*P_n)_{\ell^2(\Z_+)}.
\end{equation}
In view of this, we define the normalized partition function
\[
\bar{Z}_n(D) = \frac{Z_n(D)}{r_\infty(D)^{n(n+1)}}.
\]
The M\"obius invariant \emph{Loewner energy} \cite{Wa} of a Jordan curve $\eta=\partial D$ is defined by
\begin{align}\label{def:Loewner-energy}
I^L(\eta)= \mathcal{D}_\D(\log|f'|) + \mathcal{D}_{\D^*}(\log|g'|)+ 4 \log |f'(0)|/|g'(\infty)|.
\end{align}
Here $f : \D \to D$ is the interior conformal map normalized so that $f(0)=0, \, f'(0)>0$ and we write
\[
\mathcal{D}_{\Omega}(u) = \frac{1}{\pi} \int_\Omega |\nabla u|^2 d^2z
\]
for the Dirichlet integral of $u$ over $\Omega$. The Loewner energy was introduced as a large deviation rate function for Schramm-Loewner evolution curves, but it arises in several seemingly unrelated contexts, notably in relation to the Weil-Petersson metric on universal Teichm\"uller space, see \cite{Wa, TT, WaSurvey}. A curve has finite Loewner energy if and only if it is a \emph{Weil-Petersson quasicircle}, the class of which has a number of equivalent characterizations, see \cite{Wa, TT, Bi}. One such characterization is that $B$ is a Hilbert-Schmidt operator and in this case the Loewner energy can be written as the Fredholm determinant
\begin{equation}\label{ILDet}
I^L(\eta)=-12\log\det(I-BB^*)_{\ell^2(\Z_+)}.
\end{equation}
We see that the following characterization of Weil-Petersson quasicircles is an immediate consequence of Proposition \ref{Prop:PartDet}.
\begin{theorem}\label{Thm:WP}
Let $\eta$ be a Jordan curve and let $D$ be the interior of $\eta$. Then $\eta$ is
a Weil-Petersson quasicircle if and only if
\begin{equation}\label{WPcrit}
\varlimsup_{n\to\infty}-12\log\frac{\bar Z_n(D)}{\bar Z_n(\D)}<\infty,
\end{equation}
in which case the limit exists and equals the Loewner energy $I^L(\eta)$.
\end{theorem}
What happens if $\eta$ is not a Weil-Petersson quasicircle? Curves with finite Loewner energy are always rectifiable but need not be $C^1$. However, it is easy to see from \eqref{def:Loewner-energy} that $I^L(\eta) = \infty$ if $\eta$ has a corner, which also represents the simplest possible geometric singularity if $D$ is viewed as a surface. The main goal of this paper is to investigate the effect that the presence of corner singularities has on the Coulomb gas. 
Clearly the free energy is finite for any Jordan curve whenever $n < \infty$ but by \eqref{PartDetId2} the determinant of the truncated Grunsky operator must diverge as $n \to \infty$ if $\eta$ is not a Weil-Petersson quasicircle. That is, the expansion of the free energy will contain terms of order between $n(n+1)$ and $1$. It is quite natural to expect such terms to carry geometric information about $\eta$, see Section~\ref{sect:discussion}.

To state our main result, for $m \in \Z_+$, let $\mathcal{D}_m$ be the set of (bounded) Jordan domains with piecewise analytic boundary, which consists of $m$ analytic arcs and $m$ \emph{corners} at the points $w_p$ with interior angles $\alpha_p\pi$, $0<\alpha_p<2$, $\alpha_p \neq 1$, $p=1,\ldots,m$. (See Section~\ref{sect:prel} for precise definitions.) The following theorem will be proved in Section \ref{Sec:propthm}.
\begin{theorem}\label{Thm:CornerAs}
If $D\in\mathcal{D}_m$, then
\begin{align}\label{CornerAsFor}
\lim_{n\to\infty}-\frac 1{\log n} \log \frac{\bar Z_n(D)}{\bar Z_n(\D)} =\frac 16\sum_{p=1}^m \left(\alpha_p+\frac 1{\alpha_p}-2 \right).
\end{align}
\end{theorem} 
The function of the corner angles on the right in \eqref{CornerAsFor} has appeared in other contexts and may be considered universal, we comment on this in Section~\ref{sect:discussion}. The assumption that the curve is piecewise analytic is important for the proof since analyticity is used in the computation of asymptotics for the Grunsky coefficients, Theorem \ref{Thm:bklAs} below. However, we expect that the condition of analyticity can be considerably relaxed.
The same methods that will be used to prove Theorem \ref{Thm:CornerAs} can be used to prove the following result on the Loewner energy for a sequence of approximating outer ``equipotentials''. Here and below we write $\mathbb{T} = \partial \mathbb{D}$ for the unit circle.

\begin{theorem}\label{Thm:LoewAs}
Assume that $D \in \mathcal{D}_m$ and let $\eta = \partial D$. Define the Jordan curve $\eta_r$, for $r>1$, to be the image of $\mathbb{T}$ under the conformal map $g_r(z)=g(rz)/r$, so that $\eta_r$ is an
analytic Jordan curve with the same capacity as $\eta$. Then
\begin{equation}\label{LoewAsFor}
\lim_{r\to1+}\frac 1{\log(\frac 1{r-1})}I^L(\eta_r)=2\sum_{p=1}^m \left(\alpha_p+\frac 1{\alpha_p}-2 \right).
\end{equation}
\end{theorem}
This theorem will also be proved in Section \ref{Sec:propthm}. Notice that the expression in \eqref{LoewAsFor} is not invariant with respect to $\alpha_p \mapsto 2-\alpha_p$, i.e., approximating using interior equipotentials will result in a different limit. See also Proposition~\ref{prop:partdet2}.

Let $D$ be a Jordan domain and for $n \ge 2$, consider the minimal logarithmic energy 
\[
 I_n(D) = \inf_{z_1,\ldots,z_n \in D} \sum_{1 \le k < \ell \le n} \log |z_k-z_\ell|^{-1}.
\]
Tuples realizing the minimum are called \emph{Fekete points}, which in general are not unique. Fekete points always lie on $\partial D = \eta$ so $I_n(D) = I_n(\eta)$. We define the partition function for the Coulomb gas on $D$ at inverse temperature $\beta = +\infty$ by the maximal discriminant,
\[
Z_{n, \infty}(D) := e^{-I_n(\eta)} = \sup_{z_1,\ldots,z_n \in D} \prod_{1 \le k < \ell \le n} |z_k -z_\ell|^2.
\]
It is a classical result that 
\[
\lim_{n \to \infty} Z_{n, \infty}(D)^{1/n(n-1)} = r_\infty(D),
\]
i.e., the transfinite diameter equals the capacity.
See Chapter 2 of \cite{Ahl73}; the left-hand side represents a decreasing sequence.
If $D=\D$, then Fekete points are equidistant on $\T$ and it is not hard to see that $Z_{n, \infty}(\D) = n^n$. Define
\begin{align}\label{def:pommerenke-energy}
I^F(\eta) = \mathcal{D}_D(\phi_i) +  \mathcal{D}_{D^*}(\phi_e), 
\end{align}
where $\phi_e = \log|(g^{-1})'|$ and $\phi_i = P_D[\phi_e]$ is the harmonic extension of $\phi_e$ (extended to $\partial D^*$ by non-tangential limits) into $D$.  We call $I^F(\eta)$ the \emph{Fekete-Pommerenke energy} of $\eta$. Like the Leowner energy it can also be written in terms of the Neumann jump operator and is finite if and only if $\eta$ is a Weil-Petersson quasicircle (Lemma~\ref{lemma:pommerenkewp2}).
\begin{theorem}[Pommerenke \cite{pom67, pom69}]\label{thm:pomthm}
If $\eta$ is analytic, then as $n\to\infty$
    \begin{align}\label{pommereneke-expansion}
\log Z_{n, \infty}(D) = n(n-1) \log r_\infty(D)  +  n\log n +  \frac{1}{8}I^F(\eta) + o(1).
\end{align}
\end{theorem}
It is not immediate that Pommerenke's expression as given in \cite{pom67, pom69} is equivalent to \eqref{pommereneke-expansion} but we will explain this below. We believe that Pommerenke's result holds if and only if $\eta$ is Weil-Petersson but we have not proved this, see Section~\ref{sect:conjectures}. We will prove the following result which parallels Theorem~\ref{Thm:LoewAs} but with interior replaced by exterior angles.
\begin{theorem}\label{Thm:PommEn}
    Consider the same setup as in Theorem \ref{Thm:LoewAs}. We have the following limit for the Fekete-Pommerenke energy of the exterior equipotential $\eta_r$,
\begin{equation}\label{PommEn}
    \lim_{r\to 1+}\frac{1}{\log(\frac 1{r-1})} I^F(\eta_r)=2\sum_{p=1}^m \left( \gamma_p+\frac 1{\gamma_p} -2 \right),
\end{equation}
where $\gamma_p=2-\alpha_p$, $1\le p\le m$, are the exterior angles at the corners.

\end{theorem}
The theorem will be proved in Section \ref{Sec:propthm}.

\subsection{Asymptotics of Grunsky coefficients}
As far as we know, the asymptotics of Grunsky coefficients and how this relates to the geometry of the domain have not been much investigated.
The proofs of our main results are based on an asymptotic analysis of the Grunsky coefficients. We want to isolate the main contribution to $b_{k\ell}$ as $k$ and $\ell$ become large. To formulate our result on the asymptotics we need some notation.
Consider a domain $D\in\mathcal{D}_m$ and let $g$ be the exterior mapping as above. The map $g$ can be extended homeomorphically to the unit circle $\T=\partial\D$, and we have
$w_p=g(z_p)$, where $z_p\in\T$, $1\le p\le m$. Write
\begin{equation}\label{gammap}
\gamma_p=2-\alpha_p,
\end{equation}
so that $\pi\gamma_p$ is the exterior angle at the corner $w_p$. Define 
\begin{equation}\label{Kpkl}
K_p(u,v)=-\frac{\gamma_p\sin\pi\gamma_p}{\pi\sqrt{uv}}\frac 1{(u/v)^{\gamma_p}+(v/u)^{\gamma_p}-2\cos\pi\gamma_p},
\end{equation}
for $u,v>0$, and
\begin{equation}\label{Kkl}
K(k,\ell)=\sum_{p=1}^mz_p^{k+\ell}K_p(k,\ell),
\end{equation}
for $k,\ell\ge 1$. Let
\begin{equation}\label{rho}
\rho=\min(1,\gamma_1,\dots,\gamma_m) > 0.
\end{equation}
We have the following asymptotic result for the Grunsky coefficients which will be proved in Section \ref{Subsec:Asym}.

\begin{theorem}\label{Thm:bklAs}
Suppose $D \in \mathcal{D}_m$. Let $b_{k\ell}$ be defined by \eqref{bkl}, and define
\begin{equation}\label{rkl}
r_{k\ell}=b_{k\ell}-K(k,\ell).
\end{equation}
Then there is a constant $C$ so that
\begin{equation}\label{rklEst}
|r_{k\ell}|\le\frac{C\sqrt{k\ell}}{k+\ell}\left(\frac 1{k^\rho\ell}+\frac 1{k\ell^\rho}\right),
\end{equation}
for all $k,\ell\ge 1$.
\end{theorem}
Here and below $C$ denotes a constant that only depends on the geometry of the domain but whose value can change.

When $k=\ell$, we see that $K_p(k,\ell)\sim C/k$, which gives rise to the the divergence of order $\log n$ in the asymptotics, whereas $|r_{kk}|\le C/k^{1+\rho}$, which has a faster decay since $\rho>0$. As can be seen from the proofs  below this faster decay means that terms involving $r_{k\ell}$ will go into a quantity that divided by $\log n$ goes to $0$ as $n\to\infty$. The details of this are complicated however.
\subsection{Discussion}\label{sect:discussion}
Let $D$ be a Jordan domain and let $\Delta_D$ be the Laplace operator with Dirichlet boundary condition on $\partial D$. Write $\lambda_j, j =1,2,\ldots,$ for the positive eigenvalues of $-\Delta_D$. The $\zeta$-regularized determinant of $\Delta_{D}$ is defined using the spectral $\zeta$-function \[\zeta_{-\Delta}(s) = \sum_{j=1}^\infty \lambda_j^{-s} = \frac{1}{\Gamma(s)} \int_0^\infty t^{s-1} \sum_{j=1}^\infty e^{-\lambda_j t}  dt\] analytically continued to a neighborhood of $0$,
\[
\log \textrm{det}_{\zeta} (-\Delta_{D}) := -\zeta_{-\Delta}'(0).
\]
This definition extends to the Laplace-Beltrami operator $\Delta_{D,g}$ for $D$ equipped with a Riemannian metric $g$.
If $\eta = \partial D$ is smooth, then \cite{Wa}
\begin{align}\label{def:Loewner-energy-determinant}
I^L(\eta) = -12 \log \frac{\det_{\zeta}(-\Delta_{D})\det_{\zeta}(-\Delta_{D^*})}{\det_{\zeta}(-\Delta_{\D})\det_{\zeta}(-\Delta_{\D^*})}.
\end{align}   
 It is customary to interpret $\log \textrm{det}_{\zeta} (-\Delta_{D})$ as the free energy for the Gaussian free field on $D$. 
We may compare this with the limiting formula which holds if $\eta$ is a Weil-Petersson quasicircle
\[
I^L(\eta) = -12 \lim_{n \to \infty} \log \frac{\bar Z_n(D)}{\bar Z_n(\D)}  = -12 \lim_{n \to \infty} \log \frac{\bar Z_n^*(D^*)}{\bar Z_n^*(\D^*)} =  -6 \lim_{n \to \infty} \log \frac{\bar Z_n(D) \bar Z_n^*(D^*)}{\bar Z_n(\D) \bar Z_n^*(\D^*)}.\]
See Proposition~\ref{prop:partdet2}, also for the definition of $Z^*_n(D^*)$ for a Coulomb gas in the exterior domain.
Recall that $\log \bar{Z}_n(D)$ is obtained by subtracting from $\log Z_n(D)$ the term $n(n+1) \log r_\infty(D)$, which represents the logarithmic energy of the equilibrium measure. In the limit as $n \to \infty$, we expect $\log \bar{Z}_n(D)$ to relate to Gaussian fluctuations around the equilibrium measure. 

The particular function of the corner angles in Theorem~\ref{Thm:CornerAs} has appeared in other related contexts. In physics, logarithmically divergent terms involving it or very similar expressions are believed to be universal in free energy expansions for planar critical systems in geometries with conical singularities \cite{AS, CardyPeschel, Laskin}. Besides our work here, rigorous results in this direction are known for various settings related to short-time asymptotics of the \emph{heat trace}, see, e.g., \cite{Kenyon, IK, PW}. The relevant formulas seem to trace back to work of Kac, see, e.g., \cite{Kac, vdB, Cheeger, MR} and the references in the latter. The heat trace is the function $t \mapsto \sum_{j} e^{-\lambda_j t}$ and if $\partial D$ is smooth, then as $t \to 0+$, 
\begin{align}\label{heat-trace}
\sum_{j=1}^\infty e^{-\lambda_j t} = \frac{\textrm{area}(D)}{4 \pi t} - \frac{\textrm{length}(\partial D)}{8 \sqrt{\pi} \sqrt{t}} + a_2 + o(1).
\end{align}
Here the constant term $a_2=(\int_D K d^2z + \int_{\partial D} \kappa |dz|)/12\pi = \chi(D)/6=1/6$, where $\chi(D)$ is the Euler characteristic and $K$ and $\kappa$ are the Gauss and geodesic curvatures, respectively. If $D \in  \mathcal{D}_m$, then $a_2$ acquires a corner ``anomaly'' and in this case \begin{align}\label{heat-trace-anomaly} a_2 = \frac{1}{6} + \frac{1}{24}\sum_{p=1}^m
\left(\alpha_p + \frac{1}{\alpha_p} -2\right).\end{align}
As opposed to area and perimeter, $a_2$ is not continuous with respect to smooth approximation of the domain. In Theorem~\ref{Thm:CornerAs} the same function of the interior angles appears in the large $n$ asymptotics for the determinant of the truncated Grunsky operator, which does not seem related to the heat trace in any direct way. However, using Theorem~\ref{Thm:LoewAs} there is \emph{a posteriori} a link via the Loewner energy. When $\eta$ is smooth, $I^L(\eta)$ can be expressed using loop measures as well as $\zeta$-regularized determinants, both which are directly related to the heat trace. (And the expression as in \eqref{heat-trace-anomaly} appears in the value of the $\zeta$-function at $0$, see \cite[Theorem~4.4]{Cheeger}.) See \cite[Theorem~4.1]{WaLoop} for a formula for equipotential approximations using the SLE$_{8/3}$ loop measure, and \cite{APPS}, \cite[Section~6.3]{PW} for discussions of the Brownian loop measure  with small scale cut-off (i.e., small loops are omitted). In the presence of corners a logarithmic divergence involving \eqref{heat-trace-anomaly} appears as the cut-off is removed. These represent geometric regularizations quite different from truncating the Grunsky operator.

In a smooth setting, Dirichlet integrals as in \eqref{def:Loewner-energy} are related to $\zeta$-regularized determinants via the Polyakov-Alvarez formula when using the conformal map to change metric \cite{A}. Let $D \in \mathcal{D}_m$. Corners correspond to conical singularities and while the $\zeta$-regularized determinants can be defined, the Dirichlet integrals \eqref{def:Loewner-energy} are divergent. Let $D_\delta, D^*_\delta$ be the domains obtained by cutting out discs of radius $\delta > 0$ (assumed small) centered at the corners from $D, D^*$ and write $\D_\delta = f^{-1}(D_\delta)$ and  $\D^*_\delta = g^{-1}(D^*_\delta)$. Consider the first integral in \eqref{def:Loewner-energy} restricted to $\D_\delta$. If $D$ is a polygon, then by the Schwarz-Christoffel formula  
\[
\frac{f''(z)}{f'(z)} = -\sum_{p = 1}^m \frac{1-\alpha_p}{z-z_p}, \quad z_p=f^{-1}(w_p).
\]
  (So $|f''(z)/f'(z)|^2$ can be interpreted as the energy density of the electric field in $\D$ generated by point charges of magnitude $1-\alpha_p$ sitting at the points $z_p, p=1,\ldots, m$.)
The preimage of a disc of radius $\delta$ around $w_p$ (intersected with $D$) in $\D$ is very close to a half disc of radius $\sim \delta^{1/ \alpha_p}$ for small $\delta$. 
 We can use this to see that there exists $J(D)$ such that, as $\delta  \to 0$,
\[ \frac{1}{\pi}\int_{\D_\delta}\left|\frac{f''(z)}{f'(z)} \right|^2 d^2 z = \log \delta^{-1} \sum_{p=1}^m \frac{1}{\alpha_p}(1-\alpha_p)^2 + J(D) + o(1) = \log \delta^{-1} \sum_{p=1}^m  \left(\alpha_p + \frac{1}{\alpha_p} -2 \right) +J(D) + o(1).\]
More generally, for $D \in \mathcal{D}_m$, near a corner preimage $z_p$, the map $f$ behaves like $z \mapsto c_p (z-z_p)^{\alpha_p}$ and one can show that $|f''(z)/f'(z)|^2 = (1-\alpha_p)^2|z-z_p|^{-2}(1+o(1))$ (see \cite{Leh} and Lemma~\ref{Lem:regularity}). We obtain a similar asymptotic formula in this case, too. This regularization scheme is one possibility for generalizing the Loewner energy and the class of Weil-Petersson quasicircles to allow for isolated corners.

We note that the right-hand side of \eqref{def:Loewner-energy-determinant} is well-defined even if $D \in \mathcal{D}_m$ but is not the limit of the equipotential approximations. It would be interesting to relate the regularized Dirichlet energy $J(D)$ to the $\zeta$-regularized determinant for the Laplacian in $D$ and the Polyakov-Alvarez formula in the presence of singularities, see \cite{sonoda, Kalvin, AKR}. 

 The Neumann-Poincar\'e operator on a Jordan curve is the Fredholm integral operator arising from the double layer potential. Its spectral analysis is a very classical topic. The non-compact setting of curves with corners was considered already in Carleman's thesis and attracts significant attention today, see \cite{PP} for recent results and further references. There is a well-known link to the Grunsky operator \cite{Schiffer81, TT}, essentially, the singular values of the Grunsky matrix coincide with the Fredholm eigenvalues of the domain. It would be interesting to try to use the methods of this paper to study the spectrum of the Neumann-Poincar\'e operator in similar settings as we consider here. 

\subsection{Conjectures and problems}\label{sect:conjectures}
We first note that we compute the limits in Theorems~\ref{Thm:LoewAs} and~\ref{Thm:PommEn} in an indirect manner using the
Grunsky operator. Is it possible to prove these results in a more direct way starting from the definitions
of the Loewner and Fekete-Pommerenke energies?

It is natural to conjecture that a stronger version of the main theorem holds.
\begin{conjecture}Consider the setting of Theorem \ref{Thm:CornerAs}. The limit
\begin{equation*}
  J^L(D):= \lim_{n\to\infty}-\log \frac{\bar Z_n(D)}{\bar Z_n(\D)}-\frac {\log n}6\sum_{p=1}^m \left(\alpha_p+\frac 1{\alpha_p}-2 \right)
\end{equation*}
exists when $D \in \mathcal{D}_m$. The analogous statement holds for the setting in Theorem \ref{Thm:PommEn}. 
\end{conjecture}
The methods of this paper are not strong enough to prove such a result. Assuming the conjecture is correct it seems reasonable to expect the ``reduced'' Loewner energy $J^L(D)$ to be closely related to the $\zeta$-regularized determinants for $D$ and  $D^*$, see Section~\ref{sect:discussion}. One may also ask about the optimal regularity.

Consider a bounded Jordan curve $\eta$. The partition function for a planar Coulomb gas on the curve at inverse temperature $\beta$ is defined by 
\begin{equation}\label{Partcurve}
    Z_{n,\beta}(\eta)=\frac 1{n!}\int_{\eta^n}e^{-\beta \sum_{1\le k < \ell\le n}\log|z_k-z_\ell|^{-1}}\prod_{k=1}^n|dz_k|.
\end{equation}
It was proved in \cite{Jo} that
\begin{equation}\label{Zn2lim}
    \lim_{n\to\infty}\log\frac{Z_{n,2}(\eta)}{Z_{n,2}(\T)r_\infty(\eta)^{n^2}}=-\frac 12\log\det(I-BB^*) = \frac{1}{24} I^L(\eta),
\end{equation}
if and only if $\eta$ is a Weil-Petersson quasicircle. Furthermore, it was proved in \cite{CoJo} that we have the limit
\begin{equation}\label{Znbetalim}
    \lim_{n\to\infty}\log \frac{Z_{n,\beta}(\eta)}{Z_{n,\beta}(\T)r_\infty(\eta)^{\beta n^2/2+(1-\beta/2)n}}=\frac{1}{24}I^L(\eta) +\frac{1}{8}\left(\sqrt{\frac{\beta}{2}}-\sqrt{\frac{2}{\beta}}\right)^2I^F(\eta),
\end{equation}
if the curve $\eta$ is $C^{12+\alpha}$, $\alpha>0$. 
See also \cite{WiZa} for a non-rigorous derivation. 
\begin{rem}The constant $c=1-6\left(\sqrt{\beta/2}-\sqrt{2/\beta}\right)^2$ is known as the central charge parameter.
\end{rem} 
\begin{lemma}\label{lemma:pommerenkewp}
    Let $\eta$ be a Jordan curve. Then $\eta$ is a Weil-Petersson quasicircle if and only if $I^F(\eta) < \infty$.
\end{lemma}
We prove the lemma in Section~\ref{sect:pommerenke-energy}.

The limit \eqref{Zn2lim} shows that there is a close similarity between the asymptotics of the partition function $Z_n(D)$ of a Jordan domain and the partition function $Z_{n,\beta}(\partial D)$ of its boundary curve $\eta=\partial D$. In view of this similarity, Lemma~\ref{lemma:pommerenkewp}, Theorem \ref{Thm:CornerAs}, Theorem \ref{Thm:PommEn} and \eqref{Znbetalim} it is natural to make the following conjectures.
\begin{conjecture}\label{Conj:discriminantbeta}
Let $\eta$ be a Jordan curve of unit capacity. Then
\begin{align} \label{pommerenke-discriminant-conjbeta}
\varlimsup_{n\to\infty}\log \frac{Z_{n,\beta}(\eta)}{Z_{n,\beta}(\T)} < \infty
\end{align}
if and only if $\eta$ is a Weil-Petersson quasicircle, in which case the limit exists and  \eqref{Znbetalim} holds. 
\end{conjecture}

\begin{conjecture}\label{Conj:Curvecorners}
Let $D \in \mathcal{D}_m$ and write $\eta = \partial D$. Then,
\begin{align}\label{Znbetaconj}
\lim_{n\to\infty}\frac 1{\log n}\log \frac{Z_{n,\beta}(\eta)}{Z_{n,\beta}(\T)r_\infty(\eta)^{\beta n^2/2+(1-\beta/2)n}}=\frac 1{12}\sum_{p=1}^m (\alpha_p+\frac 1{\alpha_p}-2)+
\frac 14\left(\sqrt{\frac{\beta}{2}}-\sqrt{\frac{2}{\beta}}\right)^2\sum_{p=1}^m (\gamma_p+\frac 1{\gamma_p}-2),
\end{align}
where  $\gamma_p  = 2-\alpha_p$ are the exterior angles.
\end{conjecture}
For more regular curves $\eta$ it was possible in \cite{Jo} to derive asymptotics for the partition function $Z_{n,2}(\eta)$, which is a generalized Toeplitz determinant, by, in a sense, reducing it to an application of the classical strong Szeg\H{o} limit theorem. For $\beta=2$ Conjecture \ref{Conj:Curvecorners} has some
similarities with Fisher-Hartwig asymptotics for singular symbols. However, we have not been able to see any direct connection.

Equation \eqref{Znbetalim} does not make sense in the limit $\beta \to 
 +\infty$ without further normalization. But given Theorem~\ref{thm:pomthm} and Lemma~\ref{lemma:pommerenkewp}, it is natural to conjecture the following statement for the free energy of Fekete points.
\begin{conjecture}\label{Conj:discriminant}
Let $\eta$ be a Jordan curve of unit capacity. Then
\begin{align} \label{pommerenke-discriminant-conj}
\varlimsup_{n \to \infty} 8 \log \frac{Z_{n,\infty}(\eta)}{Z_{n,\infty}(\T)} < \infty
\end{align}
if and only if $\eta$ is a Weil-Petersson quasicircle, in which case the limit exists and equals the Fekete-Pommerenke energy $I^F(\eta)$. 
\end{conjecture}
 By Satz~2 of \cite{pom67}, \eqref{pommerenke-discriminant-conj} implies that $\eta$ is a Weil-Petersson quasicircle. In the opposite direction, \cite{pom65} shows that if  $\eta \in C^{2+\epsilon}, \, \epsilon >0$, then \eqref{pommerenke-discriminant-conj} holds. If $\eta$ is analytic, then \cite{pom69} shows that the limit exists and equals $I^F(\eta)$. 
\subsection{Outline of the paper}
Section~\ref{sect:prel} gives basic definitions, discusses preliminary material on the Grunsky operator and Faber polynomials, as well as the Fekete-Pommerenke energy. Then in Section~\ref{Subsec:PartDet} we prove the basic determinant formula Proposition~\ref{Prop:PartDet}. Section~\ref{Sec:propthm} contains the proofs of the main theorems, assuming certain more technical results whose proofs are given in Section~\ref{sec:proofsbasic} and Section~\ref{Sec:lemmas}.
\subsection*{Acknowlegements}K.\ J.\ acknowledges support from the Knut and Alice Wallenberg Foundation. F.\ V.\ acknowledges support from the Knut and Alice Wallenberg Foundation and the Swedish Research Council (VR). We thank Yilin Wang and Mingchang Liu for comments on the paper, Mihai Putinar for suggesting additional references, and Paul Wiegmann for interesting discussions. We thank the anonymous referees for helpful comments and suggestions that led to some simplifications, the correction of mistakes, and improved the presentation.

\section{Preliminaries}\label{sect:prel}
\subsection{Basic definitions and notation}
Let $D$ be a bounded Jordan domain containing $0$ and write $\eta = \partial D$. Let $f: \D \to D$ be the conformal map such that $f(0)=0, f'(0)>0$. The conformal radius of $D$ with respect to $0$ is then defined by $r_0(D) = f'(0)$. Let $D^* = \hat{\mathbb{C}} \smallsetminus D$ be the simply connected domain exterior to $\eta$ and let $g : \D^* \to D^*$ be the exterior conformal map such that $g(\infty) = \infty$. Then $g(z) = r_\infty z + O(1)$ as $z \to \infty$, where $r_\infty=r_\infty(D)$ is the capacity of $D$. Note that $r_\infty^{-1}$ is the conformal radius of $D^*$ with respect to $\infty$. 

By Carath\'eodory's theorem (see \cite{GM}), both conformal maps $f$ and $g$ extend to homeomorphisms $\mathbb{T} \to \eta$. The map $g^{-1}\circ f : \mathbb{T} \to \mathbb{T}$ is the welding homeomorphism associated to $\eta$. We say that the Jordan curve $\eta$ is a quasicircle if it is the image of the unit circle under a quasiconformal homeomorphism of the plane and the welding homeomorphism associated to a quasicircle is said to be quasisymmetric, see Chapter~VII of \cite{GM}. Recall that $\eta$ is a Weil-Petersson quasicircle if and only if $I^L(\eta) < \infty$, where the Loewner energy was defined in \eqref{def:Loewner-energy}. An equivalent characterization is that the welding homoeomorphism $\alpha : \mathbb{T} \to \mathbb{T}$ satisfies $\log |\alpha'| \in H^{1/2}(\T)$, where $H^{1/2}$ is defined just below.

We say that $\eta$ has a corner at $f(e^{it})=w \in \eta$ of inner opening angle $\alpha \pi,  0 < \alpha < 2, \,  \alpha \neq 1,$ if there exists $\beta$ such that $\lim_{s \to 0+} \arg(f(e^{i(t+s)}) - f(e^{it})) = \beta$ while $\lim_{s \to 0-} \arg(f(e^{i(t+s)}) - f(e^{it})) = \beta + \alpha \pi$. It is easy to see that this definition is independent of the particular choice of conformal map. See Chapter~3 of \cite{PommerenkeBoundary}. Note that our definition of ``corner'' excludes the case where the arcs meet at an angle $\pi$.

An analytic (Jordan) arc is the image of a closed interval $I$ under a map conformal in a neighborhood of $I$. In particular, by definition, the arc is analytic at the end-points.

For $u : D \to \mathbb{C}$ we write
\[
\mathcal{D}_D(u) = \frac{1}{\pi}\int_D |\nabla u|^2 d^2z
\]
for the Dirichlet energy of $u$ in $D$. As is seen by direct computation, this quantity is conformally invariant: if $f:D \to f(D)=:D'$ is a conformal map, then $\mathcal{D}_{D'}(u \circ f^{-1}) =\mathcal{D}_{D}(u)$.

Let $H^{1/2}(\T)$ be the space of real-valued functions $f:\T \to \mathbb{R}$ such that $f \in L^2$, and if $\sum_{k=-\infty}^\infty f_k e^{ik\theta}$ is the Fourier series for $f$ ($f_{-k} = \overline{f}_k$), then
\[
 \|f\|_{\mathscr{H}}^2 : = \sum_{k=-\infty}^\infty |k||f_k|^2  = 2\sum_{k=1}^\infty k|f_k|^2 < \infty.
\]
We write $H^{1/2}(\T)/\mathbb{R}$ for the elements of $H^{1/2}(\T)$ such that $\int_0^{2\pi} f(e^{i\theta}) d\theta = 0$, i.e., $f_0 =0$.
It is well-known that the Fourier series for such $f$ converge quasi-everywhere on $\T$.  The space $H^{1/2}(\T)/\mathbb{R}$ becomes a real Hilbert space using the norm $\|\cdot\| = \|\cdot\|_{\mathscr{H}}$, and we write $\mathscr{H}=\mathscr{H}(\T)$ for this Hilbert space, see \cite{NS}. 

Let $f \in H^{1/2}(\T)$ and consider its harmonic extension (Poisson integral) $u_f$ to $\D$ 
\[
u_f(z) = \frac{1}{2\pi} \int_{0}^{2\pi}\frac{1-|z|^2}{|z-e^{i\theta}|^2}f(e^{i\theta}) d\theta, \quad z \in \mathbb{D}.
\]
Then a computation shows
\begin{align}\label{eq:douglas}
    \mathcal{D}_{\D}(u_f) = 2 \|f\|_{\mathscr{H}}^2. 
\end{align}
Conversely, given a harmonic $u$ with $\mathcal{D}_{\D}(u) < \infty$, one may define a trace $f \in H^{1/2}(\T)$ by taking non-tangential limits and \eqref{eq:douglas} holds. 

 The conjugation operator (Hilbert transform) $J:\mathscr{H} \to \mathscr{H}$ is defined by
 \[
 J:  \sum_{k=-\infty}^\infty f_k e^{ik\theta} \mapsto -  \sum_{k=-\infty}^\infty i\,  \textrm{sgn}(k) f_k e^{ik\theta},
 \]
 where $ \textrm{sgn}(0)=0,  \textrm{sgn}(k) = +1$ if $k > 0$ and  $ \textrm{sgn}(k) = -1$ if $k<0$. 
\subsection{Grunsky operator and Faber polynomials}We now review material on the generalized Grunsky operator and Faber polynomials, see \cite{Schiffer81, TT}. We shall mostly use the set up and notation of \cite{TT} but we warn the reader that we use $a_{k\ell}$ for what is called $b_{k\ell}$ in \cite{TT}. 

Let $D$ be a Jordan domain containing $0$ as an interior point with $\eta = \partial D$ and let $f$ and $g$ be the interior and exterior conformal maps as above and assume that $f(0)=0, f'(0)=1$. We now define the generalized Grunsky coefficients associated to the normalized pair $(f,g)$.  Write $a=r_\infty(D)$ and $a_{00} = \log a$. Then
\begin{equation}\label{defGrunskyB1}
\log\frac{f(\zeta)-f(z)}{\zeta-z}=a_{00}-\sum_{k,\ell=0}^\infty a_{-k,-\ell}\zeta^{k}z^{\ell}, \quad |\zeta|,|z| < 1;
\end{equation}
\begin{equation}\label{defGrunskyB4}
\log\frac{g(\zeta)-g(z)}{\zeta-z}=a_{00}-\sum_{k,\ell=1}^\infty a_{k,\ell}\zeta^{-k}z^{-\ell}, \quad |\zeta|,|z| > 1;
\end{equation}
and finally
\begin{equation}\label{defGrunskyB3}
\log\frac{g(z)-f(\zeta)}{a z}=-\sum_{k=1}^\infty\sum_{\ell=0}^\infty a_{k,-\ell}z^{-k}\zeta^\ell, \quad |\zeta| < 1,\, |z| > 1.
\end{equation}
Next define semi-infinite matrices for $k, \ell \ge 1$,
\[
(B_1)_{k\ell} =b_{-k, -\ell}:= \sqrt{k\ell} a_{-k, -\ell}, \quad (B_4)_{k\ell} =b_{k, \ell}:= \sqrt{k\ell} a_{k,\ell}. 
\]
\[
(B_2)_{k\ell} = b_{-k, \ell} := \sqrt{k\ell}a_{-k,\ell}, \quad   (B_3)_{k\ell} = b_{k,- \ell} := \sqrt{k\ell}a_{k,-\ell}, 
\]
where for $k, \ell \ge 1$, $a_{-k,\ell} := a_{\ell,-k}$. Note that the coefficients in the matrices $B_j, j=1,\ldots, 4$ are unchanged if $f$ and $g$ are simultaneously multiplied by $\lambda > 0$. Note also that $B_4 = B$ is the classical Grunsky operator as defined in the introduction. 
We will need the following inequality which is a strengthening of the classical Grunsky inequality.
Assume that $\eta$ is a quasicircle. Then there is a constant $\kappa<1$ such that
\begin{equation}\label{Grunskyin}
\left|\sum_{k,\ell=1}^\infty b_{k\ell}w_kw_\ell\right|\le\kappa\sum_{k=1}^\infty |w_k|^2,
\end{equation}
for all complex sequences $\{w_k\}$, see \cite[Sec. 9.4]{Po}. 

The generalized Grunsky operator \[G= \begin{pmatrix}
B_1 & B_2 \\
B_3 & B_4
\end{pmatrix}\] acting on $\ell^2(\mathbb{Z}_+) \oplus \ell^2(\mathbb{Z}_+)$ is unitary in the present setting, see Section~II.2.1 of \cite{TT}. In particular we have the following lemma, see $(II.2.2)$ of \cite{TT}:
\begin{lemma}\label{grunsky-identities}
We have that
\begin{equation}\label{B1B2}
(B_2B_2^*)_{k\ell}=\delta_{k\ell}-(B_1B_1^*)_{k\ell}, \qquad k,\ell \ge 1,
\end{equation} 
and
   \begin{equation}\label{B3B4}
(B_3B_3^*)_{k\ell}=\delta_{k\ell}-(B_4B_4^*)_{k\ell}, \qquad k,\ell \ge 1.
\end{equation} 
\end{lemma}

The \emph{Faber polynomials} $P_k(w)=w^k/r_\infty^k+\dots$ associated to the exterior map $g$ are defined by
\begin{equation}\label{Faberg}
\log\frac{g(z)-w}{r_\infty z}=-\sum_{k=1}^\infty\frac{P_k(w)}kz^{-k},
\end{equation}
for $z$ in a neighborhood of infinity. 
If we put $w=f(\zeta)$ in \eqref{Faberg} and use \eqref{defGrunskyB3}, we see that
\begin{equation}\label{Pnf}
P_k \circ f(\zeta)=ka_{k0}+k\sum_{\ell=1}^\infty a_{k,-\ell}\zeta^\ell.
\end{equation}
Similarly, we define generalized Faber polynomials (in $1/w$): $Q_k(w) = 1/w^k + \ldots$ for the interior map $f$ by
\begin{align}\label{Faberf}
\log \frac{w-f(z)}{w} = \log \frac{f(z)}{z} - \sum_{k=1}^\infty \frac{Q_k(w)}{k} z^k
\end{align}
and we have the relation
\begin{align}\label{Qng}
Q_k \circ g(z) = -k a_{-k,0} + k \sum_{\ell = 1}^\infty a_{\ell,-k} z^{-\ell}.
\end{align}
Write
\begin{equation}\label{ek}
e_k(\zeta)=\sqrt{\frac k{\pi}}\zeta^{k-1}, \quad f_k(z) = \sqrt \frac{k}{\pi}z^{-(k+1)}
\end{equation}
and note that $\{e_k\}_{k \ge 1}$ and $\{f_k\}_{k \ge 1}$ form orthogonal families of functions for the $L^2$ spaces on $\D$, $\D^*$, respectively:
\begin{equation}\label{ekorth}
\int_{\D}e_k(\zeta)\overline{e_\ell(\zeta)}\,d^2\zeta=\int_{\D^*}f_k(\zeta)\overline{f_\ell(\zeta)}\,d^2\zeta=\delta_{k\ell}.
\end{equation}
Further, note that if we differentiate \eqref{Pnf} and \eqref{Qng} we obtain
\begin{equation}\label{Pnprime}
P_k'(f(\zeta))f'(\zeta)=\sum_{\ell=1}^\infty \sqrt{k\ell} b_{k,-\ell}\zeta^{\ell-1}, \quad Q'_k(g(z))g'(z) = -\sum_{\ell = 1}^\infty
\sqrt{k\ell} b_{\ell, -k} z^{-(\ell+1)}.\end{equation}
Now define
\begin{equation*}
\alpha_k(\zeta)=\frac{P_k'(\zeta)}{\sqrt{\pi k}}, \quad \beta_k(z)= \frac{Q'_k(z)}{\sqrt{\pi k}}.
\end{equation*}
\begin{lemma}\label{lem:Faber-Fourier}
Suppose $D$ is a bounded Jordan domain containing $0$ whose boundary $\eta$ is a Weil-Petersson quasicircle. For any sequence ${\bf h}=(\sqrt{k} h_{k})_{k=1}^\infty \in \ell^2(\Z_+)$, the Faber series associated to the exterior map $g$,
\[H = \lim_{N \to \infty} \sum_{k=1}^N h_k P_k\] 
converges locally uniformly in $D$ to an analytic function with finite Dirichlet energy. Moreover, for each $k =1,2 ,\ldots$, the series $\sum_\ell \ell a_{k\ell} h_\ell=: h_{-k}$ converges and $H \circ g(e^{i\theta})=\sum_{k=-\infty}^\infty h_k e^{ik \theta}$ a.e.\ on $\T$. 
\end{lemma}
\begin{proof}
The existence of the limit defining $H$ and the fact that it has finite Dirichlet energy is Theorem~4.1 of \cite{Shen09}. It follows that $H$ has non-tangential limits a.e.\ on $\eta$ and we write $H$ for the so defined function on $\eta$ as well. Since $\eta$ is a Weil-Petersson quasicircle, it is in particular a chord-arc curve (see e.g., \cite{Bi}). Therefore, harmonic measure and arc-length are mutually absolutely continuous on $\eta$, see Theorem~VII.4.3 of \cite{GM}. Using this, the fact that $g$ extends continuously to $\mathbb{T}$ implies that $H \circ g$ is well-defined a.e.\ on $\mathbb{T}$. We wish to show that $H \circ g \in H^{1/2}$. For this, note that on $\T$, $H\circ g = H \circ f \circ \beta$, where $\beta = f^{-1} \circ g$ is the welding homeomorphism of the quasicircle $\eta$ and hence quasisymmetric. Since $H \circ f \in H^{1/2}(\mathbb{T})$ by conformal invariance, the claim follows since precompositon with a quasisymmetric homeomorphism preserves $H^{1/2}(\mathbb{T})$, see Theorem~3.1 of \cite{NS}. It follows that the Fourier series for $H \circ g$ converges a.e.\ on $S^1$ and for $k=0,1,\ldots$, $h_k = \int (H \circ g(e^{i\theta})) e^{-ik\theta} d\theta/2\pi$. 

The final statement about summability follows from the proof of the lemma on p.\ 591 of \cite{curtiss} with a stronger assumption on the curve but a weaker one on $H$: Since $\log |g'| \in H^{1/2}(\mathbb{T})$, a weak-type estimate implies $|g'| \in L^p(\mathbb{T})$ for all $p < \infty$ (see, e.g., Lemma~5.1 of \cite{VW20}) and the only step left to check is the application of the Cauchy formula since $H$ is not assumed to be continuous. By \cite[Theorem 10.4]{Duren}, if $(H \circ f(z)) f'(z) \in \mathcal{H}^1$ (the Hardy space), then the Cauchy formula in $D$ applies. We know that $H\circ f$ is analytic with finite Dirichlet energy, so in particular $H\circ f \in \mathcal{H}^2$. Moreover, since $\log f'$ has finite Dirichlet energy, $f' \in \mathcal{H}^p$ for all $p < \infty$. H\"older's inequality then immediately implies that $(H \circ f(z)) f'(z) \in \mathcal{H}^1$.
\end{proof}
\begin{rem}
    Since $b_{k\ell} = \sqrt{k \ell}a_{k \ell}$ we see that the last statement of the lemma provides an interpretation for the action of the Grunsky operator $B=B_4$. Namely, in the notation of the lemma we have $\sqrt{k} h_{-k} = \sum_\ell b_{k\ell}  \sqrt{\ell} h_\ell$. In other words, the sequence of (scaled) negative index Fourier coefficients of $H \circ g$ is obtained from the sequence of (scaled) positive index Fourier coefficients by acting on the latter with $B_4$.
\end{rem}
\subsection{The Fekete-Pommerenke energy} \label{sect:pommerenke-energy}
Let $\eta$ be a Jordan curve with interior and exterior domain $D$ and $D^*$, respectively. As above, we let $f$ and $g$ be the interior and exterior conformal maps.
Recall that $I^F(\eta)$ was defined in \eqref{def:pommerenke-energy}.
\begin{lemma}\label{lemma:pommerenkewp2}
    The Jordan curve $\eta$ is a Weil-Petersson quasicircle if and only if $I^F(\eta) < \infty$.
\end{lemma}
\begin{proof} Since $\eta$ is assumed to be a Jordan curve, \cite[Thm 2.1.12]{TT} implies that $\eta$ is a Weil-Petersson quasicircle if and only if $\mathcal{D}_{D^*}(\phi_e) < \infty$. We only have to prove that if $\eta$ is a Weil-Petersson quasicircle, then $\mathcal{D}_{D}(\phi_i) < \infty$. For this, note that $\mathcal{D}_{\D^*}(\log|g'|) < \infty$ so $\log|g'|$ has non-tangential limits a.e.\ on $\T$. By \eqref{eq:douglas}, the trace $\log |g'| \in H^{1/2}(\T)$. Since $\eta$ in particular is a quasicircle, the welding homeomorphism $\alpha =  g^{-1} \circ f$ is quasisymmetric on $\T$ and consequently $\log|g'| \circ \alpha \in H^{1/2}(\T)$ (see \cite{NS}). Since $\phi_i \circ f$ is the Poisson integral (for $\D$) of $\log|g'| \circ \alpha$, conformal invariance and \eqref{eq:douglas} show
\[
\mathcal{D}_D(\phi_i) = \mathcal{D}_{\D}(\phi_i \circ f) < \infty,
\]
as desired.
\end{proof}
We will now express $I^F(\eta)$ using the Grunsky operator $B=B_4$ for $\eta$. 
We have an expansion
\[
\log g'(z) = \log r_\infty -\sum_{k=1}^\infty  d_k z^{-k}.
\]
In terms of the Grunsky coefficients,
\[
d_k = \sum_{j=1}^{k-1}a_{j, k-j} = \sum_{j=1}^{k-1}\frac{1}{\sqrt{j(k-j)}}b_{j, k-j}.
\]
Define
\begin{equation}\label{dvec}
    {\bf d}=(\sqrt{k} d_k)_{k \ge 1}.
\end{equation}
Let ${\bf v} = (\sqrt{k} v_k)_{k \ge 1}^t \in \ell^2(\Z_+)$ and write $\overline {\bf v} = (\sqrt{k} \bar{v}_k)_{k \ge 1}^t$.
Pommerenke's analysis of Fekete points on $\eta$ led him to considering the following equation under the assumption that $\eta$ is analytic:
\begin{equation}\label{Peq}
\overline {\bf v} + B {\bf v}= {\bf u},
\end{equation}
where ${\bf u} = (\sqrt{k} u_k)_{k \ge 1}^t \in \ell^2(\Z_+)$ is given.
Splitting into real and imaginary parts, we can write this as 
\begin{equation}\label{Peq-linear}
 (I+\mathcal{K}) \begin{pmatrix}
     \re {\bf v} \\ -\im {\bf v}
 \end{pmatrix}
 =
\begin{pmatrix}
     \re {\bf u} \\ \im {\bf u}
 \end{pmatrix}.
\end{equation}
Here for $B=B^{(1)}+iB^{(2)}$  we have defined the operator $\mathcal{K}$ on $\ell^2(\Z_+)\oplus\ell^2(\Z_+)$ by
\begin{equation}\label{Kop}
 \mathcal{K}=\begin{pmatrix} B^{(1)} & B^{(2)}\\  B^{(2)} &-B^{(1)} \end{pmatrix}.  
\end{equation}
By the Grunsky inequality, $\mathcal{K}$ is a strict contraction if $\eta$ is a quasicircle. In particular, in this case, \eqref{Peq-linear} (and equivalently \eqref{Peq}) has a solution which determines a unique ${\bf v} \in \ell^2(\Z_+)$. 
 The Grunsky operator has the representation $B=U\Lambda U^t$, where $U=R+iS$ is unitary and $\Lambda=\text{diag}(\lambda_1,\lambda_2,\dots)$ is a diagonal matrix whose eigenvalues are the singular values of $B$. If we let
  \begin{equation*}
      T=\begin{pmatrix}
          R  &S \\ S &-R
      \end{pmatrix},
  \end{equation*}
  then $T$ is orthogonal and 
  \begin{equation}\label{Kdiag}
      \mathcal{K}=T\begin{pmatrix}
          \Lambda & 0 \\  0 & -\Lambda
      \end{pmatrix}T^t.
  \end{equation}

The following is Proposition~4.1 of \cite{NS}.
\begin{lemma}\label{lem:S-symmetry}
Define 
\begin{align}\label{NS-symplectic}
S(f,g) = - i\sum_{k=-\infty}^\infty k \hat f_k \hat g_{-k}, \quad f,g \in \mathscr{H},
\end{align}
where $\hat{f}_k$ and $\hat{g}_k$ denote Fourier coefficients of $f$ and $g$, respectively.

Suppose $\alpha: \T \to \T$ is a quasisymmetric homeomorphism and for $f \in \mathscr{H}$, define the normalized composition operator
\[
 V_\alpha:  f \mapsto f\circ \alpha - \frac{1}{2\pi}\int_0^{2\pi} f \circ \alpha(e^{i\theta}) \, d\theta.
\]
 Then $ V_\alpha f \in \mathscr{H}$ and 
\[
S( V_\alpha f,  V_\alpha g) = S(f,g),
\]
where $S$ is as in \eqref{NS-symplectic}.
\end{lemma}

\begin{lemma}\label{lem:Penergy-Dirichlet}
Suppose $\eta$ is a Weil-Petersson quasicircle. Then 
\begin{equation}\label{Epomm}
 I^F(\eta) =2\begin{pmatrix}
     \re {\bf d} \\ \im {\bf d}
 \end{pmatrix} ^t(I+\mathcal{K})^{-1} 
\begin{pmatrix}
     \re {\bf d} \\ \im {\bf d}
 \end{pmatrix}.
\end{equation}
\end{lemma}
\begin{proof}
First note that \eqref{Epomm} is invariant with respect to scaling $\eta$, so we may assume that $r_\infty = 1$.
By \eqref{eq:douglas}, ${\bf d} \in \ell^2(\Z_+)$ if and only if $\mathcal{D}_{\D^*}(\log|g'|) < \infty$ which by \cite[Thm 2.1.12]{TT} in turn holds if and only if $\eta$ is a Weil-Petersson quasicircle. So the solution ${\bf h}=(\sqrt{k} h_k)_{k \ge 1} \in \ell^2(\Z_+)$ to \eqref{Peq} with ${\bf u} = {\bf d}/2$ exists and is uniquely determined. 
Recalling \eqref{Peq-linear} we see that
\[2 \Re \sum_{k = 1}^\infty k h_k d_k=\begin{pmatrix}
     \re {\bf d} \\ \im {\bf d}
 \end{pmatrix} ^t(I+\mathcal{K})^{-1} 
\begin{pmatrix}
     \re {\bf d} \\ \im {\bf d}
 \end{pmatrix}.
 \]
 The left-hand side is ($4$ times) the quantity appearing in Pommerenke's work \cite{pom67, pom69}.
Let $P_k$ be the Faber polynomials (attached to $g$) for $D$. By Lemma~\ref{lem:Faber-Fourier} \[
H(z) = \lim_{n \to \infty} \sum_{k=1}^n h_k P_k(z).
\] 
defines a function analytic in $D$ with finite Dirichlet energy and on $\T$
\[H \circ g = \sum_{k=-\infty}^\infty h_k e^{ik\theta},\]
where for $k=1,2,\ldots$, $h_{-k}$ is obtained from $h_k$ as in the statement of Lemma~\ref{lem:Faber-Fourier}.
We claim that
$\Re H \circ g = -\log|g'|/2 \in \mathscr{H}$. Indeed, first note that
\[
 - \frac{1}{2}\log|g'| = \frac{1}{4} \sum_{k=1}^\infty d_k e^{-ik \theta} + \overline{d_k} e^{ik \theta} \in \mathscr{H}.
\]
With $u_k=d_k/2$, \eqref{Peq} gives 
\begin{equation} \label{P:eq2}
\overline{h_k} + \sum_{\ell =1}^\infty \ell a_{k\ell} h_{\ell} = \frac{d_k}{2}, \quad k=1,2,\ldots.
\end{equation}
By the second part of Lemma~\ref{lem:Faber-Fourier} this implies 
\[
\frac{1}{2}(h_k + \overline{h_{-k}})= \overline{\frac{d_k}{4}}, \quad k=1,2,\ldots.
\]
which proves the claim. Recall the definition of the bilinear form $S$ in \eqref{NS-symplectic}.
Note that 
\[
S(u,Ju) = \|u\|^2,
\]
 where $J$ denotes the Hilbert transform. 
If $H_g := H \circ g$ we can write 
\[   \Re \sum_{k = 1}^\infty k h_k d_k =S(\Re H_g, \arg g') + S(\Im H_g, \log|g'|).\]
We know that $\Re H_g = - \log| g' | /2$ and we have $\arg g' = -J \log|g'|$, so we deduce \[S(\Re H_g, \arg g') = \|\log|g'|\|^2/2= \mathcal{D}_{\D^*}(\log|g'|)/4.\] 
Let $\alpha = g^{-1} \circ f$ be the welding homeomorphism for $\eta$ and write $V_\alpha$ for the corresponding  composition operator, see Lemma~\ref{lem:S-symmetry}. Then $H_g\circ \alpha =  H\circ f \mid_{\T}$ represents the trace to $\T$ of a function analytic in $\D$. Hence, in $\mathscr{H}$, $\Im V_\alpha H_g = J   \Re V_\alpha H_g = -J V_\alpha \log |g'|/2.$ Using Lemma~\ref{lem:S-symmetry} we get
\begin{align*}
S(\Im H_g, \log|g'|) & =   S(-J V_\alpha \log |g'|, V_\alpha \log|g'|)/2 \\ 
& =\|V_\alpha \log |g'|\|^2/2.
\end{align*}
On the other hand, let $u$ be the harmonic extension of $\log |g'|\circ \alpha  = -\log|(g^{-1})'| \circ f$ to $\D$. Then using conformal invariance,
\[
\|V_\alpha \log |g'|\|^2 = \mathcal{D}_{\D}(u)/2 =  \mathcal{D}_{D}(u \circ f^{-1})/2 =  \mathcal{D}_{D}(\phi_i)/2.
\]
This completes the proof.
\end{proof}
\begin{rem}
In a smooth setting, another approach to Lemma~\ref{lem:Penergy-Dirichlet} uses a representation of the operator $\mathcal{K}$ in terms of layer potentials and the Neumann jump operator (see \cite{CoJo, WiZa}). While that approach is arguably more geometrically transparent, working with Fourier series allows a short proof for the statement in optimal regularity. 
\end{rem}
We finally give a lemma that expresses the Fekete-Pommerenke energy directly in terms of the Grunsky operator instead of the operator $\mathcal{K}$. This formula is the starting point for the proof of Theorem \ref{Thm:PommEn}.
\begin{lemma}\label{Lem:KtoB}
Suppose $\eta$ is a Weil-Petersson quasicircle. Then
\begin{equation}
\begin{pmatrix}
     \re {\bf d} \\ \im {\bf d}
 \end{pmatrix} ^t(I+\mathcal{K})^{-1} 
\begin{pmatrix}
     \re {\bf d} \\ \im {\bf d}
 \end{pmatrix}=\re {\bf d}^*(I-BB^*)^{-1}({\bf d}-B\bar{{\bf d}})
 \end{equation}
\end{lemma}

\begin{proof}
Let ${\bf h}=(\sqrt{k} h_k)_{k \ge 1} \in \ell^2(\Z_+)$ be the solution to \eqref{Peq} with ${\bf u} = {\bf d}/2$, where ${\bf d}$ is as in \eqref{dvec}. Then
\[
I^F(\eta) = 4 \Re \sum_{k = 1}^\infty k h_k d_k.
 \]
 In terms of ${\bf v} = 2{\bf h}$, which satisfies 
 \begin{align}\label{eq:v}
 \overline{{\bf v}} + B {\bf v} = {\bf d},
 \end{align}
 we can write this as
 \begin{align}\label{eq:v0}
 I^F(\eta) = 2\Re {\bf d}^* \overline{{\bf v}}.
 \end{align}
 Taking the conjugate of \eqref{eq:v} and multiplying by $B$ gives,
  \begin{align}\label{eq:v2}
  B {\bf v} +BB^* \overline{{\bf v}} = B \overline{{\bf d}}.
\end{align}
This uses the symmetry of Grunsky coefficients, $B^* = \overline{B}$. We now subtract \eqref{eq:v2} from \eqref{eq:v} and get
\[
(I-BB^*) \overline{{\bf v}}  = {\bf d} - B \overline{{\bf d}}.
\]
We solve for $\overline{{\bf v}}$ and plug into \eqref{eq:v0},
\[
\re {\bf d}^*(I-BB^*)^{-1}({\bf d}-B\bar{{\bf d}}) = \frac{1}{2}I^F(\eta)
\]
On the other hand, by Lemma~\ref{lem:Penergy-Dirichlet} we have 
    \[
    I^F(\eta) = 2 \begin{pmatrix}
     \re {\bf d} \\ \im {\bf d}
 \end{pmatrix} ^t(I+\mathcal{K})^{-1} \begin{pmatrix}
     \re {\bf d} \\ \im {\bf d}
 \end{pmatrix},
    \]
 so combining the last two formulas completes the proof of the lemma.
\end{proof}
\section{Determinant formula}\label{Subsec:PartDet}
We now give the proof of Proposition \ref{Prop:PartDet}, which essentially follows from arguments in \cite{TT}, Section~2.2.3. We prove a more general result which also includes a statement about the partition function for a Coulomb gas with an appropriate potential on the unbounded domain $D^*$. Define  \[
 Z_n^*(D^*) = \frac{1}{n!}\int_{(D^*)^n} e^{-(n+1) \sum_{k=1}^n\log \, |z_k|^2}\prod_{1\le k < \ell  \le n}|z_k-z_\ell|^2 \prod_{\mu =1}^n d^2 z_\mu.
  \]
\begin{proposition}\label{prop:partdet2}
    Let $\eta$ be a bounded Jordan curve and write $D$ and $D^*$ for the interior and exterior domains of $\eta$, respectively. Assume $0 \in D$. We have the identities 
\begin{equation}\label{eq:free-energy1}
\log Z_n(D)=\log \frac{\pi^n}{n!} + n(n+1) \log  r_\infty(D)  + \log \det(I-P_nB_4B_4^*P_n)_{\ell^2(\Z_+)}
\end{equation}
and
\begin{equation}\label{eq:free-energy2}
\log Z_n^*(D^*)=\log \frac{\pi^n}{n!} - n(n+1) \log  r_0(D)  + \log \det(I-P_nB_1B_1^*P_n)_{\ell^2(\Z_+)}
\end{equation}
for $n\ge 1$.
\end{proposition}
    \begin{proof}
    We first prove \eqref{eq:free-energy1}. We will write $r_\infty = r_\infty(D)$ throughout.
    It follows from Andr\'eief's identity (the generalized Cauchy-Binet identity) that 
\begin{equation}\label{Zndet}
Z_n(D)=\det\left(\int_D z^{k-1}\bar{z}^{\ell-1}\,d^2z\right)_{1\le k,\ell\le n}.
\end{equation}
Recall the definition
\begin{equation*}
\alpha_k(\zeta)=\frac{P_k'(\zeta)}{\sqrt{\pi k}},
\end{equation*}
where $P_k$ is the $k$:th Faber polynomial attched to the exterior map $g$. Note that $\alpha_k$ is a polynomial of degree $k-1$ with leading coefficient $\sqrt{k/\pi r_\infty^k}$. Row and column operations and a change of variables in \eqref{Zndet} give
\begin{align*}
Z_n(D)&=\frac{\pi^n r_\infty^{n(n+1)}}{n!}\det\left(\int_D \sqrt{\frac k{\pi r_\infty^k}}z^{k-1}\sqrt{\frac \ell{\pi r_\infty^k}}z^{\ell-1}\,d^2z\right)_{1\le k,\ell\le n} \\
&=
\frac{\pi^n r_\infty^{n(n+1)}}{n!}\det\left(\int_D\alpha_k(z)\overline{\alpha_\ell(z)}\,d^2z\right)_{1\le k,\ell\le n}\\
&=\frac{\pi^n r_\infty^{n(n+1)}}{n!}\det\left(\int_\D\alpha_k(f(\zeta))f'(\zeta)\overline{\alpha_\ell(f(\zeta))f'(\zeta)}\,d^2\zeta\right)_{1\le k,\ell\le n}.
\end{align*}
Now, by \eqref{Pnprime}, we have the identity
\begin{equation*}
\alpha_k(f(\zeta))f'(\zeta)=\sum_{j=1}^\infty b_{k,-j}e_j(\zeta).
\end{equation*}
Recall that $(B_3)_{k\ell} = b_{k, -\ell}$ so using \eqref{ekorth} and \eqref{B3B4} of Lemma~\ref{grunsky-identities} it follows that
\begin{align*}
Z_n(D)&=\frac{\pi^n r_\infty^{n(n+1)} }{n!}\det\left(\sum_{j_1,j_2=1}^\infty b_{k,-j_1}\overline{b_{\ell,-j_2}}\int_\D e_{j_1}(\zeta)\overline{e_{j_2}(\zeta)}\,d^2\zeta\right)_{1\le k,\ell\le n}\\
&=\frac{\pi^n r_\infty^{n(n+1)} }{n!}\det\left(\sum_{j=1}^\infty b_{k,-j}\overline{b_{\ell,-j}}\right)_{1\le k,\ell\le n}=\frac{\pi^nr_\infty^{n(n+1)}}{n!}\det\left((B_3B_3^*)_{k\ell}\right)_{1\le k,\ell\le n}\\
&=\frac{\pi^n r_\infty^{n(n+1)}}{n!}\det\left(\delta_{k\ell}-(B_4B_4^*)_{k\ell}\right)_{1\le k,\ell\le n}=\frac{\pi^n r_\infty^{n(n+1)}}{n!}\det(I-P_nBB^*P_n)_{\ell^2(\Z_+)}.
\end{align*}
This completes the proof of Proposition~\ref{Prop:PartDet}.

We turn to the proof of \eqref{eq:free-energy2}. Note that if $\lambda >0$, then by a change of variable,
 \begin{align}\label{eq:scaling-Z*}
     Z_n^*(\lambda D^*) = \lambda^{-n(n+1)}Z_n^*(D^*)
         \end{align}
         and $r_0(\lambda D) = \lambda r_0(D)$. It is therefore enough to prove \eqref{eq:free-energy2} under the assumption $r_0(D) = 1$, which we assumed in our definition of Faber polynomials attached to the interior map $f$.

Set $w(z) = z^{-2}$. The integrand in $Z_n^*$ satisfies
  \begin{align*}
  e^{-(n+1) \sum_{j=1}^n\log \, |z_j|^2} \prod_{1 \le j < k \le n}|z_j-z_k|^2 & = \prod_{j = 1}^n \frac{1}{|z_j|^{2(n+1)}} \prod_{1 \le j < k \le n}|z_j-z_k|^2 \\
  & = \prod_{j=1}^n|w(z_j)|^2 \prod_{1 \le j < k \le n}|z_j^{-1}-z_k^{-1}|^2.
  \end{align*}
Therefore, using the formula for Vandermonde determinants and Andr\'eief's identity,
\begin{align*}
Z_n^*(D^*) & =  \frac{1}{n!}\int_{(D^*)^n} \det(w(z_k)z_k^{1-\ell})_{1\le k,\ell \le n} \det(w(\overline{z_k})  \overline{z_k}^{1-\ell})_{1\le k,\ell \le n}d^n \lambda  \\&  =  \det \left( \int_{D^*} z^{-k-1} \overline{z}^{-\ell-1} d^2z \right)_{1 \le k, \ell \le n}.
\end{align*}
On the other hand,  $Q_k$ is the $k$:th Faber polynomial attached to $f$  (see Section~\ref{sect:prel}), so we see that
\[
 \beta_k(z) = \frac{Q'_k(z)}{\sqrt{\pi k}}\]
 is a polynomial of degree $k+1$ in $1/z$ with leading coefficient $\sqrt{k/\pi}$. (We have assumed that $r_0(D) = 1$.) So row and column operations give
  \begin{align*}
  \det \left( \int_{D^*} z^{-k-1} \overline{z}^{-\ell-1} d^2z \right)_{1 \le k, \ell \le n} & = \frac{\pi^n}{n!}  \det \left( \int_{D^*} \sqrt{\frac{k}{\pi}}z^{-k-1} \sqrt{\frac{\ell}{\pi}} \overline{z}^{-\ell-1} d^2z \right)_{1 \le k, \ell \le n} \\
  & = \frac{\pi^n}{n!} \det \left(\int_{D^*} \beta_k(z) \overline{\beta_\ell(z)} d^2z\right)_{1 \le j,k \le n}.
  \end{align*}
Recalling that
  \[
  \beta_k( g(z)) g'(z) =  \sum_{j=1}^\infty  b_{j,-k} f_j(z)
   \]
we can now argue as before, by changing variables to integrate over $\D^*$. Using the definition of  $B_2$ and Lemma~\ref{grunsky-identities}, in particular \eqref{B1B2}, it follows that
\[
Z_n^*(D^*) = \frac{\pi^n}{n!}  \det((B_2B_2^*)_{k\ell})_{1 \le k, \ell \le n} = \frac{\pi^n}{n!}  \det(I - P_nB_1B_1^*P_n)_{\ell^2(\Z_+)}.
\]
The proof is complete.

    \end{proof}
\begin{rem}
    Set $\iota(z) = 1/z$. Using $|w(z)|^2 = 1/|z|^4=|\iota'(z)|^2$, the proof of Proposition~\ref{prop:partdet2} shows that $Z_n^*(D^*)=Z_n(\iota(D^*))$, so in particular $Z_n^*(\D^*) = n!/\pi^n$. We immediately obtain the more symmetric formula
    \begin{align}\label{symmetricformula}
        \log \frac{Z_n(D)Z_n^*(D^*)}{Z_n(\D)Z_n^*(\D^*)} = n(n+1)\log \frac{r_\infty}{r_0} + \log \det(I-P_nB_4B_4^*P_n)_{\ell^2(\Z_+)} + \log \det(I-P_nB_1B_1^*P_n)_{\ell^2(\Z_+)}.
    \end{align}
    If $\eta = \partial D$ is a Weil-Petersson quasicircle, then
    \[
    \lim_{n \to \infty} \log \det(I-P_nB_1B_1^*P_n)_{\ell^2(\Z_+)} =    \lim_{n \to \infty} \log \det(I-P_nB_4B_4^*P_n)_{\ell^2(\Z_+)} = - \frac{1}{12}I^L(\eta),
    \]
    but, in general, for finite $n$ 
    \[
    \log \det(I-P_nB_1B_1^*P_n)_{\ell^2(\Z_+)} \neq \log \det(I-P_nB_4B_4^*P_n)_{\ell^2(\Z_+)}.
    \]
     \end{rem}
\begin{rem}
    Note that Proposition~\ref{prop:partdet2} implies that if $D$ has capacity $1$, then $n \mapsto \log Z_n(D)/Z_n(\D)$ is monotone increasing. This is also true for the Coulomb gas on $\partial D$, see \cite{Jo}.
    \end{rem}
\begin{rem}
The Coulomb gas point process on $D$ for $n <\infty$ is the probability measure on $D^n$ whose density with respect to product area measure $d^n \lambda = d^n \lambda(z_1,\ldots, z_n)=\prod_{k=1}^n d^2 z_k$ restricted to $D^n$ is
    \[
    \rho_n(z_1,\ldots, z_n) = \frac{1}{Z_n(D)} \prod_{1\le k < \ell \le n}|z_k -z_\ell|^2 = \frac{1}{Z_n(D)} \det(z_k^{\ell-1})_{1 \le k,\ell \le n}\det(\overline{z_k}^{\ell-1})_{1 \le k,\ell \le n}.
    \]
This is a determinantal point process. With the help of Proposition~\ref{prop:partdet2}, we can express its correlation kernel using the Faber polynomials. This allows, e.g., to easily identify the limit as $n \to \infty$ by transplanting to $\D$. We will study this point process in a forthcoming paper. 
\end{rem}

\section{Proofs of the main theorems}\label{Sec:propthm}
In this section we prove Theorem~\ref{Thm:CornerAs}, Theorem~\ref{Thm:LoewAs}, and Theorem \ref{Thm:PommEn} assuming certain basic propositions whose proofs we defer to Section~\ref{sec:proofsbasic} and Section~\ref{Sec:lemmas}. 

 Given a Jordan domain $D$, we always let $B=B_4$ be the associated Grunsky operator on $\ell^2(\Z_+)$ with elements given by \eqref{bkl}. Recall that we write $z_p = g^{-1}(w_p)$ for the preimages of the corners on $\T$ and that   
\[
K_p(u,v)=-\frac{\gamma_p\sin\pi\gamma_p}{\pi\sqrt{uv}}\frac 1{(u/v)^{\gamma_p}+(v/u)^{\gamma_p}-2\cos\pi\gamma_p},
\]
 for $u,v>0$. We consider $K_p(k,\ell)$ as the matix of an operator on $\ell^2(\Z_+)$.

 The basis for the proof of Theorem~\ref{Thm:CornerAs} is the expansion of $\log\det(I-P_nBB^*P_N)_{\ell^2(\Z_+)}$ in traces. The next proposition relates these traces to traces involving $K_p^2$.
\begin{proposition}\label{Prop:TrComp1}
Suppose $D\in\mathcal{D}_m$. Let $K_p$ be defined by \eqref{Kpkl}. Then there is a constant $C$ such that
\begin{equation}\label{TrComp1}
\bigl|\Tr(P_nBB^*P_n)^i-\sum_{p=1}^m\Tr(P_nK_p^2P_n)^i\bigr|\le C^i,
\end{equation}
for all $i\ge 1$.
\end{proposition}

The proof will be given in Section~\ref{Subsec:Basicprop}.

Write
\begin{equation}\label{Hpt}
H_p(t)=\frac{\gamma_p\sin\pi\gamma_p}{2\pi}\frac 1{\cos\pi\gamma_p-\cosh\gamma_pt},
\end{equation}
for $1\le p\le m$ and $t>0$. 
We see that,
\begin{equation}\label{KpHp}
K_p(k,\ell)=\frac 1{\sqrt{k\ell}}H_p(\log\frac k{\ell}),
\end{equation}
which shows that $K_p$ is a so-called \emph{Hardy kernel}, see \cite{Pu}. For a function $f$ on $\R$ we define the Fourier transform by
\begin{equation}\label{FourTr}
\hat{f}(\xi)=\int_\R e^{- \I \xi x}f(x)\,dx.
\end{equation}
A computation using the residue theorem gives
\begin{equation}\label{HpFT}
\hat{H}_p(\xi)=\frac{\sinh(1-1/\gamma_p)\pi\xi}{\sinh\pi\xi/\gamma_p}.
\end{equation}
By \eqref{KpHp},
\[
\Tr(P_nK_pP_n) = \left( \frac 1{2\pi}\int_\R \hat{H}_p(\xi)\,d\xi \right) \left( \sum_{\ell=1}^n \frac{1}{\ell} \right),
\]
and a consequence of a more general result in \cite{Pu} is that for each $i\ge 1$
\begin{equation*}
\lim_{n\to\infty}\frac 1{\log n}\Tr(P_nK_pP_n)^i=\frac 1{2\pi}\int_\R \hat{H}_p(\xi)^i\,d\xi.
\end{equation*}
To control the traces involving $K_p^2$ we will need the following stronger version of this result.

\begin{proposition}\label{Prop:TrComp2}
Suppose $D\in\mathcal{D}_m$. There is a constant $C$ such that 
\begin{equation}\label{TrComp2}
\left|\Tr(P_nK_p^2P_n)^i-\left(\frac 1{2\pi}\int_\R \hat{H}_p(\xi)^{2i}\,d\xi\right)\left(\sum_{\ell=1}^n\frac 1\ell\right)\right|\le C^i,
\end{equation}
for all $i,n\ge 1$.
\end{proposition}
The proof is postponed to Section \ref{Subsec:Basicprop}. 
The fact that not one or a few traces dominate the others, but that all traces are of size $\log n$, is an important reason why the proof of our main result is rather complicated.

If $D\in\mathcal{D}_m$ then its boundary is a quasicircle (e.g., by Ahlfors' three-point condition) so we know that the strong Grunsky inequality \eqref{Grunskyin} holds.

We are now in position to prove Theorem \ref{Thm:CornerAs} if we also use the following lemma.
\begin{lemma}\label{Lem:Finalintegral}
We have the integral formula
\begin{equation}\label{Finalintegral}
-\frac 1{2\pi^2}\int_\R\log\left(1-\frac{\sinh^2\beta x}{\sinh^2x}\right)\,dx=\frac{\beta^2}{6(1-\beta^2)},
\end{equation}
for any $\beta\in (-1,1)$.
\end{lemma}
The proof is given in Section \ref{Sec:lemmas}.
\begin{proof}[Proof of Theorem \ref{Thm:CornerAs}]
Write $C_n=P_nBB^*P_n$. By \eqref{PartDetId2} it is enough to study the asymptotics of  $(\log n)^{-1}\log \det(I-C_n)$. Let $\lambda_{j,n}$, $1\le j \le n$, be the eigenvalues of $C_n$. Then, since $\partial D$ is a quasicircle, by the Grunsky inequality \eqref{Grunskyin}, we have the 
estimate $0\le\lambda_{j,n}\le\kappa^2<1$. Thus, for $|z|<1/\kappa^2$, 
\begin{equation*}
\log\det(I-zC_n)=-\sum_{i=1}^\infty\frac{z^i}i\Tr C_n^i,
\end{equation*}
and for $i \ge 1$
\begin{equation}\label{Trest}
0\le\Tr C_n^i=\sum_{j=1}^n\lambda_{j,n}^i\le \kappa^{2(i-1)}\sum_{j=1}^n\lambda_{j,n}=\kappa^{2(i-1)}\Tr C_n.
\end{equation}
Define the function 
\begin{equation*}
L_n(z)=\frac 1{\log n}\log\det(I-zC_n),
\end{equation*}
which is analytic for $|z|<1/\kappa^2$. 

Combining \eqref{TrComp1} and \eqref{TrComp2} we see that there is a constant $c_0$ so that
\begin{equation*}
\left|\Tr C_n^i-(\sum_{\ell=1}^n\frac 1{\ell})\sum_{p=1}^m\left(\frac 1{2\pi}\int_\R \hat{H}_p(\xi)^{2i}d\xi\right)\right|\le c_0^i.
\end{equation*}
The formula \eqref{HpFT} and a change of variables gives
\begin{equation}\label{TrComp3}
\left|\Tr C_n^i-(\sum_{\ell=1}^n\frac 1{\ell})\sum_{p=1}^m\frac {\gamma_p}{2\pi^2}\int_\R \left(\frac{\sinh^2(1-\gamma_p)x}{\sinh^2x}\right)^idx\right|\le c_0^i,
\end{equation}
for all $i,n\ge 1$. 
Note that
\begin{equation*}
|L_n(z)|\le\frac 1{\log n}\sum_{i=1}^\infty\frac{|z|^i}i\kappa^{2(i-1)}\Tr C_n\le\frac C{1-\kappa^2|z|},
\end{equation*}
by \eqref{TrComp3} with $i=1$, \eqref{Trest}, and the fact that $\frac 1{\log n}\sum_{\ell=1}^n\frac 1{\ell}$ is bounded. Hence, $\{L_n(z)\}_{n\ge 1}$ is a normal family in $|z|<1/\kappa^2$, a set which includes $\bar \D$. It follows from \eqref{TrComp3} that if $|z|<1/c_0$, then by Vitali's theorem \cite[p.\ 168]{Titch},
\begin{equation*}
\lim_{n\to\infty}L_n(z)=\sum_{p=1}^m\left(-\sum_{i=1}^\infty\frac{z^i}i\frac {\gamma_p}{2\pi^2}\int_\R \left(\frac{\sinh^2(1-\gamma_p)x}{\sinh^2x}\right)^idx\right).
\end{equation*}
Here we used the fact that $\frac 1{\log n}\sum_{\ell=1}^n\frac 1{\ell}$ converges to $1$ as $n\to \infty$. Consequently, by Lemma \ref{Lem:Finalintegral},
\begin{align*}
&\lim_{n\to\infty}-\frac 1{\log n}\log\det(I-C_n)=\sum_{p=1}^m-\frac{\gamma_p}{2\pi^2}\int_\R\log\left(1-\frac{\sinh^2(1-\gamma_p)x}{\sinh^2x}\right)dx\\
&=\frac 16\sum_{p=1}^m\frac{\gamma_p(1-\gamma_p)^2}{2\gamma_p-\gamma_p^2}=\frac 16\sum_{p=1}^m\left(\alpha_p+\frac 1{\alpha_p}-2\right),
\end{align*}
since $\gamma_p=2-\alpha_p$.
This completes the proof.
\end{proof}

The proof of Theorem \ref{Thm:LoewAs} is based on the following proposition.
\begin{proposition}\label{Prop:TrComp4}
For $r > 1$, let $B_r$ be the Grunsky operator for $\eta_r = g(r \eta)/r$, that is, 
\begin{equation}\label{Br}
(B_r)_{k\ell}=\left(\frac 1{r^{k+\ell}}b_{k\ell}\right).
\end{equation}
Then there is a constant $C$ such that
\begin{equation}\label{TrComp4}
\bigl|\Tr(B_rB_r^*)^i-\Tr(P_nBB^*P_n)^i\bigr|\le C^i,
\end{equation}
for all $r>1$, $i\ge 1$, if $n=\lfloor \frac 1{r-1}\rfloor $. 
\end{proposition}

The proof of the proposition will be given in Section \ref{Subsec:Basicprop}.  
\begin{proof}[Proof of Theorem~\ref{Thm:LoewAs}]
The proof follows the same line of argument as the proof of Theorem \ref{Thm:CornerAs} so we will only outline the main steps. It follows from \eqref{ILDet}, and the fact that $\eta_r$ has Grunsky matrix $B_r$ that in order to prove the theorem we have to verify that,
\begin{equation}\label{Etarform}
   \lim_{r\to 1+} -\frac 6{\log\frac 1{r-1}}\log\det(I-B_rB_r^*)=\sum_{p=1}^m\left(\alpha_p+\frac 1{\alpha_p}-2\right).
\end{equation}
It follows from the strengthened Grunsky inequality that we have the power series expansion
\begin{equation*}
    \log\det(I-zB_rB_r^*)=-\sum_{i=1}^\infty\frac{z^i}{i}\Tr(B_rB_r^*)^i
\end{equation*}
for $|z|<1/\kappa^2$, where $\kappa <1$. As in the proof of Theorem \ref{Thm:CornerAs}, we have the estimate
\begin{equation*}
    0\le \Tr(B_rB_r^*)^i\le \kappa^{2(i-1)}\Tr(B_rB_r^*),
\end{equation*}
and we define the function
\begin{equation*}
    L_r(z)=\frac 1{\log\frac 1{r-1}}\log\det(I-zB_rB_r^*).
\end{equation*}
By \eqref{TrComp1}, \eqref{TrComp2} and \eqref{TrComp4}, there is a constant $c_0$ such that
\begin{equation*}
\left|\Tr (B_rB_r^*)^i-(\sum_{\ell=1}^n\frac 1{\ell})\sum_{p=1}^m\left(\frac 1{2\pi}\int_\R \hat{H}_p(\xi)^{2i}d\xi\right)\right|\le c_0^i.
\end{equation*}
Recall that $n=\lfloor \frac 1{r-1}\rfloor $. The same argument as in the proof of Theorem \ref{Thm:CornerAs} now gives
\begin{equation*}
\lim_{r\to 1+}L_r(z)=\sum_{p=1}^m\left(-\sum_{i=1}^\infty\frac{z^i}i\frac {\gamma_p}{2\pi^2}\int_\R \left(\frac{\sinh^2(1-\gamma_p)x}{\sinh^2x}\right)^idx\right),
\end{equation*}
and \eqref{Etarform}. This completes the proof of the theorem.
\end{proof}

We finally come to Theorem \ref{Thm:PommEn}. Before giving the proof we state two lemmas, both proved in Section~\ref{Subsec:ProofPommEn}. Let $\mathcal{K}_r$ and ${\bf d}_r$ be the objects given by \eqref{Kop} and \eqref{dvec} with $B$ replaced by $B_r$.
Define for $k\ge 1$,
\begin{equation}\label{xikr}
 (\xi_r)_k=\frac 1{r^k\sqrt{k}}\sum_{p=1}^mz_p^k(\gamma_p-1).
\end{equation}
\begin{lemma}\label{Lem:dapprox}
    There is a constant $C$ such that
    \begin{equation}\label{drxir}
        |({\bf d}_r)_k-(\xi_r)_k|\le\frac{C}{r^kk^{(1+\rho)/2}},
    \end{equation}
    for all $k\ge 1$ and $r>1$.
\end{lemma}
\begin{lemma}\label{Lem:Brdas}
\begin{equation}\label{Brdas1}
   \left|{\bf d}_r^*(B_rB_r^*)^i{\bf d}_r-\sum_{p=1}^m(\gamma_p-1)^{2i+2}\log n\right|\le C^i,  
\end{equation}
and
\begin{equation}\label{Brdas2}
   \left|{\bf d}_r^*(B_rB_r^*)^iB_r{\bf d}_r-\sum_{p=1}^m(\gamma_p-1)^{2i+3}\log n\right|\le C^i,  
\end{equation}
for all $i\ge 1$. Here $n=[\frac 1{r-1}]$.
\end{lemma}
\begin{proof}[Proof of Theorem \ref{Thm:PommEn}]
For $r>1$ we define the function
 \begin{equation*}
Q_r(w)=\left(\log(\frac 1{r-1})\right)^{-1}\begin{pmatrix}
     \re {\bf d}_r \\ \im {\bf d}_r
 \end{pmatrix} ^t(I+w \mathcal{K}_r)^{-1} 
\begin{pmatrix}
     \re {\bf d}_r \\ \im {\bf d}_r
 \end{pmatrix}.
 \end{equation*}  
By \eqref{Kdiag}, we can write
\begin{equation*}
\mathcal{K}_r=T_r\begin{pmatrix}
    \Lambda_r & 0 \\  0 & -\Lambda_r
    \end{pmatrix}T_r^t,
\end{equation*}
where $\Lambda_r$ is diagonal with the eigenvalues for $\mathcal{K}_r$ as elements $(\lambda_r)_k$, $k\ge 1$, and $T_r$ is orthogonal. Also, by the Grunsky inequality \eqref{Grunskyin}, $|(\lambda_r)_k|\le\kappa<1$, for all $r\ge 1$ and $k\ge 1$. If we write
\begin{equation*}
\begin{pmatrix}
    x_r \\ y_r
\end{pmatrix}=T_r^t\begin{pmatrix}
     \re {\bf d}_r \\ \im {\bf d}_r
 \end{pmatrix},
 \end{equation*}
then
\begin{equation*}
    Q_r(w)=\left(\log(\frac 1{r-1})\right)^{-1}\sum_{i=0}^\infty (-w)^ix_r^t\Lambda_r^ix_r+w^iy_r^t\Lambda^iy_r,
\end{equation*}
and thus, since $T_r$ is orthogonal
\begin{equation*}
    |Q_r(w)|\le \left(\log(\frac 1{r-1})\right)^{-1}\left(\sum_{i=0}^\infty |\kappa w|^i\right)(x_r^tx_r+y_r^ty_r)=\frac{|{\bf d}_r|^2}{(1-\kappa|w|)\log\frac 1{r-1}}\le \frac C{1-\kappa|w|}.
\end{equation*}
Here we used the estimate
\begin{equation*}
    |{\bf d}_r|^2\le C\sum_{k=1}^\infty \frac 1{r^k k}= C\log\frac r{r-1},
\end{equation*}
which follows from \eqref{xikr} and \eqref{drxir}. Hence, $\{Q_r(w)\}$ with $1<r<2$ is a normal family in $|w|<1/\kappa$ which includes the closed unit disc. We can write
\begin{align*}
    Q_r(w)&=\sum_{p=1}^m\left(\log(\frac 1{r-1})\right)^{-1}\left[\sum_{i=0}^\infty w^{2i}\left(\Big({\bf d}_r^*(B_rB_r^*)^i{\bf d}_r-(\gamma_p-1)^{2i+2}\log n\Big)\right.\right.\\
    &\left.\left.-w\Big({\bf d}_r^*(B_rB_r^*)^iB_r\bar{{\bf d}}_r-(\gamma_p-1)^{2i+3}\log n\Big)\right)\right]
+\frac{\log n}{\log\frac 1{r-1}}\sum_{p=1}^m w^{2i}\left[(\gamma_p-1)^{2i+2}-(\gamma_p-1)^{2i+3}\right],
\end{align*}
where $n=[\frac 1{r-1}]$. It follows from Lemma \ref{Lem:Brdas} that for small $|w|$,
\begin{equation}
    \lim_{r\to 1+}Q_r(w)=\sum_{p=1}^m\sum_{i=0}^\infty w^{2i}\Big[(\gamma_p-1)^{2i+2}-w(\gamma_p-1)^{2i+3}\Big].
\end{equation}
Thus,
\begin{equation*}
    \lim_{r\to 1+}\left(\log(\frac 1{r-1})\right)^{-1}I^F(\eta_r)=2 \lim_{r\to 1+}Q_r(1)=2\sum_{p=1}^m\sum_{i=0}^\infty\Big[(\gamma_p-1)^{2i+2}-(\gamma_p-1)^{2i+3}\Big]=2 \sum_{p=1}^m\frac{(\gamma_p-1)^2}{\gamma_p},
\end{equation*}
and the theorem is proved.
\end{proof}

As can be seen from the proof of Theorem \ref{Thm:CornerAs} it is essential to have estimates of the traces with an error of the form $C^i$. To prove Proposition \ref{Prop:TrComp1} we use the asymptotic result for the Grunsky coefficients in Theorem \ref{Thm:bklAs} to replace $B$ by $K$, defined by \eqref{Kkl}, in the traces, and then use a cancellation to get to the expression involving just $K_p$ and a sum over the corners. The proof of Proposition \ref{Prop:TrComp2} consists of successive approximations extending sums, replacing sums by integrals and extending domains of integration. The proof of Theorem \ref{Thm:PommEn} is based on Lemma \ref{Lem:KtoB} so we are no longer estimating traces. This requires some modifications which leads us to use \eqref{TrComp1:2}, Lemma \ref{Lem:dapprox} and Lemma \ref{Lem:Brdas} instead. The proofs of these results have many similarities with the proofs of Proposition \ref{Prop:TrComp1} and Proposition \ref{Prop:TrComp2}.

\section{Asymptotics of Grunsky coefficients and corner contributions}\label{sec:proofsbasic}
In this section we will prove Propositions \ref{Prop:TrComp1}, \ref{Prop:TrComp2} and \ref{Prop:TrComp4}. We will also prove Theorem
\ref{Thm:bklAs}, which is the basic asymptotic result. The proofs of some technical lemmas will be given in Section~\ref{Sec:lemmas}.

\subsection{Asymptotics of the Grunsky coefficients}\label{Subsec:Asym}
The starting point for the asymptotic analysis is a double contour integral formula for the Grunsky coefficients given in the next lemma.
\begin{lemma}\label{Lem:Intfor} 
Let $g:\D^* \to  D^*$ be the exterior conformal map for $D$ and $a_{k\ell}$ the Grunsky coefficients for $D$ as in \eqref{defGrunsky}. Then,
\begin{equation}\label{aklformula}
(k+\ell)a_{k\ell}=\frac 1{(2\pi\I)^2}\int_{|z|=r}\int_{|w|=r} z^{k-1}w^{\ell-1}\Psi(z,w)dz dw,
\end{equation}
for $k,\ell\ge 1$ and $r>1$, where
\begin{equation}\label{Psi}
\Psi(z,w)=\frac{zg'(z)-wg'(w)}{g(z)-g(w)}.
\end{equation}
\end{lemma}

\begin{proof}
Differentiation of the formula \eqref{defGrunsky} with respect to $z$ and $w$ gives
\begin{equation*}
\frac{g'(z)}{g(z)-g(w)}-\frac 1{z-w}=\sum_{k,\ell=1}^\infty ka_{k\ell}z^{-k-1}w^{-\ell}
\end{equation*}
and
\begin{equation*}
\frac{-g'(w)}{g(z)-g(w)}+\frac 1{z-w}=\sum_{k,\ell=1}^\infty \ell a_{k\ell}z^{-k}w^{-\ell-1}.
\end{equation*}
Multiply these identities with $z$ and $w$ respectively and add them. This gives
\begin{equation*}
\Psi(z,w)=1+\sum_{k,\ell=1}^\infty (k+\ell) a_{k\ell}z^{-k}w^{-\ell},
\end{equation*}
from which \eqref{aklformula} follows.
\end{proof}

In order to analyze the asymptotics of the right side of \eqref{aklformula} for $k,\ell$ large, we want to deform the contours of integration. Recall that $D$ is bounded by $m$ analytic 
arcs which form angles $\pi\alpha_p$ at points $w_p=g(z_p)$ on the boundary. Since the arcs are analytic, $g$ can be extended analytically to a subdomain of $\D$, by the Schwarz reflection principle, except at the
preimages $z_p$ of the corners. 

Let $\delta>0$ be a parameter, it will be chosen small enough  below. Consider angles and lengths as in Fig.
\ref{Fig:Geo}. Here, $P$ is the point $-1+(1-\delta)e^{-\I\phi}$, $Q$ is the point $-1+e^{-\I\phi}$, and we choose $\phi=\phi(\delta)$ so that $(-1)PO$ is a right angle, which gives

 \begin{figure}
 \begin{center}
 \includegraphics[height=4cm]{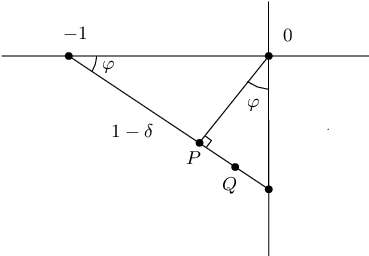}
 \caption{Angles and lengths.}
 \label{Fig:Geo}
 \end{center}
 \end{figure}

\begin{equation}\label{Defphi}
\sin(\phi)=|OP|=\sqrt{2\delta-\delta^2}
\end{equation}
Define the angles $\theta_p\in(-\pi,\pi]$ by
\begin{equation}\label{zp}
z_p=e^{\I\theta_p},
\end{equation}
and let
\begin{equation*}
z_p(\pm 1)=(1-\delta)e^{\I(\theta_p\pm \phi)}
\end{equation*}
be points close to $z_p$ inside $\D$. These points lie on the circle centered at the origin of radius $1-\delta$. Let $L_{p,\delta}(1)$ be the oriented line segment from $z_p$ to 
$z_p(1)$, and $L_{p,\delta}(-1)$ be the oriented line segment from $z_p(-1)$ to $z_p$. Furthermore, let $C_{p,\delta}$ be the circular arc on $(1-\delta)\T$ from $z_{p-1}(1)$ to $z_p(-1)$
with the convention that $C_{m+1,\delta}=C_{1,\delta}$, see Fig. \ref{Fig:Local}.

\begin{figure}
\begin{center}
\includegraphics[height=8cm]{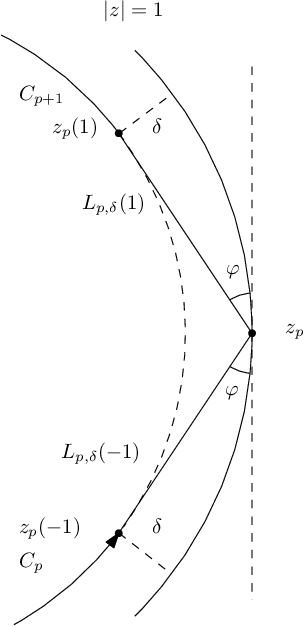}
\caption{Definition of $\Gamma_\delta.$}
\label{Fig:Local}
\end{center}
\end{figure}

Write $L_{p,\delta}=L_{p,\delta}(-1)+L_{p,\delta}(1)$ and let $\Gamma_\delta$ be the closed curve given by
\begin{equation}\label{Gammadelta}
\Gamma_\delta=\sum_{p=1}^mC_{p,\delta}+L_{p,\delta}.
\end{equation}

\begin{lemma}\label{lem:g-extension}
    If $\delta$ is chosen sufficiently small, then the exterior conformal map $g$ can be extended to a univalent function in a neighborhood of $\Gamma_\delta$ except at the points $z_p$.
\end{lemma}
\begin{proof}
Let $z_p, z_{p+1}$ be two consecutive points on $\mathbb{T}$ mapped to corners $w_p=g(z_p), w_{p+1}=g(z_{p+1})$. Write $I_{p, p+1}$ for the circular arc connecting $z_p$ and $z_{p+1}$ such that $J_{p, p+1}=g(I_{p, p+1})$ is an analytic arc. By definition, there exists a neighborhood $N$ of $I_{p, p+1}$ and a conformal map $\varphi:N \to \varphi(N)$ such that $\varphi(I_{p, p+1}) = J_{p, p+1}$. Note that $z_p, z_{p+1}$ are interior points of $N$. Let $S_\pm$ be a radial segment in $\mathbb{D}^*$ ending in $z_p$ and $z_{p+1}$ respectively. Let $r>1$ be close to $1$ and write $A$ for the simply connected domain bounded by the four arcs $I_{p,p+1}$, $S_{\pm}$ and $\{r e^{i\theta}: \theta \in [\theta_p, \theta_{p+1}]\}$, where $e^{i\theta_p} = z_p, e^{i\theta_{p+1}} = z_{p+1}$. We claim that if $r$ is chosen sufficiently close to $1$, then $h = \varphi^{-1} \circ g$ is a conformal map of $A$ onto a simply connected domain $B \subset \mathbb{D}^*$ and $h(I_{p, p+1})=I_{p, p+1}$. Indeed, this is clear if both exterior corner angles $\gamma_{p}, \gamma_{p+1} <1$. Suppose that $\gamma_p > 1$. Let $I_- \subset N \cap \mathbb{T}$ be a subarc of $I_{{p-1},p}$ ending in $z_p$. Then set $\tilde I_- = g^{-1} \circ \varphi (I_-)$. This is an arc in $\mathbb{D}^*$ ending in $z_p$ (possibly after taking a further subarc of $I_-$). Moreover, by \cite[Theorem~3.9]{PommerenkeBoundary}, $\tilde I_-$ makes an angle $\pi/2 < \nu_- < \pi$ with $I_{p,p+1}$. We argue the same way if $\gamma_{p+1} >1$. The claim follows. By Schwarz reflection across $\mathbb{T}$ (see \cite[Section~1.2]{PommerenkeBoundary}), $h$ extends to a conformal map of (the interior of) $A \cup A^* \cup I_{p, p+1}$ onto (the interior of) $B \cup  B^* \cup I_{p, p+1}$, where $A^*, B^*$ are the reflected domains. It follows that $\varphi \circ h$ defines a univalent extension of $g$ on $A \cup A^* \cup I_{p, p+1}$. Repeating the same argument for all $z_p$, we see that $\Gamma_\delta \smallsetminus \{z_p\}_{p=1}^m$ is contained in the union of the reflected domains, provided $\delta >0$ is chosen sufficiently small.
\end{proof}

By Lemma \ref{Lem:Intfor}, Lemma \ref{lem:g-extension}, and Cauchy's theorem, we obtain the formula
\begin{equation}\label{aklformula2}
(k+\ell)a_{k\ell}=\frac 1{(2\pi\I)^2}\int_{\Gamma_\delta}dz\int_{\Gamma_\delta}dw z^{k-1}w^{\ell-1}\Psi(z,w).
\end{equation}
We split the integral in the right side of \eqref{aklformula2} into four parts. Let
\begin{align}\label{Ikl1}
I_1(k,\ell)&=\sum_{p=1}^m\frac 1{(2\pi\I)^2}\int_{L_{p,\delta}}dz\int_{L_{p, \delta}}dw z^{k-1}w^{\ell-1}\Psi(z,w),\\
\label{Ikl2}
I_2(k,\ell)&=\sum_{\substack{
   p,q=1\\p\neq q}}^m\frac 1{(2\pi\I)^2}\int_{L_{p,\delta}}dz\int_{L_{q,\delta}}dw z^{k-1}w^{\ell-1}\Psi(z,w),\\
\label{Ikl3}
I_3(k,\ell)&=\sum_{p,q=1}^m\frac 1{(2\pi\I)^2}\left(\int_{L_{p,\delta}}dz\int_{C_{q,\delta}}dw +\int_{C_{p, \delta}}dz\int_{L_{q, \delta}}dw\right) z^{k-1}w^{\ell-1}\Psi(z,w),\\
\label{Ikl4}
I_4(k,\ell)&=\sum_{p,q=1}^m\frac 1{(2\pi\I)^2}\int_{C_{p,\delta}}dz\int_{C_{q,\delta}}dw z^{k-1}w^{\ell-1}\Psi(z,w),
\end{align}
so that 
\begin{equation}\label{aklsplit}
(k+\ell)a_{k\ell}=\sum_{r=1}^4I_r(k,\ell).
\end{equation}
The main contribution for $k,\ell$ large comes from $I_1(k,\ell)$. 

In a neighborhood of $z_p$ containing $L_{p,\delta}$ it is convenient to standardize by changing coordinates
\begin{equation}\label{zu}
z=z_p(1+u).
\end{equation}
In the $u$-plane, we get the same local geometry at all the corners.

We define $z(\tau)$ and $L_\delta(\tau)$, $\tau=\pm1$, to be the inverse images of $z_p(\tau)$ and $L_{p,\delta}(\tau)$, $\tau=\pm1$, under the
map \eqref{zu}. Let $L_\delta=L_\delta(-1)+L_\delta(1)$. Set
\begin{equation}\label{lambda}
\lambda=e^{\I(\phi+\pi/2)}.
\end{equation}
Then $\tau L_{\delta}(\tau)$, $\tau=\pm1$, is parametrized by
\begin{equation}\label{Lpar}
u=t\lambda^\tau, \quad 0\le t\le b:=\sqrt{2\delta-\delta^2}.
\end{equation}

We want to give a local approximation to $\Psi(z_p(1+u),z_p(1+v))$ for $u,v\in L_\delta$. Define
\begin{equation}\label{Gp}
G_p(u,v)=\gamma_p\frac{u^{\gamma_p-1}-v^{\gamma_p-1}}{u^{\gamma_p}-v^{\gamma_p}},
\end{equation}
where as always in this paper the powers are defined using the principal logarithm so that we have branch cuts along the negative $u$- and $v$-axes in \eqref{Gp}. Write
\begin{equation}\label{Rp}
R_p(u,v)=\Psi(z_p(1+u),z_p(1+v))-G_p(u,v).
\end{equation}
The control of $R_p$ is based on the following lemma on the regularity of $g$ close to the corners.

\begin{lemma}\label{Lem:regularity}
Let $L_\delta=L_\delta(-1)+L_\delta(1)$, where $L_\delta(\pm 1)$ are given by the parametrization \eqref{Lpar}. If we choose $\delta$ small enough, there is a constant
$A\neq 0$ such that if we define
\begin{align}
r_1(u)&=z_pg'(z_p(1+u))-\gamma_pAu^{\gamma_p-1},\label{r1u}\\
r_2(u)&=z_p^2g''(z_p(1+u))-\gamma_p(\gamma_p-1)Au^{\gamma_p-2},\label{r2u}
\end{align}
then there is a constant $C$ so that we have the estimates
\begin{equation}\label{r12est}
|r_1(u)|\le C|u|^{\lambda_p-1},\quad  |r_2(u)|\le C|u|^{\lambda_p-2},
\end{equation}
for all $u\in L_\delta$, where
\begin{equation}\label{lambdap}
\lambda_p=\min(2\gamma_p,\gamma_p+1).
\end{equation}
\end{lemma}

\begin{proof}
It follows from Theorem 1 in 
\cite{Leh} that if $\gamma_p$ is irrational, then close to a corner the function $g$ has an expansion
\begin{equation*}
g(z_p(1+u))-g(z_p)=A_{01}u^{\gamma_p}+A_{02}u^{2\gamma_p}+A_{11}u^{1+\gamma_p}+...
\end{equation*}
for $u\in L_{\delta}$ where $A_{01}\neq 0$ provided that $\delta$ is small enough. Higher order terms are smaller relative to the terms given for small $u$. If
$\gamma_p< 1$, then $|u^{2\gamma_p}|$ is larger than $|u^{1+\gamma_p}|$ for small $u$, so if we write
\begin{equation*}
g(z_p(1+u))-g(z_p)=A_{01}u^{\gamma_p}+r_0(u),
\end{equation*}
then $|r_0(u)/u^{2\gamma_p}|\le C$ when $u$ is small. If $1\le \gamma_p<2$, then $|u^{1+\gamma_p}|$ is larger than $|u^{2\gamma_p}|$ for small $u$, so we have instead
$|r_0(u)/u^{\gamma_p+1}|\le C$. Thus, provided that $\delta$ is sufficiently small, we have the bound $|r_0(u)/u^{\lambda_p}|\le C$. If $\gamma_p=p'/q'$ is rational, then according to 
\cite{Leh},
\begin{equation}\label{psirat}
g(z_p(1+u))-g(z_p)=A_{010}u^{\gamma_p}+A_{020}u^{2\gamma_p}+A_{110}u^{1+\gamma_p}+A_{111}u^{\gamma_p+1}\log u+...
\end{equation}
where $A_{010}$ and $A_{111}$ can be $\neq 0$ only if $1/p'\ge 1$, in which case $\gamma_p\le 1/2$. Thus $u^{2\gamma_p}$ is a larger term than
$u^{\gamma_p+1}\log u$ and we get the same conclusion about $r_0$. All other terms will be even smaller. According to \cite{Leh} we can differentiate the
expansion for $g(z_p(1+u))-g(z_p)$ and the same argument then proves the lemma.
\end{proof}

The next lemma gives estimates of $G_p$, and in particular $R_p$ based on the previous lemma.
We postpone the proof to Section \ref{Sec:lemmas}.

\begin{lemma}\label{Lem:GpRp}
Assume that $D\in\mathcal{D}_m$ and let $\rho$ be defined by \eqref{rho}. There is a constant $C$ such that
\begin{align}
|R_p(u,v)|&\le C(|u|^{\rho-1}+|v|^{\rho-1}),\label{Rpbound}\\
|G_p(u,v)|&\le 
\begin{cases}
&C\max(|u|,|v|)^{-1}, \quad\text{if $1<\gamma_p<2$}\\
&C\frac{\min(|u|,|v|)^{\gamma_p-1}}{\max(|u|,|v|)^{\gamma_p}}\quad \text{if $0<\gamma_p<1$},
\end{cases}\label{Gpbound}
\end{align}
for $1\le p\le m$ and $u,v\in L_\delta$, if $\delta$ is small enough.
\end{lemma}

The estimate \eqref{Rpbound} is the crucial estimate that enables us to show \eqref{rklEst}.
The relation \eqref{Rp} leads to a further split of $I_1(k,\ell)$. Define
\begin{align}
I(k,\ell)&=\sum_{p=1}^m\frac{z_p^{k+\ell}}{(2\pi\I)^2}\int_{L_\delta}du\int_{L_\delta}dv (1+u)^{k-1}(1+v)^{\ell-1}G_p(u,v), \label{Ikl}\\
I'(k,\ell)&=\sum_{p=1}^m\frac{z_p^{k+\ell}}{(2\pi\I)^2}\int_{L_\delta}du\int_{L_\delta}dv (1+u)^{k-1}(1+v)^{\ell-1}R_p(u,v) \label{Iklprime}.
\end{align}
It follows from \eqref{Ikl1}, \eqref{Rp} and a change of variables that
\begin{equation}\label{Ikl1split}
I_1(k,\ell)=I(k,\ell)+I'(k,\ell).
\end{equation}

\begin{lemma}\label{Lem:Iklprimeest}
There is a constant $C$ such that
\begin{equation}\label{Iklprimeest}
|I'(k,\ell)|\le C\left(\frac 1{k^\rho\ell}+\frac 1{k\ell^\rho}\right),
\end{equation}
for all $k,\ell\ge 1$.
\end{lemma}

\begin{proof}
It follows from \eqref{Iklprime} using the parametrization \eqref{Lpar} that
\begin{equation}\label{Iklprimefor}
I'(k,\ell)=\sum_{p=1}^m z_p^{k+\ell}\sum_{\tau_1,\tau_2=\pm 1}\frac{\tau_1\tau_2 \lambda^{\tau_1+\tau_2}}{(2\pi\I)^2}\int_0^b dt\int_0^b ds(1+t\lambda^{\tau_1})^{k-1}
(1+s\lambda^{\tau_2})^{\ell-1}R_p(t\lambda^{\tau_1},s\lambda^{\tau_1}).
\end{equation}
A computation using \eqref{Defphi} gives $|1+t\lambda^\tau|^2=1+t^2-2bt$ for $\tau=\pm 1$. Since $0\le t\le b$, we have the estimate $1+t^2-2bt\le 1-bt$. If we use this estimate and
\eqref{Rpbound} in \eqref{Iklprimefor}, we obtain the bound
\begin{align*}
|I'(k,\ell)|&\le C\int_0^b dt\int_0^b ds(1-bt)^{(k-1)/2}(1-bs)^{(\ell-1)/2}(t^{\rho-1}+s^{\rho-1})\\
&\le \frac{C}{k\ell}\int_0^{bk} dt\int_0^{b\ell} ds\left(1-\frac{bt}k\right)^{k/2}\left(1-\frac{bs}\ell\right)^{\ell/2}\left(\frac{t^{\rho-1}}{k^{\rho-1}}+\frac{s^{\rho-1}}{\ell^{\rho-1}}\right)\\
&\le \frac{C}{k\ell}\int_0^{\infty} dt\int_0^{\infty} ds e^{-bt/2-bs/2}\left(\frac{t^{\rho-1}}{k^{\rho-1}}+\frac{s^{\rho-1}}{\ell^{\rho-1}}\right)\le C\left(\frac 1{k^\rho\ell}+\frac 1{k\ell^\rho}\right),
\end{align*}
since $\rho,b>0$.
\end{proof}

We also want to bound $I_r(k,\ell)$, for $2\le r\le 4$.

\begin{lemma}\label{Lem:Iklnrbound}
There is a constant $C$ such that the following estimates hold
\begin{align}
|I_2(k,\ell)|&\le C\left(\frac 1{k^\rho\ell}+\frac 1{k\ell^\rho}\right),\label{Ikl2bound}\\
|I_3(k,\ell)|&\le C\left(\frac {(1-\delta)^\ell}{k^\rho}+\frac {(1-\delta)^k}{\ell^\rho}\right),\label{Ikl3bound}\\
|I_4(k,\ell)|&\le C(1-\delta)^{k+\ell}\label{Ikl4bound}.
\end{align}
The constant $C$ depends on the contour $\eta$ and the choice of $\delta>0$ but not on $k,\ell$.
\end{lemma}
\begin{proof}
Since $g$ is analytic in a neighborhood of $C_1+\dots+C_m$ and $|z|=1-\delta$ for $z\in C_p$, the bound \eqref{Ikl4bound} follows immediately from \eqref{Ikl4}.
Recall the definition of $I_3(k,\ell)$ in \eqref{Ikl3}. Split $L_{p,\delta}(\pm 1)$ into two halves $L_{p,\delta}^{(1)}(\pm 1)$ and $L_{p,\delta}^{(2)}(\pm 1)$, where
$L_{p,\delta}^{(1)}(\pm 1)$ are the halves closest to $z_p$. Then, 
\begin{align*}
&\sum_{p,q=1}^m\frac 1{(2\pi\I)^2}\int_{L_{p,\delta}}dz \int_{C_{q,\delta}}dwz^{k-1}w^{\ell-1}\Psi(z,w)\\
&=\sum_{p,q=1}^m\sum_{r=1}^2\frac 1{(2\pi\I)^2}\int_{L_{p,\delta}^{(r)}(-1)+L_{p,\delta}^{(r)}(1)}dz \int_{C_{q,\delta}}dwz^{k-1}w^{\ell-1}\Psi(z,w).
\end{align*}
There is a constant $d_0\in(0,1)$ so that if $z\in L_{p,\delta}^{(2)}(\pm 1)$, then $|z|\le 1-d_0\delta$. Since $\Psi(z,w)$ is bounded for 
$z\in L_{p,\delta}^{(2)}(1)+L_{p,\delta}^{(2)}(-1)$, $w\in C_q$, and $|w|=1-\delta$ if $w\in C_q$, we get the bound
\begin{equation*}
\left|\sum_{p,q=1}^m\sum_{r=1}^2\frac 1{(2\pi\I)^2}\int_{L_{p,\delta}^{(2)}(-1)+L_{p,\delta}^{(2)}(1)}dz \int_{C_{q,\delta}}dwz^{k-1}w^{\ell-1}\Psi(z,w)\right|\le
C(1-d_0\delta)^k(1-\delta)^\ell.
\end{equation*}
Our extended function $g$ is continuous and injective on the closures of $C_q$ and $L_{p,\delta}^{(1)}(\pm 1)$. Since $C_q$ and $L_{p,\delta}^{(1)}(\pm 1)$ are at a positive distance from each other, there is a constant $c_0>0$ such that if $z\in L_{p,\delta}^{(1)}(\pm 1)$ and $w\in C_q$, then
\begin{equation}\label{psidifflb}
|g(z)-g(w)|\ge c_0.
\end{equation}
Hence, by \eqref{Psi}, for $z\in L_{p,\delta}^{(1)}(\pm 1)$ and $w\in C_q$, we have the bound
\begin{equation*}
|\Psi(z,w)|\le\frac 1{c_0}(|zg'(z)|+|wg'(w)|)\le C(1+|zg'(z)|),
\end{equation*}
since $|g'(w)|$ is bounded on $C_q$. It follows from Lemma \ref{Lem:regularity} below that there is a constant $C$ such that
\begin{equation}\label{psiprimeest}
|g'(z_p(1+u))|\le C|u|^{\gamma_p-1},
\end{equation}
for $u\in L_{\delta}^{(1)}(\pm 1)$, the inverse image under \eqref{zu} of $L_{p,\delta}^{(1)}(\pm 1)$. Thus, using the same estimates as in the proof of 
Lemma \ref{Lem:Iklprimeest}, we see that
\begin{align*}
\left|\frac 1{(2\pi\I)^2}\int_{L_{p,\delta}^{(1)}(-1)+L_{p,\delta}^{(1)}(1)}dz \int_{C_{q,\delta}}dwz^{k-1}w^{\ell-1}\Psi(z,w)\right|&\le
C(1-\delta)^\ell\int_0^b(1-bt)^{k-1}t^{\gamma_p-1}\,dt\\
&\le C\frac{(1-\delta)^\ell}{k^{\gamma_p}}.
\end{align*}
Combining all the estimates proves \eqref{Ikl3bound}.

Consider now $I_2(k,\ell)$. For $z\in L_p$ and $w\in L_q$ with $p\neq q$, there is still a constant $c_0$ so that \eqref{psidifflb} holds. Using the estimate
\eqref{psiprimeest}, we see that
\begin{align*}
&\left|\frac 1{(2\pi\I)^2}\int_{L_{p,\delta}}dz\int_{L_{q,\delta}}dwz^{k-1}w^{\ell-1}\Psi(z,w)\right|\\
&\le C\int_{L_\delta}|du|\int_{L_\delta}|dv||1+u|^{k-1}|1+v|^{\ell-1}(|u|^{\gamma_p-1}+|v|^{\gamma_p-1})\le C\left(\frac 1{k^{\gamma_p}\ell}+
\frac 1{k\ell^{\gamma_p}}\right),
\end{align*}
by the same argument as in the proof of Lemma \ref{Lem:Iklprimeest}. The estimate \ref{Ikl2bound} follows.
\end{proof}

Combining \eqref{aklsplit}, \eqref{Ikl1split}, and the bounds in Lemma \ref{Lem:Iklprimeest}, and Lemma \ref{Lem:Iklnrbound}, we obtain the estimate
\begin{equation}\label{aklest1}
|(k+\ell)a_{k\ell}-I(k,\ell)|\le C\left(\frac 1{k^\rho\ell}+\frac 1{k\ell^\rho}\right),
\end{equation}
for all $k,\ell\ge 1$. We want to analyze $I(k,\ell)$ further. For this we will use the parametrization \eqref{Lpar} of $L_\delta(\pm 1)$. 

\begin{lemma}\label{Lem:expest}
Choose $\delta$ so small that $b=\sqrt{2\delta-\delta^2}\le 1/2$. We have the estimate
\begin{equation}\label{expest}
\left|(1+t/k)\lambda^\tau-e^{t\lambda^\tau}\right|\le\frac 4k(t^2+t)e^{-t\sin\phi},
\end{equation}
for $0\le t\le bk$, $k\ge 1$, $\tau=\pm1$ with $\lambda$ as in \eqref{lambda}.
\end{lemma}

The straightforward proof can be found in Section \ref{Sec:lemmas}.

Let $L(\pm 1)$ be the infinite half-lines extending the segments $L_\delta(\pm 1)$, i.e., $\tau L(\tau)$ has the parametrization, $\tau=\pm1$,
\begin{equation}\label{Lparinf}
u=t\lambda^\tau, \quad t\ge 0.
\end{equation}
Define
\begin{equation}\label{Jpkl}
J_p(k,\ell)=\frac 1{(2\pi\I)^2}\int_{L}du\int_{L}dv e^{ku+\ell v}G_p(u,v),
\end{equation}
with $L=L(-1)+L(1)$, and
\begin{equation}\label{Jkl}
J(k,\ell)=\sum_{p=1}^m z_p^{k+\ell}J_p(k,\ell).
\end{equation}

\begin{lemma}\label{Lem:JklIkl}
There is a constant $C$ such that
\begin{equation}\label{JklIkl}
|I(k,\ell)-J(k,\ell)|\le C\Big(\frac 1{k^\rho\ell}+\frac 1{k\ell^\rho}\Big),
\end{equation}
for all $k,\ell\ge 1$.
\end{lemma}
\begin{proof}
Let
\begin{equation*}
I''(k,\ell)=\sum_{p=1}^m z_p^{k+\ell}\sum_{\tau_1,\tau_2=\pm 1}\frac{\tau_1\tau_2\lambda^{\tau_1+\tau_2}}{(2\pi\I)^2k\ell}\int_0^{bk}dt\int_0^{b\ell}ds\,
G_p(\frac {t\lambda^{\tau_1}}k,\frac {s\lambda^{\tau_2}}\ell)e^{t\lambda^{\tau_1}+s\lambda^{\tau_2}}.
\end{equation*}
Using \eqref{Gpbound} and \eqref{expest}, we see that
\begin{align}\label{Iklstep}
&|I(k,\ell)-I''(k,\ell)|\le\sum_{p=1}^m\frac{C}{k\ell}\int_0^{bk}dt\int_0^{b\ell}ds\left[\left|\left((1+\frac {t\lambda^{\tau_1}}k)^{k-1}-e^{t\lambda^{\tau_1}}\right)
\left((1+\frac {s\lambda^{\tau_2}}\ell)^{\ell-1}-e^{s\lambda^{\tau_2}}\right)\right|\right.\notag\\
&+\left.\left|e^{t\lambda^{\tau_1}}\left((1+\frac {s\lambda^{\tau_2}}\ell)^{\ell-1}-e^{s\lambda^{\tau_2}}\right)\right|
+\left|e^{s\lambda^{\tau_2}}\left((1+\frac {t\lambda^{\tau_1}}\ell)^{k-1}-e^{t\lambda^{\tau_1}}\right)\right|\right]\Big|G_p(\frac tk,\frac s\ell)\Big|\notag\\
&\le\frac{C}{k\ell}\int_0^{bk}dt\int_0^{b\ell}ds\left[\frac {16}{k\ell}(t^2+t)(s^2+s)+\frac 4\ell(s^2+s)+\frac 4k(t^2+t)\right]e^{-(t+s)\sin\phi}
\Big|G_p(\frac tk,\frac s\ell)\Big|.
\end{align}
Assume that $\gamma_p>1$. If $t/k\ge s/\ell$, then by 
\begin{align*}
 \left[\frac {16}{k\ell}(t^2+t)(s^2+s)+\frac 4\ell(s^2+s)+\frac 4k(t^2+t)\right]\Big|G_p(\frac tk,\frac s\ell)\Big|
 \le \frac {16}{\ell}(t+1)(s^2+s)+4(s+t+2).
\end{align*}
and analogously for $t/k\le s/\ell$. If we insert these kind of estimates into the integrals in \eqref{Iklstep} we see that
\begin{align*}
 \frac{C}{k\ell}\int_0^{bk}dt\int_0^{b\ell}ds\left[\frac {16}{k\ell}(t^2+t)(s^2+s)+\frac 4\ell(s^2+s)+\frac 4k(t^2+t)\right]e^{-(t+s)\sin\phi}
\Big|G_p(\frac tk,\frac s\ell)\Big|   
\le \frac C{k\ell}\le C\Big(\frac 1{k^\rho\ell}+\frac 1{k\ell^\rho}\Big),
\end{align*}
since $\rho\le 1$.

Next, assume that $\gamma_p<1$. Suppose that $t/k\ge s/\ell$. Then
\begin{align*}
 &\frac{C}{k\ell}\int_0^{bk}dt\int_0^{b\ell}ds\left[\frac {16}{k\ell}(t^2+t)(s^2+s)+\frac 4\ell(s^2+s)+\frac 4k(t^2+t)\right]e^{-(t+s)\sin\phi}
\Big|G_p(\frac tk,\frac s\ell)\Big|\\   
&\le \frac{C}{k\ell}\int_0^{bk}dt\int_0^{b\ell}ds\left[\frac {16}{k\ell}(t^2+t)(s^2+s)+\frac 4\ell(s^2+s)+\frac 4k(t^2+t)\right]e^{-(t+s)\sin\phi}\frac{(s/\ell)^{\gamma_p-1}}{(t/k)^{\gamma_p}}\\
&\le \frac C{k^{2-\gamma_p}\ell^{1+\gamma_p}}+C\int_0^{bk}dt\int_0^{b\ell}ds\left[
\frac {1(s\le\ell t/k)}{k^{1-\gamma_p}\ell^{1+\gamma_p}}(s+1)s^{\gamma_p}t^{-\gamma_p}+\frac 1{k^{2-\gamma_p}\ell^{\gamma_p}}(t+1)s^{\gamma_p-1}t^{1-\gamma_p}\right]e^{-(t+s)\sin\phi}\\
&\le \frac C{k\ell} +C\int_0^\infty dt\int_0^\infty ds\left[\frac 1{k\ell}(s+1)+\frac 1{k^{2-\gamma_p}\ell^{\gamma_p}}(t+1)s^{\gamma_p-1}t^{1-\gamma_p}\right]e^{-(t+s)\sin\phi}
\le \frac C{k\ell^\rho}.
\end{align*}
Here $1(\cdot)$ denote the indicator function for the condition inside the parenthesis. If $t/k\le s/\ell$, the same type of argument gives the bound $C/k^\rho\ell$.

To prove the lemma, we also need a bound on
\begin{equation}\label{BoundA}
   \frac 1{k\ell} \left|\left(\int_{bk}^\infty dt\int_0^{b\ell}ds+\int_{0}^{bk} dt\int_0^{\infty}ds+\int_{bk}^\infty dt\int_{b\ell}^{\infty}ds e^{t\lambda^{\tau_1}+s\lambda^{\tau_2}}G_p(\frac{t\lambda^{\tau_1}}k,\frac{s\lambda^{\tau_2}}{\ell})
   \right)\right|.
\end{equation}
Consider the case $t\ge bk$, $s\le b\ell$ so that $t/k\ge b\ge s/\ell$. First assume that $\gamma_p>1$. Then,
\begin{align*}
 \frac 1{k\ell} \int_{bk}^\infty dt\int_0^{b\ell}ds \,e^{-(t+s)\sin\phi}\left|G_p(\frac{t\lambda^{\tau_1}}k,\frac{s\lambda^{\tau_2}}{\ell})
   \right|  \le
   \frac C{k\ell} \int_{bk}^\infty dt\int_0^{b\ell}ds \,e^{-(t+s)\sin\phi}\frac{1}{t/k}\le
   \frac C{k\ell}\le \frac{C}{k\ell^\rho}.
\end{align*}
Next, if $\gamma_p<1$, then
\begin{align*}
&\frac 1{k\ell} \int_{bk}^\infty dt\int_0^{b\ell}ds \,e^{-(t+s)\sin\phi}\left|G_p(\frac{t\lambda^{\tau_1}}k,\frac{s\lambda^{\tau_2}}{\ell})
   \right| \\
   &\le \frac C{k\ell} \int_{bk}^\infty dt\int_0^{b\ell}ds \,e^{-(t+s)\sin\phi}
   \frac{(s/\ell)^{\gamma_p-1}}{(t/k)^{\gamma_p}}\le \frac C{kl}\frac{k^{\gamma_p}}{\ell^{\gamma_p-1}}\frac 1{(bk)^{\gamma_p}}\le\frac C{k\ell^\rho},
\end{align*}
where we used $t^{\gamma_p}\ge (bk)^{\gamma_p}$. Thus,
\begin{equation*}
 \frac 1{k\ell} \int_{bk}^\infty dt\int_0^{b\ell}ds \,e^{-(t+s)\sin\phi}\left|G_p(\frac{t\lambda^{\tau_1}}k,\frac{s\lambda^{\tau_2}}{\ell})
   \right|  \le\frac C{k\ell^\rho}. 
\end{equation*}
Similarly,
\begin{equation*}
 \frac 1{k\ell} \int_{0}^{bk} dt\int_{b\ell}^{\infty}ds \,e^{-(t+s)\sin\phi}\left|G_p(\frac{t\lambda^{\tau_1}}k,\frac{s\lambda^{\tau_2}}{\ell})
   \right|  \le\frac C{k^\rho\ell}. 
\end{equation*}
Also, if $\gamma_p>1$
\begin{align*}
\frac 1{k\ell} \int_{bk}^{\infty} dt\int_{b\ell}^{\infty}ds \,e^{-(t+s)\sin\phi}
\left[ \frac k{\ell}1(t/k\ge s/\ell)+\frac{\ell}s 1(t/k\le s/\ell)\right]
\le
\frac C{k\ell}\le C\Big(\frac 1{k^\rho\ell}+\frac 1{k\ell^\rho}\Big),
\end{align*}
and, if $\gamma_p<1$,
\begin{align*}
\frac 1{k\ell} \int_{bk}^{\infty} dt\int_{b\ell}^{\infty}ds \,e^{-(t+s)\sin\phi}
\left[ \frac {(s/\ell)^{\gamma_p-1}}{(t/k)^{\gamma_p}}1(t/k\ge s/\ell)+
\frac {(t/k)^{\gamma_p-1}}{(s/\ell)^{\gamma_p}}1(t/k\le s/\ell)\right]
\le
\frac C{k\ell}\le C\Big(\frac 1{k^\rho\ell}+\frac 1{k\ell^\rho}\Big),
\end{align*}
since $t/k\ge b$ and $s/\ell\ge b$. Together the estimates above show that the expression in \eqref{BoundA} is bounded by
$\le C(\frac 1{k^\rho\ell}+\frac 1{k\ell^\rho})$, and we are done.
\end{proof}

The next step is to go from $J(k,\ell)$ to $K(k,\ell)$ defined by \eqref{Kkl}.
\begin{lemma}\label{Lem:JklKkl}
We have the identity
\begin{equation}\label{JklKkl}
K_p(k,\ell)=\frac{\sqrt{k\ell}}{k+\ell}J_p(k,\ell),
\end{equation}
for all $k,\ell\ge 1$, $1\le p\le m$, provided that $\delta$ is small enough.
\end{lemma}

\begin{proof}
Inserting the parametrization \eqref{Lparinf} into \eqref{Jpkl} gives
\begin{equation}\label{Jpklfor1}
\frac{\sqrt{k\ell}}{k+\ell}J_p(k,\ell)=\sum_{\tau_1,\tau_2=\pm 1}\frac{\tau_1\tau_2\lambda^{\tau_1+\tau_2}\sqrt{k\ell}}{(2\pi\I)^2(k+\ell)}
\int_0^\infty dt\int_0^\infty ds e^{kt\lambda^{\tau_1}+\ell s\lambda^{\tau_2}}G_p(t\lambda^{\tau_1},s\lambda^{\tau_2}).
\end{equation}
We see from \eqref{Gp} that
\begin{equation}\label{Gpfor}
G_p(t\lambda^{\tau_1},s\lambda^{\tau_2})=\frac{\gamma_p\lambda^{-\tau_1}}t\frac{
1-(s/t)^{\gamma_p-1}\lambda^{(\tau_2-\tau_1)(\gamma_p-1)}}{1-(s/t)^{\gamma_p}\lambda^{(\tau_2-\tau_1)(\gamma_p)}}.
\end{equation}
Insert \eqref{Gpfor} into \eqref{Jpklfor1} and make the change of variables $s=xt$ in the $s$-integral. This gives
\begin{align}\label{Jpklfor2}
\frac{\sqrt{k\ell}}{k+\ell}J_p(k,\ell)&=-\sum_{\tau_1,\tau_2=\pm 1}\frac{\gamma_p\tau_1\tau_2\lambda^{\tau_2}\sqrt{k\ell}}{4\pi^2(k+\ell)}
\int_0^\infty dx\left(\int_0^\infty e^{t(k\lambda^{\tau_1}+\ell x\lambda^{\lambda_2})}dt\right)
\frac{1-x^{\gamma_p-1}\lambda^{(\tau_2-\tau_1)(\gamma_p-1)}}{1-x^{\gamma_p}\lambda^{(\tau_2-\tau_1)(\gamma_p)}}\notag\\
&=\sum_{\tau_1,\tau_2=\pm 1}\frac{\gamma_p\tau_1\tau_2\lambda^{\tau_2}\sqrt{k\ell}}{4\pi^2(k+\ell)}\int_0^\infty
\frac 1{k\lambda^{\tau_1}+\ell x\lambda^{\tau_2}}\frac{1-x^{\gamma_p-1}\lambda^{(\tau_2-\tau_1)(\gamma_p-1)}}{1-x^{\gamma_p}\lambda^{(\tau_2-\tau_1)(\gamma_p)}}dx,
\end{align}
since $\re(k\lambda^{\tau_1}+\ell x\lambda^{\tau_2})=(k+\ell x)\cos(\phi+\pi/2)<0$ if $\delta$ and hence $\phi$ is small enough. Some manipulations give
\begin{align*}
&\frac {\sqrt{k\ell}}{(k+\ell)(k\lambda^{\tau_1}+\ell x\lambda^{\tau_2})}\frac{1-x^{\gamma_p-1}\lambda^{(\tau_2-\tau_1)(\gamma_p-1)}}{1-x^{\gamma_p}\lambda^{(\tau_2-\tau_1)(\gamma_p)}}\\
&=\frac {\lambda^{-\tau_2}}{x\sqrt{k\ell}(\sqrt{k/\ell}+\sqrt{\ell/k})((k/x\ell\lambda^\tau)^{1/2}+(x\ell\lambda^\tau/k)^{1/2})}
\frac{(x\lambda^\tau)^{(\gamma_p-1)/2}-(x\lambda^\tau)^{-(\gamma_p-1)/2}}{(x\lambda^\tau)^{\gamma_p/2}-(x\lambda^\tau)^{-\gamma_p/2}},
\end{align*}
where $\tau=\tau_2-\tau_1$. We can insert this in \eqref{Jpklfor2} and make the change of variables $x=e^s$. If we let $k/\ell=e^t$, $\phi^*=\phi+\pi/2$ and recall that 
$\tau=\tau_2-\tau_1$, we obtain
\begin{align*}
&\frac{\sqrt{k\ell}}{k+\ell}J_p(k,\ell)=\frac{\gamma_p}{16\pi^2\sqrt{kl}}\sum_{\tau_1,\tau_2=\pm 1}\int_\R
\frac {\tau_1\tau_2}{\cosh(t/2)\cosh((t-(s+\I\tau\phi^*)/2))}\frac{\sinh((\gamma_p-1)(s+\I\tau\phi^*)/2)}{\sinh(\gamma_p(s+\I\tau\phi^*)/2)}\\
&=\frac{\gamma_p}{16\pi^2\sqrt{kl}}\left(2\int_\R ds-\int_{\R+2\I\phi^*} ds-\int_{\R-2\I\phi^*}ds\right) \frac 1{\cosh(t/2)\cosh((t-s)/2)}
\frac{\sin((\gamma_p-1)s/2)}{\sin(\gamma_ps/2)}.
\end{align*}
Introduce the contours $\Gamma_+=\R-(\R+2\I\phi^*)$ and $\Gamma_-=(\R-2\I\phi^*)-\R$, where $\R$ is the oriented real line. Then,
\begin{equation*}
\frac{\sqrt{k\ell}}{k+\ell}J_p(k,\ell)=\frac{\gamma_p}{16\pi^2\sqrt{kl}}\left(\int_{\Gamma_+}ds-\int_{\Gamma_-}ds\right)
\frac 1{\cosh(t/2)\cosh((t-s)/2)}\frac{\sin((\gamma_p-1)s/2)}{\sin(\gamma_ps/2)}.
\end{equation*}
Changing $s$ to $-s$ maps $\Gamma_-$ to $\Gamma_+$, so
\begin{align}\label{Jpklfor3}
\frac{\sqrt{k\ell}}{k+\ell}J_p(k,\ell)&=\frac{\gamma_p}{16\pi^2\sqrt{kl}}\int_{\Gamma_+}\frac 1{\cosh(t/2)}\left(\frac 1{\cosh((t-s)/2)}+\frac 1{\cosh((t+s)/2)}\right)
\frac{\sin((\gamma_p-1)s/2)}{\sin(\gamma_ps/2)}ds\notag\\
&=\frac{\gamma_p}{4\pi^2\sqrt{kl}}\int_{\Gamma_+}\frac{\cosh(s/2)\sinh((\gamma_p-1)s/2)}{(\cosh t+\cosh s)\sinh(\gamma_ps/2)}ds.
\end{align}
To compute this integral, we need to find the poles of the integrand inside $\Gamma_+$. Note that $\cosh s=-\cosh t$ if and only if $s=\pm t+(2k+1)\pi \I$, $k\in\Z$. Since
$2\phi^*=2\phi+\pi>\pi$ but close to $\pi$ if $\delta$ is small, the only poles of this type inside $\Gamma_+$ are $s=\pm t+\pi\I$. Furthermore, $\sinh(\gamma_ps/2)=0$
if and only if $s=2\pi k/\gamma_p$, $k\in\Z$. Note that $0<\gamma_p<2$ gives $\pi k<2\pi k/\gamma_p$, $k\ge 1$, and if we choose $\delta$ small enough, then
$2\phi^*<2\pi/\gamma_p$ and $\sinh(\gamma_ps/2)$ has no pole inside $\Gamma_+$. The residue theorem applied to \eqref{Jpklfor3} gives, after some
manipulations,
\begin{equation*}
\frac{\sqrt{k\ell}}{k+\ell}J_p(k,\ell)=\frac{\gamma_p\I}{4\pi\sqrt{k\ell}}\sum_{s=\pm t+\I\pi}\left(\coth(s/2)-\cosh(\gamma_ps/2)\right).
\end{equation*}
Further computations show that we can rewrite this expression as
\begin{equation}\label{Jpklfor4}
\frac{\sqrt{k\ell}}{k+\ell}J_p(k,\ell)=\frac{\gamma_p}{2\pi\sqrt{k\ell}}\frac{\sin\pi\gamma_p}{\cos\pi\gamma_p-\cosh\gamma_pt}.
\end{equation}
Inserting $t=\log(k/\ell)$ into this formula and recalling the definition \eqref{Kpkl} of $K_p(k,\ell)$, we see that \eqref{JklKkl} holds.
\end{proof}

We can now give the

\begin{proof}[Proof of Theorem \ref{Thm:bklAs}]
From \eqref{bkl} and \eqref{aklest1}, we see that
\begin{equation*}
\left|b_{k\ell}-\frac{\sqrt{k\ell}}{k+\ell}I(k,\ell)\right|\le C\frac{\sqrt{k\ell}}{k+\ell}\left(\frac 1{k^\rho\ell}+\frac 1{k\ell^\rho}\right),
\end{equation*}
and combining \eqref{JklIkl}, \eqref{Jkl}, \eqref{Jpkl} and \eqref{JklKkl}, leads to the estimate
\begin{equation*}
\left|\frac{\sqrt{k\ell}}{k+\ell}I(k,\ell)-K(k,\ell)\right|\le C\frac{\sqrt{kl}}{(k+\ell)}\Big(\frac 1{k^\rho\ell}+\frac 1{k\ell^\rho}\Big),
\end{equation*}
and we have proved the theorem.
\end{proof}

\subsection{Proofs of Propositions~\ref{Prop:TrComp1}, \ref{Prop:TrComp2}, and \ref{Prop:TrComp4}}\label{Subsec:Basicprop}
This section proves the basic propositions needed for the proofs of the main theorems. Before we give the proofs we need several auxiliary results.
Define for $u,v>0$ and $\rho>0$ the function
\begin{equation}\label{frho}
f_\rho(u,v)=\frac 1{\sqrt{uv}}\left(\frac{uv}{u^2+v^2}\right)^\rho.
\end{equation}
This function has been chosen so that it gives upper bounds on both $K_p(k,\ell)$ and $r_{k\ell}$. In fact, we have the following lemma.
\begin{lemma}\label{Lem:Kpbound}
Let $\rho$ be given by \eqref{rho}. There is a constant $C$ such that
\begin{equation}\label{Kpbound1}
|K_p(u,v)|\le Cf_\rho(u,v),
\end{equation}
for $u,v>0$, $1\le p\le m$, and
\begin{equation}\label{Kpbound2}
|r_{k\ell}|\le Cf_\rho(k,\ell),
\end{equation}
for $k,\ell\ge 1$.
\end{lemma}
The lemma will be proved in Section \ref{Sec:lemmas}. An immediate consequence of this lemma and \eqref{rkl} is that
\begin{equation}\label{bklest} 
|b_{k\ell}|\le Cf_\rho(k,\ell),
\end{equation}
for $k,\ell\ge 1$ and some constant $C$.
The next lemma gives some properties of the function $f_\rho$ that we will use in our arguments. The proof is postponed to Section \ref{Sec:lemmas}.

\begin{lemma}\label{Lem:frhobounds}
There is a constant $C$ such that
\begin{equation}\label{dfrho}
    \left|\frac{\partial f_\rho}{\partial u}(u,v)\right|\le\frac{C}{u}
    f_\rho(u,v),
\end{equation}
for $u,v>0$, and
\begin{equation}\label{dKp}
    \left|\frac{\partial K_p}{\partial u}(u,v)\right|\le\frac{C}{u}
    f_\rho(u,v),
\end{equation}
for $u,v\ge 1$. 
For $v_1,v_2\ge 1$, we have the bounds,
\begin{equation}\label{frhob1}
\sum_{k=1}^\infty f_\rho(v_1,k)f_\rho(k,v_2)\le C\int_1^\infty f_\rho(v_1,u)f_\rho(u,v_2)\,du\le Cf_{\rho/2}(v_1,v_2),
\end{equation}
and
\begin{equation}\label{frhob2}
\int_1^\infty f_\rho(v_1,u)f_\rho(u,v_2)\frac{du}u\le\frac C{(v_1v_2)^{(1+\rho)/2}},
\end{equation}
for some constant $C$.
\end{lemma}
Write
\begin{equation}\label{hrho}
h_\rho(x)=\frac 1{2^\rho\cosh^\rho x},
\end{equation}
so that
\begin{equation}\label{frhouv}
    f_\rho(u,v)=\frac{1}{\sqrt{uv}} h_{\rho}(\log\frac uv),
\end{equation}
in analogy with \eqref{Hpt} and \eqref{KpHp}.
Write $F_{\rho,1}(v_1,v_2)=f_\rho(v_1,v_2)$ and, for $k\ge 2$, 
\begin{equation}\label{Frhok}
F_{\rho,k}(v_1,v_2)=\int_{[0,\infty)^{k-1}}f_\rho(v_1,u_1)f_\rho(u_1,u_2)\dots f_\rho(u_{k-1},v_2)\,du.
\end{equation}
For a function $f(x)$ on $\R$ we write $f^{*1}(x)=f(x)$, and
$f^{*k}(x)=(f*f^{*(k-1)})(x)$ for convolution with $k$ factors.
We see from \eqref{hrho}, \eqref{frhouv} and \eqref{Frhok} that
\begin{equation}\label{Frhokstar}
    F_{\rho,k}(u,v)=\frac 1{\sqrt{uv}}h_\rho^{*k}(\log\frac uv).
\end{equation}
In the proofs of many results below we will need various estimates of $F_{\rho,k}$. These are collected in the next lemma together with some estimates of $H_p^{*k}$, see the definition of $H_p$ in \eqref{Hpt}.
\begin{lemma}\label{Lem:Frhok}
There is a constant $C$ so that
\begin{align}\label{Frhok1}
F_{\rho,k}(u,v)&\le \frac{C^k}{\sqrt{uv}},
\\
\label{Frhok3}
\left|\frac{\partial F_{\rho,k}}{\partial u}(u,v)\right|&\le \frac{C}{u}F_{\rho,k}(u,v),\\
\label{Frhok4}
\int_1^\infty F_{\rho,k}(u,v)\frac{du}{\sqrt{u}}&\le\frac{C^k}{\sqrt{v}},\\
\label{Frhok5}
\sum_{j=1}^\infty\frac{1}{\sqrt{j}}F_{\rho,k}(j,u)&\le F_{\rho,k}(1,v)+C\int_1^\infty F_{\rho,k}(u,v)\frac{du}{\sqrt{u}}\le\frac{C^k}{\sqrt{v}},\\
\label{Frhok6}
\left|H_p^{*k}(\log\frac uv)\right|&\le C^k\sqrt{uv}F_{\rho,k}(u,v),
\\
\label{Frhok7}
\left|\frac{\partial}{\partial u}H_p^{*k}(\log\frac uv)\right|&\le C^k\sqrt{\frac{v}{u}}F_{\rho,k}(u,v),
\end{align}
for $u,v\ge 1$, $k\ge 1$. Furthermore,
\begin{equation}\label{Frhok2}
F_{\rho,k}(e^{s},e^{t})\le C^k\min(1,\frac 1{|s-t|^3})e^{-(s+t)/2},
\end{equation}
for all $s,t\ge 0$, $k\ge 1$.
\end{lemma}

The proof is in Section \ref{Sec:lemmas}. We will also need some estimates of sums involving $K_p$ and again the proof can be found in Section \ref{Sec:lemmas}.

\begin{lemma}\label{Lem:Kpest}
There is a constant $C$ so that
\begin{equation}\label{Kpest1}
\left|\sum_{\ell=1}^N(z_p\bar{z}_q)^\ell K_p(j,\ell)K_q(\ell,k)\right|\le\frac C{(jk)^{(1+\rho)/2}}
\end{equation}
if $p\neq q$, $j,k\ge 1$, for any $N\ge 1$. Also,
\begin{equation}\label{Kpest2}
\left|\sum_{\ell=1}^\infty K_p(v_1,\ell)K_p(\ell,v_2)-\int_1^\infty K_p(v_1,u)K_p(u,v_2)\,du\right|\le\frac C{(v_1v_2)^{(1+\rho)/2}},
\end{equation}
for $v_1,v_2\ge 1$.
\end{lemma}

The next lemma will be used to compare traces and matrix elements of powers.

\begin{lemma}\label{Lem:Tr}
Let $D_s=((D_s)_{k\ell})_{k,\ell\ge 1}$, $1\le s\le 4$, be Hilbert-Schmidt operators on $\ell^2(\Z_+)$. Assume that
\begin{equation}\label{Tr1}
|(D_s)_{k\ell}|\le Cf_\rho(k,\ell), \quad 1\le s\le 4,
\end{equation}
for some constant $C$ and all $k,\ell\ge 1$. Then, there is a constant $C$ such that
\begin{align}\label{Trkl}
|(D_1D_2)^i_{kl}-(D_3D_4)^i_{kl}|&\le C^i\sum_{j=0}^{i-1}\sum_{\mu,\nu=1}^\infty \bigg[F_{\rho,2j}(k,\mu)|(D_1)_{\mu\nu}-(D_3)_{\mu\nu}|F_{\rho,2(i-1-j)+1}(\nu,\ell)\notag
\\&+F_{\rho,2j+1}(k,\mu)|(D_2)_{\mu\nu}-(D_4)_{\mu\nu}|F_{\rho,2(i-1-j)}(\nu,\ell)\bigg],
\end{align}
for all $i,k,\ell\ge 1$, where $F_{\rho,0}(k,\ell)=\delta_{k\ell}$.
Furthermore
\begin{equation}\label{Tr2}
|\Tr(D_1D_2)^i-\Tr(D_3D_4)^i|\le C^i\sum_{k,\ell=1}^\infty(|(D_1)_{k\ell}-(D_3)_{k\ell}|+|(D_2)_{k\ell}-(D_4)_{k\ell}|)F_{\rho,2i-1}(k,\ell),
\end{equation}
for all $i\ge 1$.
\end{lemma}
\begin{proof}
    The statements of the lemma follow immediately from  the identity
    \begin{equation*}
        \sum_{j=0}^{i-1}(D_3D_4)^j(D_1-D_3)D_2(D_1D_2)^{i-1-j}+
        (D_3D_4)^jD_3(D_2-D_4)(D_1D_2)^{i-1-j}=(D_1D_2)^i-(D_3D_4)^i,
    \end{equation*}
    the assumption \eqref{Tr1}, the inequality \eqref{frhob1} and the definition \eqref{Frhok}.
\end{proof}

Consider the operator $K=(K(k,\ell))_{k\ell}$ on $\ell^2(\Z_+)$.

\begin{lemma}\label{Lem:Trcomp5}
There is a constant $C$ such that
\begin{equation}\label{TrComp5}
|\Tr(P_nBB^*P_n)^i-\Tr(P_nKK^*P_n)^i|\le C^i,
\end{equation}
and
\begin{equation}\label{BBKKest}
    \sum_{k,\ell=1}^\infty\frac 1{\sqrt{k\ell}}|(P_nBB^*P_n)^i_{k\ell}-(P_nKK^*P_n)^i_{k\ell}|\le C^i,
\end{equation}
for all $i\ge 1$.
\end{lemma}

\begin{proof}
In Lemma \ref{Lem:Tr} we take $D_1=P_nB$, $D_2=D_1^*$, $D_3=P_nK$, and $D_4=D_3^*$. Then,
\begin{equation*}
    |(D_1)_{\mu\nu}-(D_3)_{\mu\nu}|=|(P_n(B-K))_{\mu\nu}|=1(\mu\le n)|r_{\mu\nu}|,
\end{equation*}
and consequently, 
\begin{align*}
 |(P_nBB^*P_n)^i_{k\ell}-(P_nKK^*P_n)^i_{k\ell}|&\le   
 C^i\sum_{j=0}^\infty\sum_{\mu,\nu=1}^\infty\Big[F_{\rho,2j}(k,\mu)1(\mu\le n)|r_{\mu\nu}|F_{\rho,2(i-1-j)+1}(\nu,\ell)\\
 &+F_{\rho,2j+1}(k,\mu)|r_{\mu\nu}|1(\nu\le n)F_{\rho,2(i-1-j)}(\nu,\ell)\Big].
\end{align*}
It follows using \eqref{rklEst} and \eqref{Frhok1} that
\begin{align*}
&|\Tr(P_nBB^*P_n)^i-\Tr(P_nKK^*P_n)^i|\le C^i  \sum_{\mu,\nu=1}^\infty \frac{\sqrt{\mu\nu}}{\mu+\nu}\left(\frac 1{\mu^\rho\nu}+\frac 1{\mu\nu^\rho}\right)F_{\rho,2i-1}(\mu,\nu)\\
&\le C^i\sum_{\mu,\nu=1}^\infty \frac{1}{\mu+\nu}\left(\frac 1{\mu^\rho\nu}+\frac 1{\mu\nu^\rho}\right)
\le 2C^i\sum_{\nu=1}^\infty \sum_{\mu=1}^\nu
\frac{1}{\mu+\nu}\left(\frac 1{\mu^\rho\nu}+\frac 1{\mu\nu^\rho}\right)
\le C^i  \sum_{\nu=1}^\infty \sum_{\mu=1}^\nu \frac 1{\mu\nu^{1+\rho}}\\
&\le C^i \sum_{\nu=1}^\infty\frac{\log(\nu+1)}{\nu^{1+\rho}}
\le C^i,
\end{align*}
since $0<\rho\le 1$. Also, using \eqref{Frhok5},
\begin{align*}
  \sum_{k,\ell=1}^\infty\frac 1{\sqrt{k\ell}}|(P_nBB^*P_n)^i_{k\ell}-(P_nKK^*P_n)^i_{k\ell}|\le C^i  \sum_{\mu,\nu=1}^\infty\frac 1{\sqrt{\mu\nu}}|r_{\mu\nu}|
  \le \sum_{\mu,\nu=1}^\infty\frac 1{\mu+\nu}\left(\frac 1{\mu^\rho\nu}+\frac 1{\mu\nu^\rho}\right)\le C^i,
\end{align*}
so we have proved the lemma.
\end{proof}

We can now prove Proposition \ref{Prop:TrComp1}.

\begin{proof}[Proof of Proposition \ref{Prop:TrComp1}]
We see from \eqref{Kkl} that
\begin{equation*}
K^*(k,\ell)=\sum_{p=1}^m\bar{z}_p^{k+\ell}K_p(k,\ell),
\end{equation*}
since $K_p(k,\ell)$ is real and symmetric. Using  $k_0=k,k_i=\ell$, we have that 
\begin{align}\label{Canc1}
 (P_nKK^*P_n)^i_{k\ell}&=\sum_{k_1,\dots,k_{i-1}=1}^n  \sum_{\ell_1,\dots,\ell_i=1}^\infty\prod_{r=1}^i\left(\sum_{p=1}^m z_p^{k_{r-1}+\ell_r}K_p(k_{r-1},\ell_r)\right)\left(\sum_{p=1}^m \bar{z}_p^{\ell_r+k_r}K_p(\ell_r,k_r)\right)1(k,\ell\le n)\notag\\
 &=\sum_{p_1,\dots,p_{2i}=1}^m\sum_{k_1,\dots,k_{i-1}=1}^n  \sum_{\ell_1,\dots,\ell_i=1}^\infty\left(\prod_{r=1}^iz_{p_{2r-1}}^{k_{r-1}+\ell_r}\bar{z}_{p_{2r}}^{\ell_r+k_r}\prod_{r=1}^i
 K_p(k_{r-1},\ell_r)K_p(\ell_r,k_r)\right)1(k,\ell\le n).
\end{align}
The term in \eqref{Canc1} corresponding to $p_1=\dots=p_{2i}=p$ gives
\begin{equation*}
    \sum_{p=1}^mz_p^k\bar{z}_p^\ell\sum_{k_1,\dots,k_{i-1}=1}^n  \sum_{\ell_1,\dots,\ell_i=1}^\infty\prod_{r=1}^iK_p(k_{r-1},\ell_r)K_p(\ell_r,k_r)1(k,\ell\le n)=\sum_{p=1}^mz_p^k\bar{z}_p^\ell(P_nK_p^2P_n)^i_{k\ell}.
\end{equation*}
The other terms in \eqref{Canc1} are terms where there is a smallest $j\ge 1$ such that $p_j\neq p_{j+1}$. These terms will give a smaller contribution due to cancellation. It follows from \eqref{Kpest1} and \eqref{Kpbound1} that
\begin{equation}\label{Canc2}
    \left|(P_nKK^*P_n)^i_{k\ell}-\sum_{p=1}^m z_p^k\bar{z}_p^\ell(P_nK_p^2P_n)^i_{k\ell}\right|\le
    C^i\sum_{j=1}^{2i-1}\sum_{\mu,\nu=1}^\infty F_{\rho,j-1}(k,\mu)
    \frac 1{(\mu\nu)^{(1+\rho)/2)}}F_{\rho,2i-j}(\nu,\ell)1(k,\ell\le n),
\end{equation}
where again $F_{\rho,0}(k,\ell)=\delta_{k\ell}$. We see from \eqref{Canc2} and \eqref{Frhok1} that
\begin{equation*}
    \left|\Tr(P_nKK^*P_n)^i-\sum_{p=1}^m\Tr(P_nK_p^2P_n)^i\right|
    \le C^i\sum_{\mu,\nu=1}^\infty\frac 1{(\mu\nu)^{(1+\rho)/2}}F_{\rho,2i-1}(\mu,\nu)\le C^i\sum_{\mu,\nu=1}^\infty\frac 1{(\mu\nu)^{1+\rho/2}}\le C^i.
\end{equation*}
Combining this estimate with Lemma \ref{Lem:Trcomp5} proves the proposition.

Note that, using \eqref{Frhok5} in \eqref{Canc2}, we obtain the estimate
\begin{equation*}
    \sum_{k,\ell=1}^\infty\frac 1{\sqrt{k\ell}}\left|(P_nKK^*P_n)^i_{k\ell}-\sum_{p=1}^mz_p^k\bar{z}_p^\ell(P_mK_p^2P_n)^i_{k\ell}\right|
    \le C^i\sum_{\mu,\nu=1}^\infty \frac 1{(\mu\nu)^{1+\rho/2}}\le C^i.
\end{equation*}
This estimate together with \eqref{BBKKest} gives
\begin{equation}\label{TrComp1:2}
    \sum_{k,\ell=1}^\infty \frac 1{\sqrt{k\ell}}|(P_nBB^*P_n)^i_{k\ell}-\sum_{p=1}^m z_p^k\bar{z}_p^\ell(P_nK_p^2P_n)^i_{k\ell}|\le C^i,
\end{equation}
for all $i,n\ge 1$.

\end{proof}

We will now prove Proposition \ref{Prop:TrComp2} via a series of lemmas. In the next lemma we extend all summations in $\Tr(P_nK_p^2P_n)^i$ except one to sums
from one to infinity.
\begin{lemma}\label{Lem:Sumext}
There is a constant $C$ such that
\begin{equation}\label{Sumext}
\left|\Tr(P_nK_p^2P_n)^i-\Tr(P_nK_p^{2i}P_n)\right|\le C^i,
\end{equation}
and 
\begin{equation}\label{Sumextkl}
    \sum_{k,\ell=1}^\infty\frac 1{\sqrt{k\ell}}\left|(P_nK_p^2P_n)^i_{k\ell}-(P_nK_p^{2i}P_n)_{k\ell}\right| \le C^i,   
\end{equation}
for all $n, i\ge 1$.
\end{lemma}

\begin{proof}
For $1\le j\le i$, let
\begin{equation*}
D_{ij}(k,\ell)=\sum_{k_1,\dots,k_{j-1}=1}^n\sum_{k_j,\dots,k_{i-1}=1}^\infty\prod_{r=1}^{i}K_p^2(k_{r-1},k_r),
\end{equation*}
where $k_0=k_i=\ell$. 
Then,
\begin{align*}
 (P_nK_p^2P_n)^i_{k\ell}&=D_{ii}(k,\ell)1(k,\ell\le n),\\
 (P_nK_p^{2i}P_n)_{k\ell}&=D_{i1}(k,\ell)1(k,\ell\le n),
\end{align*}
and consequently
\begin{equation*}
    |(P_nK_p^2P_n)^i_{k\ell}-(P_nK_p^{2i}P_n)_{k\ell}|\le\sum_{j=1}^{i-1}|D_{i,j+1}(k,\ell)-D_{ij}(k,\ell)|1(k,\ell\le n).
\end{equation*}
It follows from \eqref{Kpbound1}, \eqref{frhob1}, and \eqref{Frhok} that
\begin{align*}
|D_{i,j+1}(k,\ell)-D_{ij}(k,\ell)|&\le
\sum_{k_1,\dots,k_{j-1}}^n\sum_{k_j=n+1}^\infty\sum_{k_{j+1},\dots,k_{i-1}=1}^\infty\prod_{r=1}^i f_{\rho/2}(k_{r-1},k_r)\\
&\le C^i   \sum_{k_{j-1}=1}^n\sum_{k_j=n+1}^\infty F_{\rho/2,j-1}(k,k_{j-1})f_{\rho/2}(k_{j-1},k_j)F_{\rho/2,j-1}(k_j,\ell).
\end{align*}
From this and \eqref{Frhok1} we see that
\begin{align*}
    |\Tr(P_nK_p^{2i}P_n-\Tr(P_nK_p^2P_n)^i|&\le C^i \sum_{k=1}^n\sum_{\ell=n+1}^\infty F_{\rho/2,i-1}(k,\ell)f_{\rho/2}(\ell,k)\\
    &\le \sum_{k=1}^n\sum_{\ell=n+1}^\infty\frac 1{k\ell}\left(\frac{k\ell}{k^2+\ell^2}\right)^{\rho/2}\le C^i,
\end{align*}
which proves \eqref{Sumext}. If we use the estimate \eqref{Frhok5}, we see that \eqref{Sumextkl} also follows.

\end{proof}

In the next lemma we replace infinite sums by integrations.

\begin{lemma}\label{Lem:SumInt}
There is a constant $C$ such that 
\begin{equation}\label{SumInt}
\left|\Tr (P_nK_p^{2i}P_n)-\sum_{\ell=1}^n\int_{[1,\infty)^{2i-1}}\prod_{r=1}^{2i}K_p(u_{r-1},u_r)\,du\right|\le C^i,
\end{equation}
for all $i\ge 1$, where $u_0=u_{2i}=\ell$, and
\begin{equation}\label{SumIntkl}
 \sum_{k,\ell=1}^n\frac{1}{\sqrt{k\ell}}\left|(P_nK_p^{2i}P_n)_{k\ell}- \int_{[1,\infty)^{2i-1}} \prod_{r=1}^{2i}K_p(u_{r-1},u_r)\,du\right|\le C^i, 
\end{equation}
where $u_0=k, u_{2i}=\ell$.
\end{lemma}

\begin{proof}
The proof has the same structure as the proof of Lemma \ref{Lem:Sumext}. Define
\begin{equation*}
\tilde{D}_{ij}(k,\ell)=\sum_{k_1,\dots,k_{j-1}=1}^\infty\int_{[1,\infty)^{2i-j}}du_j\dots du_{2i-1}\left(\prod_{r=1}^{j-1}K_p(k_{r-1},k_r)\right)K_p(k_{j-1},u_j)
\left(\prod_{r=j+1}^{2i}K_p(u_{j-1},u_j)\right),
\end{equation*}
for $1\le j\le 2i$, where $k_0=k$ and $u_{2i}=\ell$. We see that
\begin{align*}
    &\sum_{j=1}^{2i-1} |\tilde{D}_{i,j+1}(k,\ell)-\tilde{D}_{ij}(k,\ell)|
    \le \sum_{j=1}^{2i-1}\sum_{k_1,\dots,k_{j-1}=1}^\infty\int_{[1,\infty)^{2i-j}}du_j\dots du_{2i-1}
    \left(\prod_{r=}^{j-1}f_\rho(k_{r-1},k_r)\right) \\ &\times \left|\sum_{k_j=1}^\infty K_p(k_{j-1},k_j)K_p(k_j,u_{j+1})- \int_1^\infty K_p(k_{j-1},u_j)K_p(u_j,u_{j+1})\,du_j\right|\\
    &\times\int_{[1,\infty)^{2i-j-1}}du_{j+1}\dots du_{2i-1}\left(\prod_{r=j+2}^{2i}f_\rho(u_{r-1},u_r)\right)\\
    &\le C^i \sum_{j=1}^{2i-1}\sum_{k_{j-1}=1}^\infty\int_1^\infty du_{j+1}F_{\rho,j-1}(k,k_{j-1})\frac 1{(k_j u_{j+1})^{(1+\rho)/2}}F_{\rho,2i-1-j}(u_{j+1},\ell),
\end{align*}
where we used the bounds \eqref{Kpbound1}, \eqref{Frhok} and \eqref{Kpest2}. It now follows from \eqref{Frhok1} that \eqref{SumInt} holds, and using \eqref{Frhok5}, we see similarly that \eqref{SumIntkl} is true.
\end{proof}

Next, we want to replace the integrations in \eqref{SumInt} over $[1,\infty)$ by integrations over $[0,\infty)$.
\begin{lemma}\label{Lem:ChInt}
There is a constant $C$ such that we have the estimates
\begin{equation}\label{ChInt}
\left|\sum_{\ell=1}^n\left(\int_{[0,\infty)^{2i-1}}-\int_{[1,\infty)^{2i-1}}\right)\prod_{r=1}^{2i}K_p(u_{r-1},u_r)\,du\right|\le C^i
\end{equation}
where $u_0=u_{2i}=\ell$, and
\begin{equation}\label{ChIntkl}
\sum_{k,l=1}^\infty\frac 1{\sqrt{kl}}
\left|\left(\int_{[0,\infty)^{2i-1}}-\int_{[1,\infty)^{2i-1}}\right)\prod_{r=1}^{2i}K_p(u_{r-1},u_r)\,du\right|\le C^i,
\end{equation}
where $u_0=k, u_{2i}=\ell$, for all $i\ge 1$.
\end{lemma}

\begin{proof}
It follows from \eqref{frho} and \eqref{Kpbound1} that 
\begin{equation}\label{ChIntklest}
\left|\left(\int_{[0,\infty)^{2i-1}}-\int_{[1,\infty)^{2i-1}}\right)\prod_{r=1}^{2i}K_p(u_{r-1},u_r)\,du\right|\le
C^i\sum_{j=1}^{2i-1}\int_0^1duF_{\rho,j}(k,u)F_{\rho,2i-j}(u,\ell).
\end{equation}
It follows from \eqref{ChIntklest} and \eqref{Frhok2} that the left side of \eqref{ChInt} is bounded by
\begin{align*}
C^i\sum_{j=1}^{2i-1}\int_1^\infty dv\int_0^1duF_{\rho,j}(v,u)F_{\rho,2i-j}(u,v) &=   C^i\sum_{j=1}^{2i-1}\int_0^\infty ds\int_{-\infty}^0 e^{s+t}F_{\rho,j}(e^s,e^t)F_{\rho,2i-j}(e^t,e^s)\\
&\le \int_0^\infty ds\int_0^\infty dt\min(1,\frac 1{(t+s)^6})\le C^i.
\end{align*}
It also follows from \eqref{ChIntklest} that the left side of \eqref{ChIntkl} is bounded by
\begin{equation*}
    C^i\sum_{j=1}^{2i-1}\sum_{k,l=1}^\infty\frac 1{\sqrt{kl}}\int_0^1duF_{\rho,j}(k,u)F_{\rho,2i-j}(u,\ell).
\end{equation*}
Now, by \eqref{Frhok5} this is bounded by
\begin{align*}
    &C^i\sum_{j=1}^{2i-1}\int_0^1dv\left(F_{\rho,j}(1,v)+\int_{1}^\infty F_{\rho,j}(u,v)\frac{du}{\sqrt{u}}\right)\left(F_{\rho,2i-j}(1,v)+\int_{1}^\infty F_{\rho,2i-j}(u,v)\frac{du}{\sqrt{u}}\right)\\
    &=C^i\sum_{j=1}^{2i-1}\int_{-\infty}^0ds\left(F_{\rho,j}(e^0,e^s)+ \int_0^\infty F_{\rho,j}(e^t,e^s)e^{t/2}\,dt\right)\left(F_{\rho,2i-j}(e^0,e^s)+ \int_0^\infty F_{\rho,2i-j}(e^t,e^s)e^{t/2}\,dt\right)\\ &\le C^i\int_0^\infty\left(\min(1,\frac 1{s^3})+\int_0^\infty\min(1,\frac 1{(t+s)^3})\,dt\right)^2\le C^i,
    \end{align*}
where we also used the estimate \eqref{Frhok2}.
\end{proof}

Finally, we are ready for the proof of Proposition \ref{Prop:TrComp2}.

\begin{proof}[Proof of Proposition \ref{Prop:TrComp2}]
Combining Lemmas \ref{Lem:Sumext}, \ref{Lem:SumInt} and \ref{Lem:ChInt}, we see that
\begin{equation*}
\left|\Tr(P_nK_p^2P_n)^i-\sum_{\ell=1}^n\int_{[0,\infty)^{2i-1}}\prod_{r=1}^{2i}K_p(u_{r-1},u_r)\,du_1\dots du_{2i-1}\right|\le C^i,
\end{equation*}
for $i\ge 1$, where $u_0=u_{2i}=\ell$. A change of variables $u_r=e^{s_r}$, with $s_0=s_{2i}=\log\ell$ gives
\begin{equation}\label{Kpint}
\int_{[0,\infty)^{2i-1}}\prod_{r=1}^{2i}K_p(u_{r-1},u_r)\,du=\frac 1\ell \int_{\R^{2i-1}}\prod_{r=1}^{2i} H_p(s_{r-1}-s_r)\,ds,
\end{equation}
by \eqref{KpHp}. Let $f^{*1}(x)=f(x)$, and $f^{*k}(x)=f*f^{*(k-1)}(x)$ be convolution with $k$ factors. Shifting the variables $s_r$ in \eqref{Kpint} by $\log\ell$ gives
\begin{equation*}
\int_{\R^{2i-1}}\prod_{r=1}^{2i}H_p(s_{r-1}-s_r)\,ds=\int_{\R^{2i-1}}H_p(0-s_1)H_p(s_1-s_2)\dots H_p(s_{2i-1})\,ds=H_p^{*2i}(0).
\end{equation*}
By the Fourier inversion formula,
\begin{equation*}
H_p^{*2i}(0)=\frac 1{2\pi}\int_{\R}\hat{H}_p(\xi)^{2i}\,d\xi.
\end{equation*}
This proves the proposition.

\end{proof}

It remains to prove Proposition \ref{Prop:TrComp4}.

\begin{proof}[Proof of Proposition \ref{Prop:TrComp4}]
Let
\begin{equation}\label{Er}
E_r=\left(\frac 1{r^k}\delta_{k\ell}\right)_{k,\ell\ge 1},
\end{equation}
so that $B_r=E_rBE_r$. We first estimate
\begin{equation*}
    |(B_rB_r^*)^i_{k\ell}-(P_nBP_nB^*P_n)^i_{k\ell}|,
\end{equation*}
where $n=\Big[\frac{1}{r-1}$\Big]. In Lemma \ref{Lem:Tr}, we take $D_1=B_r$, $D_2=D_1^*$, $D_3=P_nBP_n$, and $D_4=D_3^*$. That the condition \eqref{Tr1} is satisfied follows from \eqref{bklest}. Then
\begin{align*}
    |(D_1)_{\mu\nu}-(D_3)_{\mu\nu}|&\le |(E_r(I-P_n)+(E_r-I)P_nBE_r)_{\mu\nu}|\\
    &+|(P_nB(E_r(I-P_n)+(E_r-I)P_n))_{\mu\nu}|\le Cd_{\mu\nu}f_\rho(\mu,\nu),
\end{align*}
where
\begin{equation}\label{dmunu}
    d_{\mu\nu}=\frac{1(\mu>n)}{r^\mu}+1(\mu\le n)(r-1)\mu+\frac{1(\nu>n)}{r^\nu}+1(\nu\le n)(r-1)\nu.
\end{equation}
Here, we used the definition \eqref{Er}, the estimate \eqref{bklest}, and the fact that $1-1/r^\mu\le C\mu(r-1)$ if $\mu\le n$. It follows from \eqref{Trkl} that
\begin{align}\label{BrBPnstep}
|(B_rB_r^*)^i_{k\ell}-(P_nBP_nB^*P_n)^i_{k\ell}|&\le C^i\sum_{j=0}^{i-1}\sum_{\mu,\nu}^\infty\Bigg[F_{\rho,2j}(k,\mu)d_{\mu\nu}f_\rho(\mu,\nu)F_{\rho,2(i-1-j)+1}(\nu,\ell)\notag\\
&+F_{\rho,2j+1}(k,\mu)d_{\mu\nu}f_\rho(\mu,\nu)F_{\rho,2(i-1-j)}(\nu,\ell)\Bigg].   
\end{align}
Using \eqref{dmunu} and \eqref{Frhok}, it follows from this estimate that
\begin{align*}
&|\Tr(B_rB_r^*)^i-\Tr(P_nBP_nB^*)^i|
\le C^i\sum_{\mu,\nu=1}^\infty (1(\mu\le n)(r-1)\mu+\frac{1(\mu>n)}{r^{2\mu}})f_\rho(\mu,\nu)F_{\rho,2i-1}(\nu,\mu)\\
&\le C^i\sum_{\mu=1}^\infty (1(\mu\le n)(r-1)\mu+\frac{1(\mu>n)}{r^{2\mu}})F_{\rho,2i}(\mu,\mu)
\le C^i\sum_{\mu=1}^\infty (1(\mu\le n)(r-1)\mu+\frac{1(\mu>n)}{r^{2\mu}})\frac 1\mu.
\end{align*}
Now,
\begin{equation*}
\sum_{\mu=1}^\infty 1(\mu\le n)(r-1)=n(r-1)\le 1, \quad \sum_{\mu=1}^\infty \frac{1(\mu>n)}{\mu r^{2\mu}}\le \frac 1n\sum_{\mu=n+1}^\infty\frac 1{r^{2\mu}}\le\frac 1{n(r-1)}\le 2,
\end{equation*}
and we have proved the estimate
\begin{equation}\label{Trb1}
|\Tr(B_rB_r^2)^i-\Tr(P_nBP_nB^*)^i|\le C^i.
\end{equation}

Also, by \eqref{BrBPnstep} and \eqref{Frhok5},
\begin{align}\label{BrBPnStep2}
&\sum_{k,\ell=1}^\infty \frac{1}{\sqrt{k\ell}}|(B_rB_r^*)^i_{k\ell}-(P_nBP_nB^*P_n)^i_{k\ell}| \notag\\&\le C^i\sum_{j=0}^{i-1}\sum_{\mu,\nu=1}^\infty\left[\left(\sum_{k=1}^\infty \frac 1{\sqrt{k}}F_{\rho,2j}(k,\mu)\right)d_{\mu\nu}f_\rho(\mu,\nu)\left(\sum_{\ell=1}^\infty \frac 1{\sqrt{\ell}}F_{\rho,2(i-1-j)+1}(\ell,\nu)\right)\right.\notag\\
&+\left.\left(\sum_{k=1}^\infty \frac 1{\sqrt{k}}F_{\rho,2j+1}(k,\mu)\right)d_{\mu\nu}f_\rho(\mu,\nu)\left(\sum_{\ell=1}^\infty \frac 1{\sqrt{\ell}}F_{\rho,2(i-1-j)}(\ell,\nu)\right)\right]
\le C^i\sum_{\mu,\nu=1}^\infty\frac 1{\sqrt{\mu\nu}}d_{\mu\nu}f_\rho(\mu,\nu)\notag\\
&\le C^i\sum_{\mu,\nu=1}^\infty\frac 1{(\mu\nu)^{1-\rho}}\Bigg(1(\mu\le n)(r-1)\mu+\frac{1(\mu>n)}{r^\mu}\Bigg)\frac{1}{(\mu^2+\nu^2)^\rho}\le C^i.
\end{align}

Next, we want to estimate
\begin{equation*}
    |(P_nBB^*P_n)^i_{k\ell}-(P_nBP_nB^*P_n)^i_{k\ell}|.
\end{equation*}
We take $D_1=P_nB$, $D_2=D_1^*$, $D_3=P_nBP_n$, and $D_4=D_3^*$ in Lemma \ref{Lem:Tr}. Write
\begin{equation*}
    \tilde{d}_{\mu\nu}=1(\mu\le n)1(\nu>n)+1(\nu\le n)1(\mu>n).
\end{equation*}
Note that
\begin{equation*}
    |(D_1)_{\mu\nu}-(D_3)_{\mu\nu}|=|(P_nB(I-P_n))_{\mu\nu}|\le C1(\mu\le n)f_\rho(\mu,\nu)1(\nu>n),
\end{equation*}
so we get the estimate
\begin{align}\label{BrBPnstep3}
|(P_nBB^*P_n)^i_{k\ell}-(P_nBP_nB^*P_n)^i_{k\ell}|
&\le C^i\sum_{j=1}^{i-1}\sum_{\mu,\nu=1}^\infty\left[F_{\rho,2j}(k,\mu)1(\mu\le n)F_\rho(\mu,\nu)1(\nu>n)F_{\rho,2(i-1-j)+1}(\nu,\ell)\right.
\notag\\&+
 \left.F_{\rho,2j+1}(k,\mu)1(\mu\le n)F_\rho(\mu,\nu)1(\nu>n)F_{\rho,2(i-1-j)}(\nu,\ell) \right].  
\end{align}
Hence, by \eqref{Frhok1},
\begin{align*}
    &|\Tr(P_nBB^*P_n)^i-\Tr(P_nBP_nB^*P_N)^i|\\
    &\le C^i \sum_{\mu,\nu=1}^\infty\tilde{d}_{\mu\nu}f_{\rho}(\mu,\nu)
    F_{\rho,2i-1}(\mu,\nu)\le C^i\sum_{\mu=1}^n\sum_{\nu=n+1}^\infty \frac 1{\mu\nu}\left(\frac{\mu\nu}{\mu^2+\nu^2}\right)^\rho \le C^i,
\end{align*}
which together with \eqref{Trb1} proves \eqref{TrComp4}. 

Using \eqref{Frhok5} in \eqref{BrBPnstep3}, we also get the estimate
\begin{equation*}
    \sum_{k,\ell=1}^\infty\frac 1{\sqrt{k\ell}}|(P_nBB^*P_n)^i_{k\ell}-(P_nBP_nB^*P_n)^i_{k\ell}|\le C^i\sum_{\mu,\nu=1}^\infty \frac 1{\sqrt{\mu\nu}}\tilde{d}_{\mu\nu}f_\rho(\mu,\nu)\le C^i.
\end{equation*}
This estimate and \eqref{BrBPnStep2} proves the following estimate that will be used in the next lemma.
\begin{equation}\label{BrBPn}
\sum_{k,\ell=1}^\infty\frac 1{\sqrt{k\ell}} |(B_rB_r^*)^i_{k\ell}-(P_nBB^*P_n)^i_{k\ell}|\le C^i.
\end{equation}
\end{proof}

\begin{lemma}\label{lem:Corsumest1}
   We have the estimate 
   \begin{equation}\label{Corsumest2}
      \sum_{k,\ell=1}^\infty\frac 1{\sqrt{k\ell}} |(B_rB_r^*)^i_{k\ell}-
      \sum_{p=1}^mz_p^k\bar{z}_p^\ell(P_nK_p^2P_n)_{k\ell}^i|\le C^i.
   \end{equation}
\end{lemma}
\begin{proof}
    This follows by combining \eqref{TrComp1:2} with \eqref{BrBPn}.
\end{proof}

\subsection{Proof of lemmas for Theorem \ref{Thm:PommEn}}\label{Subsec:ProofPommEn}

Here we prove the two lemmas needed for the proof of Theorem \ref{Thm:PommEn}.
Recall that the Grunsky operator for the curve $\eta_r$, $r>1$, is given by \eqref{Br}, i.e. we have the matrix elements $(b_r)_{k\ell}=r^{-(k+\ell)}b_{kl}$. Let $\mathcal{K}_r$ and ${\bf d}_r$ be as in \eqref{Kop} and \eqref{dvec} with $B$ replaced by $B_r$. Since $\eta_r$ is an analytic curve the operator $B_r$ is Hilbert-Schmidt and we have the identity
\begin{equation}\label{EPfor}
  I^F(\eta_r)=2\re\sum_{i=0}^\infty({\bf d}_r^*(B_rB_r^*)^i{\bf d}_r-{\bf d}_r^*(B_rB_r^*)^i B_r\bar{{\bf d}}_r )
\end{equation}
by Lemma \ref{Lem:KtoB}. 
\begin{proof}[Proof of Lemma~\ref{Lem:dapprox}]
    Recall \eqref{rkl}, \eqref{Kkl} and the estimate \eqref{rklEst}. Write
    \begin{equation}\label{etar}
     (\zeta_r)_k=\frac 1{r^k}\sum_{\ell=1}^{k-1}\sqrt{\frac{k}{\ell(k-\ell)}}K(k,\ell).   
    \end{equation}
We will first show that
\begin{equation}\label{dretar}
 |({\bf d}_r)_k-(\zeta_r)_k| \le \frac{C}{r^kk^{(1+\rho)/2}},
\end{equation}
for some constant $C$. By \eqref{Kkl}, \eqref{rkl}, \eqref{rklEst}, and \eqref{dvec}
\begin{align*}
  |({\bf d}_r)_k-(\zeta_r)_k| &\le \frac{C}{r^k}\sum_{\ell=1}^{k-1} \sqrt{\frac{k}{\ell(k-\ell)}}\frac{\sqrt{(k-\ell)\ell}}k\left(\frac 1{(k-\ell)^\rho\ell}+
\frac 1{(k-\ell)\ell^\rho}\right)\\
&=\frac{C}{r^k\sqrt{k}}\sum_{\ell=1}^{k-1}\left(\frac 1{(k-\ell)^\rho\ell}+
\frac 1{(k-\ell)\ell^\rho}\right)\le \frac{C}{r^kk^{(1+\rho)/2}},
\end{align*}
and we have shown \eqref{dretar}. 

Inserting the formula \eqref{KpHp} into \eqref{etar} gives
\begin{equation}\label{etarfor}
(\zeta_r)_k=\frac 1{r^k}\sum_{p=1}^mz_p^k\sum_{\ell=1}^{k-1}\frac 1{\ell(k-\ell)}H_p(\log(\frac{k}{\ell}-1)).
\end{equation}
The following estimates can be used to finish the proof of the lemma, and will proved below.
\begin{equation}\label{Hpsumint}
    \left|\sum_{\ell=1}^{k-1}\frac 1{\ell(k-\ell)}H_p \left(\log(\frac{k}{\ell}-1) \right)-\int_1^{k-1}\frac 1{x(k-x)}H_p \left(\log(\frac kx-1)\right)dx\right|\le \frac C{k^{1+\rho}},
\end{equation}
and 
\begin{equation}\label{Hpintint}
\left| \int_1^{k-1}\frac 1{x(k-x)}H_p\left(\log(\frac kx-1)\right)   -\frac 1k\int_0^\infty H_p(\log s)\frac{ds}s\right|\le \frac C{k^{1+\rho}},
\end{equation}
for $k\ge 2$. Note that,
\begin{equation*}
 \int_0^\infty H_p(\log s)\frac{ds}s=\int_{\R}H_p(x)\,dx=\hat{H}_p(0)=\gamma_p-1,   
\end{equation*}
by \eqref{HpFT}. Combining this with \eqref{dretar}, \eqref{etarfor}, \eqref{Hpsumint} and \eqref{Hpintint} proves the lemma.

We turn now to the proof of the estimate \eqref{Hpsumint}. 
An integration by parts in a Stieltjes integral gives
\begin{equation*}
\sum_{k=1}^\infty f(k)=\int_1^\infty f(u)\,du+f(1)+\int_1^\infty f'(u)(u-[u])\,du,
\end{equation*}
where $[u]$ is the integer part of $u$. This implies the estimate
\begin{equation}\label{sumint}
\left|\sum_{k=1}^\infty f(k)-\int_1^\infty f(u)\,du\right|\le |f(1)|+\int_1^\infty|f'(u)|\,du.
\end{equation}
It follows from this estimate that left side of \eqref{Hpsumint} is bounded by
\begin{equation*}
    \frac 1{k-1}|H_p(\log(k-1))|+\int_1^{k-1}\left|\frac{d}{dx}\frac 1{x(k-x)}H_p(\log(\frac kx-1))\right|dx.
\end{equation*}
It follows from \eqref{Hpt} that
\begin{equation}\label{Hpest}
    |H_p(t)|\le e^{-\gamma_pt},
\end{equation}
since $\cos\pi\gamma_p<1$. Hence,
\begin{equation*}
 \frac 1{k-1}|H_p(\log(k-1))|\le \frac{C}{k^{1+\gamma_p}}\le\frac{C}{k^{1+\rho}},   
\end{equation*}
since $\gamma_p\ge \rho$. If we make the change of variables $x=kt$ and then $s=1/t-1$, some computation gives
\begin{equation*}
 \int_1^{k-1}\left|\frac{d}{dx}\frac 1{x(k-x)}H_p(\log(\frac kx-1))\right|dx=\frac 1{k^2}\int_{1/(k-1)}^{k-1}\left|\frac{d}{ds}\frac{(s+1)^2}{s}H_p(\log s)\right|ds.
\end{equation*}
We can now use \eqref{Frhok6} and \eqref{Frhok7} and the definition of $F_{\rho,1}=f_\rho$, to see that this is bounded by
\begin{equation*}
    \frac 1{k^2}\int_{1/(k-1)}^{k-1}\frac{1+s^2}{s^2}\left(\frac{s}{s^2+1}\right)^\rho+\frac{(1+s)^2}{s^2}\left(\frac{s}{s^2+1}\right)^\rho ds\le \frac{C}{k^{1+\rho}},
\end{equation*}
and we have proved \eqref{Hpsumint}.

The same change of variables as above in the first integral in the left side of \eqref{Hpintint} gives
\begin{equation*}
    \frac 1k\int_{1/(k-1)}^{k-1} H_p(\log s)\frac{ds}{s}=\frac 1k\int_0^\infty H_p(\log s)\frac{ds}{s}-\frac 2k\int_{k-1}^\infty H_p(\log s)\frac{ds}{s}.
\end{equation*}
Now, by \eqref{Hpest},
\begin{equation*}
    \left|\int_{k-1}^\infty H_p(\log s)\frac{ds}{s}\right|\le\frac 1k\int_{\log(k-1)}^\infty |H_p(t)|\,dt\le\frac Ck\int_{\log(k-1)}^\infty e^{-\gamma_pt}dt\le\frac{C}{k^{1+\rho}},
\end{equation*}
since $\gamma_p\ge\rho>0$, and we have demonstrated the estimate \eqref{Hpintint}.
\end{proof}

We now have the tools to prove Lemma~\ref{Lem:Brdas}.
\begin{proof}[Proof of Lemma~\ref{Lem:Brdas}]
    It follows from \eqref{xikr} and \eqref{drxir} that
    \begin{equation}\label{drest}
        |(d_r)_k|\le \frac{C}{\sqrt{k}},
    \end{equation}
    for all $k\ge 1$, $r\ge 1$. Write
    \begin{equation}\label{Mkl}
     M^{(i)}_{k\ell}=\sum_{p=1}^mz_p^k\bar{z}_p^\ell\int_{[0,\infty)^{2i-1}}
     \prod_{j=1}^{2i}K_p(u_{j-1},u_j)\,du_1\dots du_{2i-1},
    \end{equation}
with $u_0=k, u_{2i}=\ell$. We have the bound
\begin{equation}\label{BBMest}
    \sum_{k,\ell=1}^\infty \frac 1{\sqrt{k\ell}}|(B_rB_r^*)^i_{k\ell}-M^{(i)}_{k\ell}1(k,\ell\le n)|\le C^i,
\end{equation}
for all $i\ge 1$. To see this write
\begin{align*}
&\sum_{k,\ell=1}^\infty \frac 1{\sqrt{k\ell}}\Big|(B_rB_r^*)^i_{k\ell}-M^{(i)}_{k\ell}1(k,\ell\le n)\Big|\le \sum_{k,\ell=1}^\infty \frac 1{\sqrt{k\ell}}\left\{\Big| (B_rB_r^*)^i_{k\ell}-\sum_{p=1}^mz_p^k\bar{z}_p^\ell
(P_nK_p^2P_n)^i_{k\ell}\Big|\right.\\
&\left.  \sum_{p=1}^m\Big|(P_nK_p^2P_n)^i_{k\ell}-(P_nK_p^{2i}P_n)_{k\ell}\Big|+
\sum_{p=1}^m\Big|(P_nK_p^{2i}P_n)_{k\ell}-\int_{[1,\infty)^{2i-1}}
\prod_{j=1}^{2i}K_p(u_{j-1},u_j)\,du_1\dots du_{2i-1}\Big|\right.\\
&\left.+ \sum_{p=1}^m\Big|\Big(\int_{[1,\infty)^{2i-1}}-\int_{[0,\infty)^{2i-1}}\Big) \prod_{j=1}^{2i}K_p(u_{j-1},u_j)\,du_1\dots du_{2i-1}\Big|\right\}.
\end{align*}
The bound \eqref{BBMest} now follows from \eqref{Sumextkl}, \eqref{SumIntkl}, \eqref{ChIntkl}, and \eqref{Corsumest2}.

Combining \eqref{drest} with \eqref{BBMest} this now gives the estimate
\begin{equation*}
    \left|{\bf d}_r^*(B_rB_r^*)^iB_r{\bf d}_r-\sum_{k,\ell=1}^n(\overline{{\bf d}_r})_kM^{(i)}_{k\ell}({\bf d}_r)_\ell\right|\le C^i.
\end{equation*}
By \eqref{Kpbound1} and \eqref{Frhok},
\begin{equation}\label{Mklbound}
    |M^{(i)}_{k\ell}|\le CF_{\rho,2i}(k,\ell),
\end{equation}
for $i,k,\ell\ge 1$. Hence, using \eqref{drxir} and \eqref{Frhok5}, we obtain the bound
\begin{equation*}
    \left|\sum_{k,\ell=1}^n(\overline{{\bf d}_r})_kM^{(i)}_{k\ell}({\bf d}_r)_\ell-\sum_{k,\ell=1}^n(\overline{\xi_r})_kM^{(i)}_{k\ell}(\xi_r)_\ell\right|\le C\sum_{k,\ell=1}^\infty \frac 1{k^{(1+\rho)/2}}F_{\rho,2i}(k,\ell)\frac 1{\ell^{1/2}}\le C^i\sum_{k=1}^\infty \frac 1{k^{1+\rho/2}}\le C^i.
\end{equation*}

Let 
\begin{equation*}
    \xi_k=\frac 1{\sqrt{k}}\sum_{p=1}^m z_p^k(\gamma_p-1),
\end{equation*}
so that
\begin{equation*}
    |(\xi_r)_k-\xi_k|\le\frac{C}{\sqrt{k}}|r^k-1|\le \frac{C\sqrt{k}}{n}
\end{equation*}
for $k\le n$ and $n=[\frac{1}{r-1}]$. Hence, by \eqref{Mklbound} and \eqref{Frhok5},
\begin{align*}
 \sum_{k,\ell=1}^n\overline{\xi}_kM^{(i)}_{k\ell}\xi_\ell &=
 \sum_{k,\ell=1}^n\sum_{p,q_1,q_2=1}^m\frac{(\gamma_{q_1}-1)(\gamma_{q_2}-1)}{\sqrt{k\ell}}\overline{z}_{q_1}^k z_{q_2}^\ell z_p^k z_p^\ell\int_{[0,\infty)^{2i-1}}K_p(k,u_1)\dots K_p(u_{2i-1},\ell)\,du\\
 &=\sum_{k,\ell=1}^n\sum_{p=1}^m\frac{(\gamma_p-1)^2}{\sqrt{k\ell}}\int_{[0,\infty)^{2i-1}}K_p(k,u_1)\dots K_p(u_{2i-1},\ell)\,du+\Sigma',
\end{align*}
which defines $\Sigma'$.
We can use the estimate \eqref{Sumparts} to see that if $p\neq q_1$, then
\begin{align*}
    \left|\sum_{k=1}^n\frac{(\overline{z_{q_1}}z_p)^k}{\sqrt{k}}K_p(k,u)\right|&\le C\left[\frac 1{\sqrt{n}}|K_p(n,u|)+\int_1^\infty\left|\frac{\partial}{\partial v}\left(\frac 1{\sqrt{v}}K_p(v,u)\right)\right|\,dv\right]\\
    &\le\left[\frac 1{\sqrt{n}}f_\rho(n,u)+\int_1^\infty\frac 1{v^{3/2} }f_\rho(v,u)\,dv\right],
\end{align*}
where we also used \eqref{Kpbound1} and \eqref{dKp}. 
From this bound we conclude that
\begin{equation}
    |\Sigma'|\le C\sum_{\ell=1}^n\left[\frac 1{\sqrt{n}}F_{\rho,2i}(n,\ell)\frac 1{\sqrt{\ell}}+\int_1^\infty\frac 1{v^{3/2}}F_{\rho,2i}(v,\ell)\frac 1{\sqrt{\ell}}\,dv\right]
    \le C^i(\frac 1n+\int_1^\infty\frac 1{v^{2}})\le C^i,
\end{equation}
by \eqref{Frhok5}. The identity \eqref{KpHp} and a change of variables gives
\begin{equation*}
 \int_{[0,\infty)^{2i-1}}K_p(k,u_1)\dots K_p(u_{2i-1},\ell)\,du=\frac 1{\sqrt{k\ell}}H_p^{*2i}(\log\frac k\ell),   
\end{equation*}
and we have shown that there is a constant $C$ such that
\begin{equation*}
    \left|{\bf d}_r^*(B_rB_r^*)^i{\bf d}_r-\sum_{p=1}^m(\gamma_p-1)^2\sum_{k,\ell=1}^n\frac 1{\sqrt{k\ell}}H_p^{*2i}(\log\frac k\ell)\right|\le C^i,
\end{equation*}
for all $i\ge 1$. To prove \eqref{Brdas1} it remains to prove that
\begin{equation}\label{Brdasstep}
    \left|\sum_{k,\ell=1}^n\frac 1{\sqrt{k\ell}}H_p^{*2i}(\log\frac k\ell)-(\gamma_p-1)^{2i}\log n\right|\le C^i.
\end{equation}
Note that
\begin{equation*}
    \int_0^ndu\int_0^\infty dv \frac 1{uv}H_p^{*2i}(\log\frac uv)=\int_0^{\log n}ds\int_\R dt H_p^{*2i}(t-s)=(\log n)\int_\R H^{*2i}(s)\,ds=(\gamma_p-1)^{2i}\log n,
\end{equation*}
so it is enough to show the inequality
\begin{equation}\label{Brdasstep2}
    \left|\sum_{k,\ell=1}^n\frac 1{\sqrt{k\ell}}H_p^{*2i}(\log\frac k\ell)-
\int_0^ndu\int_0^\infty dv \frac 1{uv}H_p^{*2i}(\log\frac uv)\right|\le C^i.
\end{equation}
Write
\begin{equation}\label{Tnv}
    T_n(v)=\int_1^n\frac 1uH_p^{*2i}(\log\frac uv)\,du.
\end{equation}
Then, by \eqref{sumint},
\begin{equation*}
    \left|\sum_{\ell=1}^nH_p^{*2i}(\log\frac \ell k)-T_n(k)\right|\le |H_p^{*2i}(\log k)|+\int_1^\infty \frac 1{u^2}\left|H_p^{*2i}(\log\frac uk)\right|+\frac 1u\left|\frac{d}{du}H_p^{*2i}(\log\frac uk)\right|\,du.
\end{equation*}
It follows from \eqref{Frhok6} and \eqref{Frhok7} that this is bounded by
\begin{equation*}
    C^i\sqrt{k}\left(F_{\rho,2i}(k,1)+\int_1^\infty\frac 1{u^{3/2}}F_{\rho,2i}(u,k)\,du\right).
\end{equation*}
Consequently, using also \eqref{sumint} and \eqref{Frhok4},
\begin{align*}
 &\left|\sum_{k,\ell=1}^n\frac 1{\sqrt{k\ell}}H_p^{*2i}(\log\frac k\ell)-
\int_1^ndu\int_1^n dv \frac 1{uv}H_p^{*2i}(\log\frac uv)\right| \\
&=  \left|\sum_{k=1}^n\frac 1k\left(\sum_{\ell=1}^n\frac 1\ell H_p^{*2i}(\log\frac \ell k)-T_n(k)\right)+\sum_{k=1}^n\frac{T_n(k)}k-\int_1^n\frac{T_n(v)}vdv\right|\\
&\le C^i\sum_{k=1}^n\left(F_{\rho,2i}(k,1)+\int_1^\infty\frac 1{u^{3/2}}F_{\rho,2i}(u,k)\,du\right)+|T_n(1)|+\int_1^n \left|\frac{d}{dv}\frac{T_n(v)}v\right|\,dv\\
&\le C^i+|T_n(1)|+\int_1^n\frac{|T_n(v)|}{v^2}+\frac{|T_n'(v)|}{v}dv.
\end{align*}
It follows from \eqref{Frhok2} that
\begin{equation*}
   |T_n(1)|=\left|\int_1^n\frac 1uH_p^{*2i}(\log u)\,du\right|=\left|\int_0^{\log n}H_p^{*2i}(t)\,dt\right|\le C^i\int_0^\infty\min(1,\frac 1{t^3})\,dt\le C^i.
\end{equation*}
Also,
\begin{equation*}
 |T_n'(v)|=\left|\int_1^n\frac 1u\frac{\partial}{\partial v}H_p^{*2i}(\log\frac vu)\,du\right|\le C^i\int_1^n\frac 1{\sqrt{uv}}F_{\rho,2i}(u,v)\,du,   
\end{equation*}
and
\begin{equation*}
 |T_n(v)|\le\int_1^n\frac 1u|H_p^{*2i}(\log\frac uv)|\,du\le C^i\int_1^n\frac {\sqrt{v}}{u^{3/2}}F_{\rho,2i}(u,v)\,du\le C^i,,   
\end{equation*}
Hence, by \eqref{Frhok4},
\begin{equation*}
 \int_1^n\frac{|T_n(v)|}{v^2}+\frac{|T_n'(v)|}{v}dv\le C^i\left(1+\int_1^n\frac 1{v^{3/2}}\left(\int_1^n\frac 1{\sqrt{u}}F_{\rho,2i}(u,v)\,du\right)dv\right)\le C^i,  
\end{equation*}
and we have proved the estimate \eqref{Brdasstep2}. It remains to estimate 
\begin{align*}
    &\left|\int_1^ndu\int_0^\infty dv \frac 1{uv}H_p^{*2i}(\log\frac uv)-\int_1^ndu\int_1^n dv \frac 1{uv}H_p^{*2i}(\log\frac uv)\right|\\
    &\le \int_0^{\log n}ds\int_{-\infty}^0dt|H_p^{*2i}(s-t)|+
    \int_0^{\log n}ds\int_{\log n}^0dt|H_p^{*2i}(s-t)|\\
    &=2\int_0^{\log n}ds\int_{0}^\infty dt|H_p^{*2i}(s+t)|\le C^i\int_0^{\log n}ds\int_{0}^\infty dt\min(1,\frac 1{(t+s)^3}\le C^i,
\end{align*}
where we used \eqref{Frhok2}. We have proved \eqref{Brdas1}. The proof of \eqref{Brdas2} is analogous.
\end{proof}

\section{Proofs of technical lemmas}\label{Sec:lemmas}
In this section we will give the proofs of various technical lemmas that we used above. We start with the following inequality.

\begin{lemma}\label{Lem:uvquot}
Given $\alpha,\beta\in\R\setminus\{0\}$ with $0<|\beta|<2$, then, provided that $\delta$ is small enough, there is a constant $C$ so that
we have the bound
\begin{equation}\label{uvquot}
\left|\frac{u^\alpha-v^\alpha}{u^\beta-v^\beta}\right|\le 
\begin{cases}
    &C\max(|u|, |v|)^{\alpha-\beta},\quad \text{if $\alpha>0$}\\
    & C \frac{\min(|u|, |v|)^{\alpha}}{\max(|u|, |v|)^{\beta}},
    \quad\text{if $\alpha \le 0$,}
\end{cases}
\end{equation}
for all $u,v\in L_\delta$.
\end{lemma}

\begin{proof}
First we suppose that $\alpha>0$. Also 
assume that $u,v\in L_\delta(1)$; the case $u,v\in L_\delta(-1)$ is follows by conjugation. By symmetry, we can assume that $|v|\le |u|$. Then $v=tu$, where
$0\le t\le 1$, and thus
\begin{equation*}
\left|\frac{u^\alpha-v^\alpha}{u^\beta-v^\beta}\right|=\frac{|u|^\alpha}{|u|^\beta}\frac{1-t^\alpha}{1-t^\beta}\le C|u|^{\alpha-\beta}=C\max(|u|, |v|)^{\alpha-\beta},
\end{equation*}
since $t\mapsto (1-t^\alpha)/(1-t^\beta)$ is bounded in $[0,1]$. Assume next that $u\in L_\delta(1)$ and $v\in L_\delta(-1)$. Then, $u=|u|\lambda$, $v=|v|\lambda^{-1}$,
where $\lambda$ is given by \eqref{lambda}. Assume also that $|v|\le |u|$, $|v|=t|u|$, $0\le t\le 1$. Then,
\begin{equation}\label{uvquotstep}
\left|\frac{u^\alpha-v^\alpha}{u^\beta-v^\beta}\right|=|u|^{\alpha-\beta}\left|\frac{1-t^\alpha\lambda^{-2\alpha}}{1-t^\beta\lambda^{-2\beta}}\right|.
\end{equation}
We have the identity
\begin{equation*}
|1-t^\beta\lambda^{-2\beta}|=1+t^{2\beta}-2t^\beta\cos2\beta(\phi+\pi/2).
\end{equation*}
If $0<|\beta|<2$, then $-2\pi<2\beta(\phi+\pi/2)<2\pi$ if $\delta$ and hence $\phi$ is small enough. Thus, $c=\cos2\beta(\phi+\pi/2)<1$, and we see that
$1+t^{2\beta}-2ct^\beta\ge 1-c$ for $0\le t\le 1$. Hence, by \eqref{uvquotstep},
\begin{equation*}
\left|\frac{u^\alpha-v^\alpha}{u^\beta-v^\beta}\right|\le \frac{2}{\sqrt{1-c}}|u|^{\alpha-\beta}=C\max(|u|, |v|)^{\alpha-\beta}.
\end{equation*}
The case $|v|\ge |u|$ is analogous, and the case $u\in L_\delta(-1)$, $v\in L_\delta(1)$ follows by conjugation.

Consider now the case $\alpha<0$. If $u,v\in L_\delta(\pm 1)$, $|v|\le |u|$,
$v=tu$, $0/le t\le 1$, then
\begin{align*}
\left|\frac{u^\alpha-v^\alpha}{u^\beta-v^\beta}\right|=|u|^{\alpha-\beta}t^\alpha\left|\frac{1-t^{-\alpha}}{1-t^\beta}\right|\le C|u|^{\alpha-\beta}\Big|\frac vu\Big|^\alpha
=C\frac{|v|^{\alpha}}{|u|^\beta}=
C\frac{\min(|u|, |v|)^{\alpha}}{\max(|u|, |v|)^{\beta}}.
\end{align*}
If $u\in L_\delta(1)$, $v\in L_\delta(-1)$, $|v|\le |u|$, $|v|=t|u|$, or the other way around, then
as above
\begin{align*}
\left|\frac{u^\alpha-v^\alpha}{u^\beta-v^\beta}\right|=|u|^{\alpha-\beta}t^\alpha\left|\frac{1-t^{-\alpha}\lambda^{2\alpha}}{1-t^\beta\lambda^{-2\beta}}\right|\le \frac 2{\sqrt{1-c}}|u|^{\alpha-\beta}\Big|\frac vu\Big|^\alpha
\le C\frac{\min(|u|, |v|)^{\alpha}}{\max(|u|, |v|)^{\beta}}.
\end{align*}

\end{proof}

By Lemma \ref{Lem:regularity}, we can write
\begin{equation}\label{psiEp}
g(z_p(1+u))-g(z_p)=\int_0^u\gamma_pAs^{\gamma_p-1}+r_1(s)\,ds=Au^{\gamma_p}+\int_0^u r_1(s)\,ds.
\end{equation}
Define
\begin{align}\label{Ep}
E_p(u)&=\frac 1A\int_0^u r_1(s)\,ds,\\
D_p(u)&=\frac{z_pg'(z_p(1+u))}{\gamma_p Au^{\gamma_p-1}},\label{Dp}
\end{align}
for $u\in  L_\delta$. Recall from \eqref{lambdap} that 
$\lambda_p=\min(2\gamma_p,\gamma_p+1)$.
\begin{lemma}\label{Lem:EpDp}
If $\delta$ is small enough there is a constant $C$ such that
\begin{equation}\label{Epest}
\left|\frac{E_p(u)-E_p(v)}{u^{\gamma_p}-v^{\gamma_p}}\right|\le C\max(|u|,|v|)^{\lambda_p-\gamma_p},
\end{equation}
and
\begin{equation}\label{Dpest}
\left|\frac{D_p(u)-D_p(v)}{u^{\gamma_p}-v^{\gamma_p}}\right|\le C\max(|u|,|v|)^{\lambda_p-2\gamma_p},
\end{equation}
for all $u, v \in  L_\delta$, where $\rho$ is defined in \eqref{rho}.
\end{lemma}

\begin{proof}
By the definition of $E_p$ and \eqref{r12est}, if  $u, v \in  L_\delta(1)$, then
\begin{equation*}
|E_p(u)-E_p(v)|=\left|\frac 1A\int_{|v|}^{|u|}r_1(t\lambda)\lambda \,dt\right|\le C\left|\int_{|v|}^{|u|}t^{\lambda_p-1}\,dt\right|=C|u^{\lambda_p}-v^{\lambda_p}|.
\end{equation*}
The same estimate holds if $u, v \in  L_\delta(-1)$. If $u\in L_\delta(1)$ and $v\in L_\delta(-1)$ (or the other
way around), then as in the proof of Lemma \ref{Lem:uvquot},
\begin{equation*}
|u^{\lambda_p}-v^{\lambda_p}|\ge C\max(|u|,|v|)^{\lambda_p},
\end{equation*}
since $0<\lambda_p<2$ and hence
\begin{align*}
\left|\frac{E_p(u)-E_p(v)}{u^{\lambda_p}-v^{\lambda_p}}\right|&\le C\frac{|E_p(u)-E_p(0)|}{\max(|u|,|v|)^{\lambda_p}}+C\frac{|E_p(v)-E_p(0)|}{\max(|u|,|v|)^{\lambda_p}}\\
&\le C\frac{|u|^{\lambda_p}+|v|^{\lambda_p}}{\max(|u|,|v|)^{\lambda_p}}\le C.
\end{align*}
Hence, for any  $u, v \in  L_\delta$, by Lemma \ref{Lem:uvquot}
\begin{equation*}
\left|\frac{E_p(u)-E_p(v)}{u^{\gamma_p}-v^{\gamma_p}}\right|\le C\left|\frac{u^{\lambda_p}-v^{\lambda_p}}{u^{\gamma_p}-v^{\gamma_p}}\right|
\le C\max(|u|,|v|)^{\lambda_p-\gamma_p}\le C\max(|u|,|v|)^{\rho},
\end{equation*}
since $\lambda_p>0$ and $0<\gamma_p<2$. This proves the estimate \eqref{Epest}.

It follows from \eqref{r1u} and \eqref{r12est} that $D_p(u)\to 1=D_p(0)$ as $u\to 0$. Using \eqref{r1u} and \eqref{r2u}, we see that
\begin{align*}
D_p'(u)&=\frac{z_p^2g''(z_p(1+u))u^{\gamma_p-1}-z_pg'(z_p(1+u))(\gamma_p-1)u^{\gamma_p-2}}{\gamma_pAu^{2\gamma_p-2}}\\
&=\frac 1{\gamma_pA}(r_2(u)u^{1-\gamma_p}-(\gamma_p-1)r_1(u)u^{-\gamma_p}).
\end{align*}
Thus, by \eqref{r12est},
\begin{equation*}
|D'_p(u)|\le C(|u|^{\lambda_p-2}|u|^{1-\gamma_p}+|u|^{\lambda_p-1}|u|^{-\gamma_p}) = C|u|^{\lambda_p-\gamma_p-1}.
\end{equation*}
Since $\lambda_p-\gamma_p>0$, we can now proceed as in the proof of the estimate \eqref{Epest} above to show \eqref{Dpest}.
\end{proof}

We now have the tools that we need to prove the crucial Lemma \ref{Lem:GpRp}.

\begin{proof}[Proof of Lemma \ref{Lem:GpRp}]
Note that \eqref{Gpbound} follows immediately from the definition \eqref{Gp} of $G_p(u,v)$ and Lemma \ref{Lem:uvquot}. By the definitions \eqref{Psi},
\eqref{Ep} and \eqref{Dp},
\begin{equation*}
\Psi(z_p(1+u),z_p(1+v))=\gamma_p\frac{(1+u)u^{\gamma_p-1}D_p(u)-(1+v)v^{\gamma_p-1}D_p(v)}{u^{\gamma_p}-v^{\gamma_p}+E_p(u)-E_p(v)}.
\end{equation*}
Hence, by the definition \eqref{Rp} of $R_p(u,v)$, and some manipulations,
\begin{align}\label{GpRpstep}
R_p(u,v)&=\gamma_p\frac{(1+u)u^{\gamma_p-1}D_p(u)-(1+v)v^{\gamma_p-1}D_p(v)}{u^{\gamma_p}-v^{\gamma_p}+E_p(u)-E_p(v)}
-\gamma_p\frac{u^{\gamma_p-1}-v^{\gamma_p-1}}{u^{\gamma_p}-v^{\gamma_p}}\notag\\
&=\gamma_p+\gamma_p\frac{(1+u)u^{\gamma_p-1}(D_p(u)-D_p(v))+[(1+u)u^{\gamma_p-1}-(1+v)v^{\gamma_p-1}](D_p(v)-D_p(0))}{u^{\gamma_p}-v^{\gamma_p}+E_p(u)-E_p(v)}
\notag\\
&-\gamma_p\frac{[(1+u)u^{\gamma_p-1}-(1+v)v^{\gamma_p-1}](E_p(u)-E_p(v))}{(u^{\gamma_p}-v^{\gamma_p}+E_p(u)-E_p(v))(u^{\gamma_p}-v^{\gamma_p})},
\end{align}
since $D_p(0)=1$. Introduce the notation
\begin{equation*}
\tilde{D}_p(u,v)=\gamma_p\frac{D_p(u)-D_p(v)}{u^{\gamma_p}-v^{\gamma_p}},\quad \tilde{E}_p(u,v)=\frac{E_p(u)-E_p(v)}{u^{\gamma_p}-v^{\gamma_p}}.
\end{equation*}
From \eqref{GpRpstep} we see that we can write $R_p(u,v)$ as
\begin{equation}\label{Rpformula}
R_p(u,v)=\gamma_p+\frac{(1+u)u^{\gamma_p-1}\tilde{D}_p(u,v)+(\gamma_p+G_p(u,v))\tilde{D}_p(v,0)v^{\gamma_p}}{1+\tilde{E}_p(u,v)}
-\frac{(\gamma_p+G_p(u,v))\tilde{E}_p(u,v)}{1+\tilde{E}_p(u,v)}.
\end{equation}

It follows from Lemma \ref{Lem:EpDp} that
\begin{equation}
    |\tilde{D}_p(u,v)|\le C\max(|u|,|v|)^{\lambda_p-2\gamma_p},
    \quad
     |\tilde{E}_p(u,v)|\le C\max(|u|,|v|)^{\lambda_p-\gamma_p}.
\end{equation}
Since $\lambda_p-\gamma_p=\min(\gamma_p,1)>0$, we see that $|\tilde{E}_p(u,v)|\le 1/2$ if $\delta$ is small enough since the length of $L_\delta(\pm 1)$ is $\sqrt{2\delta-\delta^2}$. Consider first the case when $\gamma_p>1$. Then $\lambda_p-\gamma_p=1$. We can assume that $|u|\ge |v|$ since $R_p(u,v)$ is symmetric in $u$ and $v$. Hence, by \eqref{Rpformula}, $\gamma_p<2$ and \eqref{Gpbound},
\begin{align*}
    |R_p(u,v)|&\le 2+C\Big[(1+|u|)|u|^{\gamma_p-1}|u|^{\lambda_p-2\gamma_p}+
    (2+|u|^{-1})|v|^{\lambda_p-2\gamma_p}|v|^{\gamma_p}+(2+|u|^{-1})|u|^{\lambda_p-\gamma_p}\Big]\\
    &\le C\le C(|u|^{\rho-1}+|v|^{\rho-1}),
\end{align*}
since $\rho\le 1$, $|u|$ and $|v|$ are small.

Next, consider the case when $\gamma_p<1$. Then $\lambda_p-\gamma_p=\gamma_p$.
It follows from \eqref{Rpformula} and $\gamma_p<1$ that
\begin{align*}
    |R_p(u,v)|&\le 1+C\Big[(1+|u|)|u|^{\gamma_p-1}|u|^{\lambda_p-2\gamma_p}+
    (1+\frac{|v|^{\gamma_p-1}}{|u|^{\gamma_p}})|v|^{\lambda_p-2\gamma_p}|v|^{\gamma_p}+(1+\frac{|v|^{\gamma_p-1}}{|u|^{\gamma_p}})|u|^{\lambda_p-\gamma_p}\Big]\\
    &\le C\Big[|u|^{\lambda_p-\gamma_p-1}+\Big|\frac vu\Big|^{\gamma_p}
    |v|^{\lambda_p-\gamma_p-1}+|v|^{\gamma_p-1}\Big]
    \le C(|u|^{\rho-1}+|v|^{\rho-1}),
\end{align*}
since $|u|$ and $|v|$ are small. This proves the lemma.

\end{proof}

Next, we give the
\begin{proof}[Proof of Lemma \ref{Lem:expest}]
If $|z|,|w|\le 1/2$, then $|\log(1+w)-w|\le |w|^2$, and $|e^z-1|\le 2|z|$. From these inequalities it follows that if $|z/k|\le 1/2$, then
\begin{equation*}
\left|\left(1+\frac zk\right)^k-e^z\right|=e^{\re z}\left|e^{k\log(1+z/k)-z}-1\right|\le e^{\re z}\frac{|z|^2}k.
\end{equation*}
Consequently, if $|t/k|\le 1/2$, then
\begin{equation*}
\left|(1+\frac tk\lambda^\tau)^k-e^{t\lambda^\tau}\right|\le 2e^{t\re\lambda}\frac{|z|^2}k
\end{equation*}
for $\tau=\pm 1$, and we see that
\begin{align*}
&\left|(1+\frac tk\lambda^\tau)^{k-1}-e^{t\lambda^\tau}\right|=\frac 1{\left|1+\frac tk\lambda^\tau\right|}\left|(1+\frac tk\lambda^\tau)^k-(1+\frac tk\lambda^\tau)e^{t\lambda^\tau}\right|\\
&\le\frac 4k(t^2+t)e^{t\re\lambda^\tau}=\frac 4k(t^2+t)e^{-t\sin\phi}.
\end{align*}
\end{proof}

We turn now to the proof of Lemma \ref{Lem:Kpbound}.
\begin{proof}[Proof of Lemma \ref{Lem:Kpbound}]
If $\cos\pi\gamma_p\le 0$, then $x^2+y^2-2xy\cos\pi\gamma_p\ge x^2+y^2$, for $x,y>0$.
Note that, 
\begin{equation*}
x^2+y^2-2xy\cos\pi\gamma_p=(1-\cos\pi\gamma_p)(x^2+y^2)+\cos\pi\gamma_p(x-y)^2.
\end{equation*}
Hence, if $\cos\pi\gamma_p>0$, then $x^2+y^2-2xy\cos\pi\gamma_p\ge(1-\cos\pi\gamma_p)(x^2+y^2)$. Consequently, since $0<\gamma_p<2$, it follows from \eqref{Kpkl} that
\begin{equation*}
|K_p(u,v)|\le\frac C{\sqrt{uv}}\frac 1{(u/v)^{\gamma_p}+(v/u)^{\gamma_p}}.
\end{equation*}
Let $u/v=e^t$. Then,
\begin{equation}\label{Kpboundstep}
|K_p(u,v)|\le\frac C{\sqrt{uv}}\frac 1{2\cosh\gamma_pt}.
\end{equation}
Since, $\frac{\partial}{\partial\gamma_p}\cosh\gamma_pt=t\sinh\gamma_pt\ge 0$, we see that the right side of \eqref{Kpboundstep} is decreasing in $\gamma_p$ for
any $t\in\R$, so
\begin{equation*}
|K_p(u,v)|\le\frac C{\sqrt{uv}}\frac{(uv)^\rho}{u^{2\rho}+v^{2\rho}}
\end{equation*}
for all $p$.
The lemma now follows if we can show that $u^{2\rho}+v^{2\rho}\ge (u^2+v^2)^\rho$. By homogeneity, it suffices to prove that $f(t)=1+t^\rho-(1+t)^\rho\ge 0$, 
for $0\le t\le 1$, $0<\rho\le 1$, which follows from $f(0)=0$ and $f'(t)\ge 0$ in $[0,1]$.

By \eqref{rklEst}, we see that in order to prove the bound \eqref{Kpbound2} it is enough to show that
\begin{equation}\label{Kpboundstep2}
\frac{\sqrt{k\ell}}{k+\ell}\left(\frac 1{k^\rho\ell}+\frac 1{k\ell^\rho}\right)\le \frac C{\sqrt{k\ell}}\left(\frac{k\ell}{k^2+\ell^2}\right)^\rho
\end{equation}
for some constant $C$, $k,\ell\ge 1$. By symmetry we can assume that $k\ge \ell\ge 1$; let $\ell=tk$, with $0<t\le 1$. Then \eqref{Kpboundstep2} becomes
\begin{equation*}
\frac 1{\ell^\rho}\frac{(1+t^2)^\rho}{1+t}(1+t^{1-\rho})\le C,
\end{equation*}
which holds since $\ell\ge 1$, $0<\rho\le 1$ and $0<t\le 1$.
\end{proof}

We will now prove the basic lemma involving $f_\rho$.

\begin{proof}[Proof of Lemma \ref{Lem:frhobounds}]
Since
\begin{equation*}
    \frac{\partial}{\partial u}\log f_\rho(u,v)=\frac{\rho-1/2}{u}-
    \frac{2\rho u}{u^2+v^2},
\end{equation*}
we see that \eqref{dfrho} holds.

It follows from \eqref{sumint} that
\begin{align*}
&\sum_{k=1}^\infty f_\rho(v_1,k)f_\rho(k,v_2)\le f_\rho(v_1,1)f_\rho(1,v_2)+\int_1^\infty  \left|\frac{\partial}{\partial u}f_\rho(v_1,u)f_\rho(u,v_2)\right|\,du\\
&\le f_\rho(v_1,1)f_\rho(1,v_2)+\int_1^\infty\left|\frac{\partial f_\rho}{\partial u}(v_1,u)\right|f_\rho(u,v_2)\,du
+\int_1^\infty f_\rho(v_1,u)\left|\frac{\partial f_\rho}{\partial u}(u,v_2)\right|\,du.
\end{align*}
Since,
\begin{equation*}
\frac 1{f_\rho(u,v)}\frac{\partial f_\rho}{\partial u}(u,v)=\frac{\rho-1/2}u-\frac{2\rho u}{u^2+v^2},
\end{equation*}
we see that
\begin{equation*}
\left|\frac{\partial f_\rho}{\partial u}(u,v)\right|\le (3\rho+1/2)\frac 1u f_\rho(u,v).
\end{equation*}
Note that
\begin{align*}
&\int_1^\infty f_\rho(v_1,u)f_\rho(u,v_2)\,du=\frac 1{\sqrt{v_1v_2}}\int_1^\infty\frac 1u\left(\frac{uv_1}{u^2+v_1^2}\right)^\rho\left(\frac{uv_2}{u^2+v_2^2}\right)^\rho\,du\\
&\ge \int_1^2\frac 12\left(\frac{v_1}{4+v_1^2}\right)^\rho\left(\frac{v_2}{4+v_2^2}\right)^\rho\,du\ge \frac 1{2\cdot 4^{2\rho}}f_\rho(v_1,1)f_\rho(1,v_2).
\end{align*}
These inequalities prove the first inequality in \eqref{frhob1}.

In order to prove the second inequality in \eqref{frhob1}, we first show that
\begin{equation}\label{frhostep1}
\int_1^\infty f_\rho(v_1,u)f_\rho(u,v_2)\,du\le Cf_\rho(v_1,v_2)(\left|\log\frac{v_2}{v_1}\right|+1).
\end{equation}
To see that this implies the second inequality in \eqref{frhob1} note that
\begin{equation*}
\frac{f_\rho(v_1,v_2)\left(\log\left|\frac{v_2}{v_1}\right|+1\right)}{f_{\rho/2}(v_1,v_2)}=\left(\frac{v_1v_2}{v_1^2+v_2^2}\right)^{\rho/2}\left(\left|\log\frac{v_2}{v_1}\right|+1\right),
\end{equation*}
which is bounded. The inequality \eqref{frhostep1} can be written
\begin{equation}\label{frhostep2}
\int_1^\infty \frac 1{(v_1^2+u^2)^\rho(v_2^2+u^2)^\rho}\frac{du}{u^{1-2\rho}}\le \frac C{(v_1^2+v_2^2)^\rho}\left(\left|\log\frac{v_2}{v_1}\right|+1\right).
\end{equation}
By symmetry we can assume that $1\le v_1\le v_2$. Split the integral in \eqref{frhostep2} into three parts. 
\begin{align*}
&\int_1^{v_1}\frac 1{(v_1^2+u^2)^\rho(v_2^2+u_2^2)^\rho}\frac{du}{u^{1-2\rho}}\le \frac 1{(v_1^2v_2^2)^\rho}\int_1^{v_1}\frac{du}{u^{1-2\rho}}
\le\frac 1{2\rho(v_2^2)^\rho}\le\frac{2^{\rho-1}}{\rho(v_1^2+v_2^2)^\rho},\\
&\int_{v_1}^{v_2} \frac 1{(v_1^2+u^2)^\rho(v_2^2+u_2^2)^\rho}\frac{du}{u^{1-2\rho}}\le\frac 1{(v_1^2+v_2^2)^\rho}\int_{v_1}^{v_2}\frac{du}u=\frac{\log(v_2/v_1)}{(v_1^2+v_2^2)^\rho},\\
&\int_{v_2}^\infty\frac 1{(v_1^2+u^2)^\rho(v_2^2+u_2^2)^\rho}\frac{du}{u^{1-2\rho}}\le\int_{v_2}^\infty\frac{du}{u^{1+2\rho}}=
\frac 1{2\rho(v_2^2)^\rho}\le\frac{2^{\rho-1}}{\rho(v_1^2+v_2^2)^\rho}.\\
\end{align*}
Together these estimates show that \eqref{frhostep2} holds.

The left side of \eqref{frhob2} equals
\begin{equation*}
\frac 1{(v_1v_2)^{1/2-\rho}}\int_1^\infty\frac 1{(v_1^2+u^2)^\rho(v_2^2+u^2)^\rho}\frac{du}{u^{2-2\rho}}.
\end{equation*}
Again, we can assume that $v_1\le v_2$ and split the integral into three parts. First we consider
\begin{equation}\label{frhostep3}
\frac 1{(v_1v_2)^{1/2-\rho}}\int_1^{v_1}\frac 1{(v_1^2+u^2)^\rho(v_2^2+u^2)^\rho}\frac{du}{u^{2-2\rho}}\le
\frac 1{(v_1v_2)^{1/2+\rho}}\int_1^{v_1}\frac{du}{u^{2-2\rho}}.
\end{equation}
If $2\rho<1$, then the integral in the right side of \eqref{frhostep3} is bounded, if $2\rho=1$ it is $\le C(\log v_1)/v_1v_2\le C(v_1v_2)^{-(1+\rho)/2}$, and if
$2\rho>1$, we get the bound
\begin{equation*}
\frac C{(v_1v_2)^{1/2+\rho}}v_1^{2\rho-1}\le\frac C{v_1v_2}\le \frac C{(v_1v_2)^{(1+\rho)/2}},
\end{equation*}
since $\rho\le 1$, $v_1\le v_2$ and $v_1v_2\ge 1$. Next, we consider
\begin{align*}
&\frac 1{(v_1v_2)^{1/2-\rho}}\int_{v_1}^{v_2}\frac 1{(v_1^2+u^2)^\rho(v_2^2+u^2)^\rho}\frac{du}{u^{2-2\rho}}\le \frac 1{v_1^{1/2-\rho}v_2^{1/2+\rho}}
\int_{v_1}^{\infty}\frac{du}{u^2}\\
&= \frac 1{v_1^{1/2+1-\rho}v_2^{\rho/2}v_2^{(1+\rho)/2}}\le\frac 1{v_1^{1/2+1-\rho}v_1^{\rho/2}v_2^{(1+\rho)/2}}\le\frac 1{(v_1v_2)^{(1+\rho)/2}},
\end{align*}
since $1-\rho/2\ge\rho/2\le 1$, and $1\le v_1\le v_2$. Finally,
\begin{equation*}
\frac 1{(v_1v_2)^{1/2-\rho}}\int_{v_2}^{\infty}\frac 1{(v_1^2+u^2)^\rho(v_2^2+u^2)^\rho}\frac{du}{u^{2-2\rho}}\le \frac 1{(v_1v_2)^{1/2-\rho}}\int_{v_2}^{\infty}
\frac{du}{u^{2+2\rho}}\le\frac C{(v_1v_2)^{(1+\rho)/2}}.
\end{equation*}

Set
\begin{equation*}
g(u,v,\lambda)=\frac{\sqrt{uv}}{u^2+v^2+2\lambda uv},
\end{equation*}
so that
\begin{equation}\label{Kpguv}
K_p(u,v)=-\frac{\gamma_p\sin\pi\gamma_p}{\pi}g(u,v,-\cos\pi\gamma_p).
\end{equation}
If $u,v\ge 1$ and $\lambda>-1$, then
\begin{equation}\label{dgdu}
\left|\frac{\partial g}{\partial u}\right|\le \left(\frac 12+\max(1,\frac 1{1+\lambda})\right)\frac{g(u,v,\lambda)}u.
\end{equation}
To see this note that $u^2+v^2+2\lambda uv=(u-v)^2+2(1+\lambda)uv\ge 0$ if $\lambda>-1$, and
\begin{align*}
&\left|\frac{\partial }{\partial u}\log g(u,v,\lambda)\right|=\left|\frac 1{2u}-\frac{2u+2\lambda v}{u^2+v^2+2\lambda uv}\right|=
\frac{1}u\left|\frac{2u^2+2\lambda uv}{u^2+v^2+2\lambda uv}-\frac 12\right|=\frac{1}u\left|\frac{2u^2+2\lambda uv}
{u^2+v^2+2\lambda uv}-1+\frac 12\right|
\\
=&\frac{1}u\left|\frac{u^2-v^2}{u^2+v^2+2\lambda uv}+\frac 12\right|\le
\left(\frac 12+\frac{u^2+v^2}{u^2+v^2+2\lambda uv}\right).
\end{align*}
If $-1<\lambda<0$, then
\begin{equation*}
\frac{u^2+v^2}{u^2+v^2+2\lambda uv}=\frac{u^2+v^2}{(1+\lambda)(u^2+v^2)-\lambda(u-v)^2}\le\frac 1{1+\lambda},
\end{equation*}
and if $\lambda\ge 0$, the same expression is $\le 1$. This proves the estimate \eqref{dgdu}.
It follows from \eqref{Kpguv} and \eqref{dgdu} that \eqref{dKp} holds.
The lemma is proved.
\end{proof}

We next turn to the proof of the lemma that gives estimates of $F_{\rho,k}$.

\begin{proof}[Proof of Lemma \ref{Lem:Frhok}]
A change of variables in \eqref{Frhokstar} gives
\begin{equation}\label{Frhokhrho}
F_{\rho,k}(e^t,e^s)=e^{-(s+t)/2}h_\rho^{*k}(t-s).
\end{equation}
The Fourier inversion formula, the Plancherel theorem and the estimate $|\hat{h}_\rho(\xi)|\le ||h_\rho||_1$ give, for $k\ge 2$,
\begin{equation*}
|h_\rho^{*k}(x)|\le \frac{||h_\rho||_1^{k-2}}{2\pi}\int_\R |\hat{h}_\rho(\xi)|^2\,d\xi=\frac 1{2\pi}||h_\rho||_1^{k-2}\||h_\rho||_2^2.
\end{equation*}
From this estimate, and the definition \eqref{hrho}, we see that $|h_\rho^{*k}(x)|\le C^k$, for $k\ge 1$, $x\in\R$. This proves \eqref{Frhok1}.

It follows from \eqref{Frhok} and \eqref{dfrho} that
\begin{equation}
    \left|\frac{\partial F_{\rho,k}}{\partial u}(u,v)\right|\le
    \int_0^\infty\left|\frac{\partial f_{\rho}}{\partial u}(u,w)\right|
    F_{\rho,k-1}(w,v)\,dw\le\frac Cu F_{\rho,k}(u,v),
\end{equation}
which proves \eqref{Frhok3}. The identity \eqref{Frhokstar} gives
\begin{align*}
    \int_1^\infty F_{\rho,k}(u,v)\frac{du}{\sqrt{u}}=\frac 1{\sqrt{v}}
    \int_1^\infty h_\rho^{*k}(\log u-\log v)\frac{du}{u}\le \frac 1{\sqrt{v}}\int_{\R} h_\rho^{*k}(t-\log v)\,dt
    =\frac 1{\sqrt{v}}\hat{h}(0)^k\le \frac{C^k}{\sqrt{v}},
\end{align*}
since $h_\rho^{*k}$ is non-negative, and we have proved \eqref{Frhok4}.

To prove \eqref{Frhok5} note that \eqref{sumint} gives
\begin{align*}
    \sum_{j=1}^\infty\frac 1{\sqrt{j}}F_{\rho,k}(j,v)&\le F_{\rho,k}(1,v)+\int_1^\infty F_{\rho,k}(u,v)\frac{du}{\sqrt{u}}+\int_1^\infty \left|\frac{d}{du}\frac{1}{\sqrt{u}}F_{\rho,k}(u,v)\right|\,du\\
    &\le F_{\rho,k}(1,v)+C\int_1^\infty F_{\rho,k}(u,v)\frac{du}{\sqrt{u}}\le \frac{C^k}{\sqrt{v}},
\end{align*}
where we used the estimates \eqref{Frhok1}, \eqref{Frhok3} and \eqref{Frhok4}.

It follows from \eqref{KpHp} and a change of variables $v_j=e^{t_j}$, that
\begin{equation}\label{KpHpk}
 \int_{[0,\infty)^{k-1}}K_p(u,v_1)K_p(v_1,v_2)\dots K_p(v_{k-1},v)\,
dv_1\dots dv_{k-1}=\frac 1{\sqrt{uv}}H_p^{*k}(\log\frac uv).  
\end{equation}
Hence, by \eqref{Kpbound1} and \eqref{Frhok}, we see that \eqref{Frhok6} holds. Also, we see from \eqref{KpHpk} that
\begin{align*}
    \left|\frac{\partial}{\partial u}H_p^{*k}(\log\frac uv)\right|&\le \frac{\sqrt{v}}{2\sqrt{u}}\left|\int_{[0,\infty)^{k-1}}K_p(u,v_1)K_p(v_1,v_2)\dots K_p(v_{k-1},v)\,
dv_1\dots dv_{k-1}\right|\\
&+\sqrt{uv}\left|\int_{[0,\infty)^{k-1}}\frac{\partial}{\partial u}K_p(u,v_1)K_p(v_1,v_2)\dots K_p(v_{k-1},v)\,
dv_1\dots dv_{k-1}\right|.
\end{align*}
The estimate \eqref{Frhok7} now follows from \eqref{Kpbound1}, \eqref{dKp} and \eqref{Frhok}.

We see from \eqref{Frhokhrho} and the Fourier inversion formula that
\begin{equation}\label{Csmt1}
F_{\rho,k}(e^t,e^s)=e^{-(s+t)/2}h_\rho^{*k}(s-t)=\frac{e^{-(s+t)/2}}{2\pi}\int_\R e^{\I(s-t)\xi}\hat{h}_\rho(\xi)^k\,d\xi.
\end{equation}
Repeated integration by parts gives
\begin{equation}\label{Csmt2}
\int_\R e^{\I(s-t)\xi}\hat{h}_\rho(\xi)^k\,d\xi=-\frac {\I}{(s-t)^3}\int_\R e^{\I(s-t)\xi}\frac{d^3}{d\xi^3}\hat{h}_\rho(\xi)^k\,d\xi.
\end{equation}
From
\begin{equation*}
h_\rho(x)=\frac 1{2^\rho\cosh^\rho(x)},  
\end{equation*}
we see that $h_\rho$ belongs to the Schwartz space $\mathcal{S}(\R)$, and hence so does $\hat{h}_\rho$. Consequently, so does $\hat{h}_\rho^k$ and its
derivatives. It follows that for $k\le 3$, we have the estimate
\begin{equation*}
\left|\int_\R e^{\I(s-t)\xi}\frac{d^3}{d\xi^3}\hat{h}_\rho(\xi)^k\,d\xi\right|\le C.
\end{equation*}
If $k\ge 4$, then
\begin{equation*}
\frac{d^3}{d\xi^3}\hat{h}_\rho(\xi)^k= \hat{h}_\rho(\xi) \Big[k(k-1)(k-2)\hat{h}_\rho(\xi)^{k-4}\hat{h}_\rho'(\xi)+
2k(k-1)\hat{h}_\rho(\xi)^{k-3}\hat{h}_\rho''(\xi)+k\hat{h}_\rho(\xi)^{k-2}\hat{h}_\rho'''(\xi)\Big],
\end{equation*}
which gives the estimate
\begin{equation*}
    \left|\frac{d^3}{d\xi^3}\hat{h}_\rho(\xi)^k\right|\le C^k|\hat{h}_\rho(\xi)|.
\end{equation*}
Hence, by \eqref{Csmt2}
\begin{equation*}
  \left|\int_\R e^{\I(s-t)\xi}\hat{h}_\rho(\xi)^k\,d\xi\right|\le\frac{C^k}{|s-t|^3}\int_\R|\hat{h}(\xi)|\,d\xi\le \frac{C^k}{|s-t|^3}.  
\end{equation*}

Using this estimate in \eqref{Csmt1} and the estimate  \eqref{Frhok1} we obtain the bound \eqref{Frhok2}.
\end{proof}

For the proof of Lemma \ref{Lem:Kpest} we need a simple estimate proved by summation by parts. Let $z\in\C$ and write
\begin{equation*}
S_\ell=\sum_{k=1}^\ell z^k,
\end{equation*}
$\ell\ge 1$, and let $S_0=0$. Assume that $|S_\ell|\le C_0$ for all $\ell\ge 1$. Then, for any $C^1$-function $f:[1,\infty)\mapsto\C$, summation by parts 
\begin{equation*}
\sum_{\ell=1}^N z^\ell f(\ell)=S_Nf(N)-\sum_{\ell=1}^{N-1}S_\ell\int_\ell^{\ell+1} f'(u)\,du,
\end{equation*}
gives the estimate
\begin{equation}\label{Sumparts}
\left|\sum_{\ell=1}^N z^\ell f(\ell)\right|\le C_0\left(|f(N)|+\int_1^N|f'(u)|\,du\right),
\end{equation}
for any $N\ge 1$.

We are now ready for the 
\begin{proof}[Proof of Lemma \ref{Lem:Kpest}]
If $p\neq q$, then
\begin{equation*}
\left|\sum_{k=1}^\ell(z_p\bar{z}_q)^k\right|\le\frac 1{|\sin(\theta_p-\theta_q)/2|}\le\max_{p\neq q}\frac 1{|\sin(\theta_p-\theta_q)/2|}=C_0,
\end{equation*}
for all $\ell\ge 1$. Hence, by \eqref{Sumparts},
\begin{equation}\label{Kpeststep1}
\left|\sum_{\ell=1}^N(z_p\bar{z}_q)^\ell K_p(j,\ell)K_q(\ell,k)\right|\le C_0|K_p(j,N)K_q(N,k)|+C_0\int_1^N\left|\frac{\partial}{\partial u}K_p(j,u)K_q(u,k)\right|\,du.
\end{equation}
If we use \eqref{Kpguv}, \eqref{dgdu}, \eqref{Kpbound1} and \eqref{dKp}, we see that the right side of \eqref{Kpeststep1} is bounded by
\begin{align*}
&C(f_\rho(j,N)f_\rho(N,k)+\int_1^Nf_\rho(j,u)f_\rho(u,k)\frac{du}u)\\
&\le C\left(\frac 1{(jk)^{1/2}}\left(\frac{jN}{j^2+N^2}\right)^\rho\left(\frac{kN}{k^2+N^2}\right)^\rho+\frac 1{(jk)^{(1+\rho)/2}}\right),
\end{align*}
where we also used \eqref{frhob2}.
The bound \eqref{Kpest1} now follows if we can show that
\begin{equation*}
\frac 1{N^{1/2}}\left(\frac{jN}{j^2+N^2}\right)^\rho\le\frac C{j^{\rho/2}},
\end{equation*}
for $j\, N\ge 1$. Letting $j=tN$, we see that this is equivalent to
\begin{equation*}
\frac{N^{\rho/2}}{N^{1/2}}\left(\frac{t^{3/2}}{1+t^2}\right)^\rho\le C,
\end{equation*}
which holds for all $t>0$ and $N\ge 1$ since $\rho\le 1$.

It follows from \eqref{sumint} that
\begin{align*}
&\left|\sum_{k=1}^\infty K_p(v_1,k)K_p(k,v_2)-\int_1^\infty K_p(v_1,u)k_p(u,v_2)\,du\right|\\
&\le |K_p(v_1,1)K_p(1,v_2)|+\int_1^\infty\left|\frac{\partial}{\partial u} K_p(v_1,u)K_p(u,v_2)\right|\,du,
\end{align*}
and we see that \eqref{Kpest2} follows by the arguments above.
\end{proof}

It remains to prove the integral formula \eqref{Finalintegral}.

\begin{proof}[Proof of Lemma \ref{Lem:Finalintegral}]
Define the function
\begin{equation*}
h(\beta)=-\frac 1{2\pi^2}\int_\R \log\left(1-\frac{\sinh\beta x}{\sinh x}\right)\,dx,
\end{equation*}
for $|\beta|<1$, where $h(0)=0$. We want to show that
\begin{equation}\label{Finint1}
h(\beta)+h(-\beta)=\frac{\beta^2}{6(1-\beta^2)}.
\end{equation}
Differentiation of $h(\beta)$ gives
\begin{align}\label{Finint2}
h'(\beta)&=\frac 1{2\pi^2}\int_\R\frac{x\cosh\beta x}{\sinh x-\sinh\beta x}dx=\frac 1{2\pi^2}\int_0^\infty\frac{x\cosh(\frac{1+\beta}2 x-\frac{1-\beta}2 x)}
{\cosh\frac{1+\beta}2 x\sinh\frac{1-\beta}2 x}\notag\\
&=\frac 1{2\pi^2}\int_0^\infty x(\coth\frac{1-\beta}2 x-1)\,dx-\frac 1{2\pi^2}\int_0^\infty x(\tanh\frac{1+\beta}2 x-1)\,dx\notag\\
&=\frac{2}{\pi^2(1-\beta)^2}\int_0^\infty t(\coth t-1)\,dt-\frac{2}{\pi^2(1+\beta)^2}\int_0^\infty t(\tanh t-1)\,dt.
\end{align}
Now,
\begin{align*}
\int_0^\infty t(\coth t-1)\,dt=2\int_0^\infty\frac{te^{-2t}}{1-e^{-2t}}dt=2\sum_{n=1}^\infty\int_0^\infty te^{-2nt}\,dt
=\frac 12\sum_{n=1}^\infty\frac 1{n^2}=\frac{\pi^2}{12}.
\end{align*}
Similarly,
\begin{equation*}
\int_0^\infty t(\tanh t-1)\,dt=-\sum_{n=1}^\infty\frac{(-1)^n}{n^2}=-\frac{\pi^2}{24}.
\end{equation*}
Inserting these integrals into \eqref{Finint2} gives
\begin{equation*}
h'(\beta)=\frac 1{6(1-\beta)^2}+\frac 1{12(1+\beta)^2}.
\end{equation*}
Thus,
\begin{equation*}
h(\beta)=\int_0^\beta h'(t)dt=\frac{3\beta+\beta^2}{12(1-\beta^2)},
\end{equation*}
which gives \eqref{Finint1}.
\end{proof}

\end{document}